\documentclass[12pt,leqno]{amsart}
\usepackage{myamsart}
\usepackage[
a4paper, 
foot=20pt, 
top=80pt,
headsep=20pt,
left=84pt, 
right=84pt
]{geometry}
\usepackage{cite}
\usepackage{amsmath,amssymb,amsthm}
\usepackage{txfonts}
\usepackage{graphicx}
\usepackage{here}
\newcommand{\rnum}[1]{\expandafter{\romannumeral #1}}
\newcommand{\Rnum}[1]{\uppercase\expandafter{\romannumeral #1}}
\newcommand{\s}[1]{#1}
\newcommand{\vPhi}{\varPhi}
\renewcommand{\u}[1]{\underline{#1}}
\newcommand{\inner}[2]{\left\langle{#1},{#2}\right\rangle}
\newcommand{\tang}[1]{T_{\mbox{\scriptsize $#1$}}}
\newcommand{\fff}{\mathrm{I}}
\newcommand{\sff}{\mathrm{II}}
\newcommand{\tff}{\mathrm{III}}
\newcommand{\cnt}{\mathrm{CNT}}
\newcommand{\sym}{\mathrm{sym}}
\newcommand{\figsize}{4.5cm}
\DeclareMathOperator{\proj}{Proj}
\DeclareMathOperator{\GC}{GC}
\theoremstyle{definition}
\newtheorem{Def}{Definition}[section]
\newtheorem{Ex}[Def]{Example}
\newtheorem{fact}[Def]{Fact}

\newtheorem{Rem}[Def]{Remark}

\theoremstyle{plain}
\newtheorem{Thm}[Def]{Theorem}
\newtheorem{Lem}[Def]{Lemma}
\newtheorem{Prop}[Def]{Proposition}
\newtheorem{Cor}[Def]{Corollary}
\theoremstyle{plain}

\makeatletter
\@addtoreset{equation}{section}

\makeatother
\allowdisplaybreaks
\author{Motoko Kotani, Hisashi Naito and Toshiaki Omori}
\address{
  M.~Kotani:
  Mathematical Institute,
  Tohoku University,
  Aoba, Sendai 980-8578, Japan
  and 
  AIMR
  Tohoku University,
  Aoba, Sendai 980-8577, Japan
}
\email{m-kotani@m.tohoku.ac.jp}
\address{
  H.~Naito:
  Graduate School of Mathematics, 
  Nagoya University, 
  Chikusa, Nagoya 464-8602, Japan
}
\email{naito@math.nagoya-u.ac.jp}
\address{
  T.~Omori:
  Department of Mathematics, 
  Faculty of Science and Technology, 
  Tokyo University of Science, 
  Noda, Chiba 278-8510, Japan
}
\email{omori{\char"5F}toshiaki@ma.noda.tus.ac.jp}
\title{A Discrete Surface Theory}
\dedicatory{Dedicated to Yumiko Naito}
\begin{document}
\begin{abstract}
  In the present paper, we propose a new discrete surface theory on 3-valent embedded graphs in the 3-dimensional Euclidean space which are not necessarily ``discretization'' or ``approximation'' of smooth surfaces.
  The Gauss curvature and the mean curvature of discrete surfaces are defined which satisfy properties corresponding to the classical surface theory.
  We also discuss the convergence of a family of subdivided discrete surfaces 
  of a given 3-valent discrete surface by using the Goldberg-Coxeter construction.
  Although discrete surfaces in general have no corresponding smooth surfaces, 
  we may find one as the limit. 
\end{abstract}
\maketitle
\section{Introduction}
\label{section:Intro}
The present paper discusses ``discrete surface theory''. 
There are several proposals of discrete surface theory 
or ``discrete differential geometry'' by several authors 
from different viewpoints. 
For example, one of the big motivations of their studies is 
to visualize a given smooth surface 
and compute its geometric quantities, 
or to consider a series of simplex complex 
which approximates the surface 
and discuss convergence theories 
of the geometric quantities. 
This can be taken as a generalization of 
classical study of geometry of polyhedrons. 
Another direction is to study 
discrete integrable systems of the integer networks.
See for example 
\cite{MR2407724,MR1396732,MR2467378,MR2359767,MR1246481,MR2233848,MR2579698,MR2657431}, 
for references.
\par
Our motivation is to develop a surface theory of embedded graphs.
An embedded graph is a mathematical model of 
atomic configurations of a matter, 
where vertices represent atoms, 
edges interactions or bondings, respectively. 
A systematic study of chemical graphs is done by 
M.\ Deza and M.\ Dutour 
(see \cite{MR2429120,MR2035314}) by applying combinatorics. 
Our approach is a little different. 
We would like to define discrete surfaces out of embedded graphs, 
and their differential geometric notions 
such as their Gauss curvature, 
and mean curvature, 
which are believed in materials science 
to indicate inner frustration and
outer stress of the atomic configurations, respectively. 
For examples, A.\ L.\ Mackay and H.\ Terrones~\cite{Mackay-Terrones:1991} proposed a carbon network, 
which is supposed to be 
a discrete Schwarzian surface 
(triply periodic minimal surface, negatively curved in particular) 
and caught much attentions in materials science, 
but there is no precise definitions of curvatures, 
as far as the authors know.
\par
In the present paper, we define discrete surface as an ``embedded'' 3-valent graph equipped with the normal vector field $\u{n}$ over its vertices,
and the Gauss curvature $K$ and the mean curvature $H$ as the determinant and the trace and  of the Weingarten map $\nabla \u{n}$. 
We say a``graph" for an abstract graph and a 3-``discrete surface" for an embedded 3-valent graph.
We show their properties in Section \ref{Section3} corresponding to the classical surface theory, the variational formula of area (Theorem \ref{Thm(variation_formula)}), and the relation between harmonic maps and minimal surfaces (Theorem \ref{Thm(min_harm)}).
In Section \ref{Section(example)}, 
we compute the Gauss curvature and the mean curvature of some examples such as plane graphs,
sphere-shaped graphs, carbon nanotubes (hexagonal graphs on a cylinder), 
and the Mackay-like crystals (spatial graphenes).
We also discuss subdivision of discrete surfaces by using the subdivision theory for abstract 3-valent graphs, 
which is called the Goldberg-Coxeter construction, in Section \ref{Section(GC)}.
The Goldberg-Coxeter subdivision, which we discussed, keeps to be of 3-valent. 
Moreover we discuss their convergence to smooth surfaces in Section \ref{sec:subdivision}. 
Topological defects of the Mackay crystal can be detected by the
Goldberg-Coxeter subdivisions.
\par
We here emphasize that we cannot apply the classical study of polyhedrons or simplex complex because there is no natural way to assign faces which bound a given 1-skeltons (graphs) and thus no associate complex so that we can apply known notions.
We even treat graphs at each vertex of which graph the vector space spanned by the edges emerge from the vertex is of 2 dimensional (flat plane) but the whole graph lies as a surface in the 3-dimensional Euclidean space.
In that case, the classical definition of curvature defined by its angle defect is zero, but the surface looks like a negatively curved surface.
We need a new definition of curvatures to take care of examples including the Mackay crystal, arising from materials sciences.
\section{The classical surface theory in $\mathbb{R}^3$}
\label{Section2}
Prior to the introduction of a discrete surfaces theory 
in Section \ref{Section3}, 
in this section we briefly review basic facts of 
the classical surface theory in $\mathbb{R}^3$ 
for the readers. 
See \cite{MR2566897} for example for details. 
\par
Let $M\subseteq \mathbb{R}^3$ be a regular surface 
(of class $C^2$), which is (locally) parameterized by, say,  
$\s{p}=\s{p}(u,v)\colon\Omega \to \mathbb{R}^3$, 
where $\Omega\subseteq \mathbb{R}^2$. 
The tangent plane $\tang{\s{p}}M$ at $\s{p}=\s{p}(u,v)$ 
is the vector space spanned by the partial derivatives 
$\partial_u\s{p}$ and $\partial_v\s{p}$ of $\s{p}$ 
with respect to $u$ and $v$, respectively. 
It is equipped with the standard inner product 
$\inner{\cdot}{\cdot}$ in $\mathbb{R}^3$. 
\par
The \emph{first fundamental form} 
$\fff=\fff(u,v)$ of $M$ at $\s{p}(u,v)$ 
is a symmetric $2$-tensor defined as 
\[
\fff 
= d\s{p}\cdot d\s{p} 
= \inner{\partial_u\s{p}}{\partial_u\s{p}}du\cdot du 
+ 2\inner{\partial_u\s{p}}{\partial_v\s{p}}du\cdot dv 
+ \inner{\partial_v\s{p}}{\partial_v\s{p}}dv\cdot dv, 
\]
which is also expressed by the matrix-form: 
\[
\fff 
= 
\begin{pmatrix}
  E & F \\
  F & G 
\end{pmatrix} 
= 
\begin{pmatrix}
  \inner{\partial_u\s{p}}{\partial_u\s{p}} 
  & 
  \inner{\partial_u\s{p}}{\partial_v\s{p}} 
  \\
  \inner{\partial_v\s{p}}{\partial_u\s{p}} 
  & 
  \inner{\partial_v\s{p}}{\partial_v\s{p}}
\end{pmatrix}.
\]
The matrix $\fff(u,v)$ has rank $2$ (positive definite) 
since we assume that $M$ is regular. 
The \emph{unit normal vector field} 
\[
\s{n} 
= 
\s{n}(u,v) 
= 
\frac{
  \partial_u\s{p}\times\partial_v\s{p}
}{
  \lvert \partial_u\s{p}\times\partial_v\s{p} \rvert
}
\]
is well-defined at every point $(u,v)\in \Omega$. 
The \emph{second fundamental form} $\sff=\sff(u,v)$ 
is then defined as 
\[
\sff 
= 
-d\s{p}\cdot d\s{n} 
= 
\begin{pmatrix}
  L & M \\ 
  M & N
\end{pmatrix}
= 
\begin{pmatrix}
  -\inner{\partial_u\s{p}}{\partial_u\s{n}} 
  & 
  -\inner{\partial_u\s{p}}{\partial_v\s{n}} 
  \\
  -\inner{\partial_v\s{p}}{\partial_u\s{n}} 
  & 
  -\inner{\partial_v\s{p}}{\partial_v\s{n}}
\end{pmatrix}, 
\]
which is also a symmetric tensor. 

\begin{fact}
  Th partial derivatives $\partial_u\s{n}$ and 
  $\partial_v\s{n}$ of $\s{n}$, 
  which is perpendicular to $\s{n}$, 
  can be represented by $\{\partial_u\s{p},\partial_v\s{p}\}$; 
  \begin{equation}
    \begin{aligned}
      \partial_u\s{n} 
      & = 
      \frac{FM-GL}{EF-F^2}\partial_u\s{p} 
      + 
      \frac{FL-EM}{EG-F^2}\partial_v\s{p}, 
      \\
      \partial_v\s{n} 
      & = 
      \frac{FN-GM}{EF-F^2}\partial_u\s{p} 
      + 
      \frac{FM-EN}{EG-F^2}\partial_v\s{p}.
    \end{aligned}
    \label{Eq.dun&dvn}
  \end{equation}
\end{fact}
We define the Weingarten map 
$S= \nabla n \colon\tang{\s{p}}M \to \tang{\s{p}}M$.
By the symmetry of $\sff$, 
$S$ is a symmetric operator in the sense that 
it satisfies $\inner{S\s{V}}{\s{W}}=\inner{\s{V}}{S\s{W}}$ 
for any $\s{V},\s{W}\in \tang{\s{p}}M$. 
The trace of $S$ is called the \emph{mean curvature} $H(\s{p})$ 
and the determinant of $S$ the \emph{Gauss curvature} 
$K(\s{p})$, respectively. 
Since the representation matrix of $S$ with respect to 
$\{\partial_u\s{p},\partial_v\s{p}\}$ is $\fff^{-1}\sff$.
\begin{fact}
  The mean curvature $H(\s{p})$ and the Gauss curvature $K(\s{p})$ are defined by 
  \begin{equation}
    \begin{aligned}
      H(\s{p}) 
      & 
      = 
      \frac{1}{2}\mathrm{tr}(\fff^{-1}\sff) 
      = 
      \frac{EN+GL-2FM}{2(EG-F^2)}, 
      \\
      K(\s{p}) 
      &
      = 
      \det(\fff^{-1}\sff) 
      = 
      \frac{LN-M^2}{EG-F^2}. 
    \end{aligned}
    \label{Eq.H&K_smooth}
  \end{equation}
\end{fact}
It is easy to see
\begin{equation}
  S^2-2H(\s{p})S+K(\s{p})\mathrm{Id}=0. 
  \label{Eq.S^2-2HS+KId=0}
\end{equation}
We also define the \emph{third fundamental form} 
$\tff=\tff(u,v)$ as 
\[
\tff 
= 
d\s{n}\cdot d\s{n} 
= 
\begin{pmatrix}
  \inner{\partial_u\s{n}}{\partial_u\s{n}} 
  & 
  \inner{\partial_u\s{n}}{\partial_v\s{n}} 
  \\
  \inner{\partial_v\s{n}}{\partial_u\s{n}} 
  & 
  \inner{\partial_v\s{n}}{\partial_v\s{n}}	
\end{pmatrix}.
\]
Because of the symmetry of $S$, 
$\inner{\partial_u\s{n}}{\partial_u\s{n}}
= \inner{S\partial_u\s{p}}{S\partial_u\s{p}}
= \inner{S^2\partial_u\s{p}}{\partial_u\s{p}}$ 
and so on, 
from (\ref{Eq.S^2-2HS+KId=0}) we infer 
\begin{equation}
  K(\s{p})\fff - 2H(\s{p})\sff + \tff = 0. 
  \label{Eq.KI-2HII+III=0}
\end{equation}

We are ready to present several different meanings 
of the Gauss curvature.
To do so let us consider the Gauss map 
$\s{n}\colon M \to \mathbb{S}^2$ 
from $M$ to the unit sphere $\mathbb{S}^2$. 
Then the Gauss curvature appears in its area element. 
\begin{fact}
  The Gauss curvature is written as 
  the ratio of the infinitesimal area elements: 
  \begin{equation}
    \lvert K(\s{p}(u_0,v_0)) \rvert 
    = 
    \lim_{\varepsilon \to 0}
    \frac{
      A_{\Omega_{\varepsilon}}(\s{n})
    }{
      A_{\Omega_{\varepsilon}}(\s{p})
    }. 
    \label{Eq.area_ratio}
  \end{equation}
\end{fact}

\begin{proof}
  It is easy by using (\ref{Eq.dun&dvn}) to have  
  \begin{equation}
    \partial_u\s{n} 
    \times 
    \partial_v\s{n} 
    = 
    \frac{LN-M^2}{EG-F^2} 
    (
    \partial_u\s{p} 
    \times 
    \partial_v\s{p}
    ) 
    = 
    K(\s{p})
    (
    \partial_u\s{p} 
    \times 
    \partial_v\s{p}
    ). 
    \label{Eq.dun*dun=Kdup*dup}
  \end{equation}
  If we take an $\varepsilon$-neighborhood 
  $\Omega_{\varepsilon}\subseteq \Omega$ 
  of $(u_0,v_0)\in \Omega$ 
  for any $\varepsilon>0$, then since 
  \begin{align*}
    A_{\Omega_{\varepsilon}}(\s{p}) 
    & = 
      \int_{\Omega_{\varepsilon}}
      \lvert 
      \partial_u\s{p} 
      \times 
      \partial_v\s{p} 
      \rvert\,
      dudv, 
    \\
    A_{\Omega_{\varepsilon}}(\s{n}) 
    & = 
      \int_{\Omega_{\varepsilon}}
      \lvert 
      \partial_u\s{n} 
      \times 
      \partial_v\s{n} 
      \rvert\,
      dudv 
      = 
      \int_{\Omega_{\varepsilon}}
      \lvert K\rvert 
      \lvert 
      \partial_u\s{p} 
      \times 
      \partial_v\s{p} 
      \rvert\,
      dudv
  \end{align*}
  are the area of the image 
  $\s{p}(\Omega_{\varepsilon})\subseteq M$ 
  and 
  $\s{n}(\Omega_{\varepsilon})\subseteq \mathbb{S}^2$, 
  respectively.
\end{proof}
A variational approach is also available 
for the formulation of the curvatures as follows. 
Let $\s{p}\colon\overline{\Omega}\to \mathbb{R}^3$ 
be a regular surface of class $C^2$. 
The functional $\mathcal{A}(\s{p})$ defined as 
\[
\mathcal{A}(\s{p}) 
:= 
\int_{\Omega}
\lvert 
\partial_u\s{p} 
\times 
\partial_v\s{p} 
\rvert\,
dudv 
= 
\int_{\Omega}
\,dA
\]
is called the \emph{area functional}, 
whose first and second variation formulas are those we want. 
Let 
$\s{q}_t=\s{q}(u,v,t)\colon 
\overline{\Omega}\times(-\varepsilon,\varepsilon)$ 
be a variation of $\s{p}$ 
with the variation vector field, say, 
\[
\s{V}(u,v) 
= 
\varphi^1(u,v)\partial_u\s{p}(u,v) 
+ 
\varphi^2(u,v)\partial_u\s{p}(u,v) 
+ 
\psi(u,v)\s{n}(u,v), 
\]
where $\varphi^i,\psi\in C^1(\overline{\Omega})$ ($i=1,2$). 

\begin{fact}
  The first variation of $\mathcal{A}$ 
  at $\s{p}$ 
  is then given as 
  \begin{equation}
    d\mathcal{A}(\s{p},\s{V}) 
    = 
    \left.
      \frac{d}{dt}
    \right\rvert_{t=0} 
    \mathcal{A}(\s{q}_t) 
    = 
    -2\int_{\Omega} 
    \psi\cdot H(\s{p})
    \lvert 
    \partial_u\s{p} 
    \times 
    \partial_v\s{p} 
    \rvert
    \,dudv, 
    \label{Eq.first_var_of_A}
  \end{equation}
  independently of variations in the tangential direction. 
  \par
  While the second variation of $\mathcal{A}$
  at a general regular surface $\s{p}$ 
  with respect to the normal variation $\s{V}=\psi\s{n}$ 
  {\upshape(}that is, $\varphi^1=\varphi^2=0$ {\upshape)} is given as 
  \begin{equation}
    d^2\mathcal{A}(\s{p},\psi\s{n}) 
    = 
    \int_{\Omega}
    \left(
      \lvert \nabla_M\psi \rvert^2 
      + 
      2\psi^2K(\s{p})
    \right)\,
    dA, 
    \label{Eq.second_var_of_A}
  \end{equation}
  where the norm $\lvert \nabla_M\psi\rvert^2$ is 
  taken with respect to $\fff$, 
  sometimes called 
  the first Beltrami differentiator. 
\end{fact}
A surface $M\subseteq \mathbb{R}^3$ 
satisfying $H(\s{p})=0$ for any point $\s{p}\in M$ 
is said to be \emph{minimal}. 
\par
At the end of this section, 
we state  a characterization of minimal surfaces 
as follows: 
\begin{fact}
  Let $\s{p}=\s{p}(u,v)\colon\Omega\to \mathbb{R}^3$ 
  be a regular surface of class $C^2$ and 
  $\s{n}\colon\Omega\to \mathbb{R}^3$ 
  be its Gauss map. Then 
  \begin{equation}
    \partial_v\s{n}\times \partial_u\s{p} 
    - 
    \partial_u\s{n}\times \partial_v\s{p} 
    = 
    2H(\s{p}) 
    \lvert 
    \partial_u\s{p}\times \partial_v\s{p}
    \rvert 
    \s{n}, 
    \label{Eq.charact_of_minimal}
  \end{equation}
  or equivalently, 
  \[
  d(\s{n}\times d\s{p}) 
  = 
  -2H(\s{p})\s{n}\,dA, 
  \]
  where 
  $
  \s{n}\times d\s{p} 
  = (\s{n}\times \partial_u\s{p})du 
  + (\s{n}\times \partial_v\s{p})dv
  $ 
  is a differential $1$-form on $\Omega$ along $\s{p}$. 
  That is to say, 
  $\s{p}\colon\Omega \to \mathbb{R}^3$ is 
  a minimal surface if and only if 
  $\s{n}\times d\s{p}$ is closed. 
\end{fact}

\section{A surface theory for graphs in $\mathbb{R}^3$}
\label{Section3}
\subsection{Definition of curvatures}
\label{subsec:definition_of_curvatures}
Let $X=(V,E)$ be a general graph, 
where $V$ denotes the set of the vertices 
and $E$ the set of the oriented edges. 
The oriented edge $e$ is identified with 
a $1$-dimensional cell complex. 
Thus we can assume that every edge $e\in E$ is 
identified with the interval $[0,1]$. 
The reverse edge is denoted by $\bar{e}$, and
$E_x$ is the set of edges which emerge 
from a vertex $x \in V$.
\par
A map $\vPhi\colon X \to \mathbb{R}^3$ 
is said to be a \emph{piecewise linear realization} 
if the restriction $(\vPhi|e)(t)$ 
on each edge $e\in E$ is linear in $t\in [0,1]$ 
and $(\vPhi|{\bar{e}})(t)=(\vPhi|e)(1-t)$. 
\begin{Def}\label{Def(disc_surf)}
  An injective piecewise linear realization 
  $\vPhi\colon X \to \mathbb{R}^3$ 
  of a graph $X=(V,E)$ is said to be 
  a \emph{discrete surface} if 
  \begin{enumerate}
  \item[{\upshape (\rnum{1})}] 
    $X=(V,E)$ is a $3$-valent graph, 
    that is a graph of degree $3$, 
  \item[{\upshape (\rnum{2})}] 
    for each $x\in V$, 
    at least two vectors in 
    $\{\vPhi(e)\mid e\in E_x\}$ 
    are linearly independent as vectors in $\mathbb{R}^3$, 
  \item[{\upshape (\rnum{3})}] 
    locally oriented, that is, 
    the order of the three edges 
    is assumed to be assigned 
    to each vertex of $X$. 
  \end{enumerate}
\end{Def}
Let 
$\vPhi\colon X=(V,E) \to M \subseteq \mathbb{R}^3$ 
be a discrete surface. 
For each vertex $x\in V$, 
we assume it is of $3$-valent, 
namely the set $E_x=\{e_1,e_2,e_3\}$ 
of edges with origin $x$ consists of 
three oriented edges. 
In the sequel, we sometimes use 
the notation $\vPhi(e)=\u{e}\in M$ 
to denote the edge in $M$ which corresponds to $e\in E$. 
The \emph{tangent plane} $T_x$ at $\vPhi(x)$ 
is then the plane with $\u{n}(x)$ 
as its \emph{oriented unit normal vector} $\u{n}(x)$ 
at $\vPhi(x)$ is defined as 
\begin{equation}
  \begin{aligned}
    \u{n}(x) 
    :={} 
    & 
    \frac{
      (\u{e_2}-\u{e}_1)
      \times (\u{e}_3-\u{e}_1)
    }{
      \lvert 
      (\u{e}_2-\u{e}_1)
      \times 
      (\u{e}_3-\u{e}_1) 
      \rvert
    } 
    \\
    ={} 
    & 
    \frac{
      \u{e}_1\times\u{e}_2 + 
      \u{e}_2\times\u{e}_3 + 
      \u{e}_3\times\u{e}_1
    }{
      \lvert 
      \u{e}_1\times\u{e}_2 + 
      \u{e}_2\times\u{e}_3 + 
      \u{e}_3\times\u{e}_1 
      \rvert
    }.
    \label{Eq.n}
  \end{aligned}
\end{equation}
Note that we use the condition of graphs 
to be $3$-valent to define its tangent plane.
\par
For each $x\in V$ and $e\in E_x$, the vector 
\begin{equation}
  \nabla_e\vPhi 
  := 
  \proj[\vPhi(e)]
  = 
  \u{e} - 
  \langle \u{e},\u{n}(x)\rangle\u{n}(x)
  \label{Eq.tangent_vec}
\end{equation}
lies on $T_x$, where $\proj$ is denoted by 
the orthogonal projection onto $T_x$ and 
$\langle\cdot,\cdot\rangle$ stands for 
the standard inner product of $\mathbb{R}^3$. 
Similarly, the directional derivative of $\u{n}$ 
along $e\in E$ is defined as
\begin{equation}
  \nabla_e\u{n} 
  := 
  \proj[\u{n}(t(e))-\u{n}(o(e))], 
  \label{Eq.deriv_of_n}
\end{equation}
so that $\nabla_e\u{n}\in T_x$. 
\par
Before we define the curvature of a surface, 
we work with a triangle 
$\triangle=\triangle(\u{x}_0,\u{x}_1,\u{x}_2)$ 
of the graph in $\mathbb{R}^3$.
Oriented unit vectors are assigned 
, say, 
$\u{n}_0$, $\u{n}_1$ and $\u{n}_2$, respectively, 
at $\u{x}_0$, $\u{x}_1$ and $\u{x}_2$. 
Later they are taken as unit normal vectors, 
but we note they need not be 
perpendicular to the triangle $\triangle$.
\par
We set $\u{v}_1:=\u{x}_1-\u{x}_0$ 
and $\u{v}_2:=\u{x}_2-\u{x}_0$ for simplicity, 
which corresponds to (\ref{Eq.tangent_vec}). 
The \emph{first fundamental form} 
$\fff_{\triangle}$ of $\triangle$ is now defined as 
\begin{equation}
  \fff_{\triangle} 
  := 
  \begin{pmatrix}
    E & F \\
    F & G
  \end{pmatrix} 
  = 
  \begin{pmatrix}
    \inner{\u{v}_1}{\u{v}_1} 
    & 
    \inner{\u{v}_1}{\u{v}_2} 
    \\
    \inner{\u{v}_2}{\u{v}_1} 
    & 
    \inner{\u{v}_2}{\u{v}_2} 
  \end{pmatrix}. 
  \label{Eq.EFG}
\end{equation}
As the directional derivative of $\u{n}_0$ 
along $\u{v}_1$ and $\u{v}_2$ 
corresponding to (\ref{Eq.deriv_of_n}), we set
\[
\nabla_i\u{n} 
:= 
\proj[\u{n}_i - \u{n}_0]
\]
for $i=1,2$, where $\proj$ is the orthogonal projection 
onto $T_{\triangle}$, a plane on which $\triangle$ lies. 
As is straightforward to check, 
$\nabla_1\u{n}$ and $\nabla_2\u{n}$ 
are in fact written, respectively, as 
\begin{equation}
  \begin{aligned}
    \nabla_1\u{n} 
    & = \frac{FM_1-GL}{EG-F^2}\u{v}_1 
    + 
    \frac{FL-EM_1}{EG-F^2}\u{v}_2, 
    \\
    \nabla_2\u{n} 
    & = \frac{FN-GM_2}{EG-F^2}\u{v}_1 
    + 
    \frac{FM_2-EN}{EG-F^2}\u{v}_2,
    \label{Eq.nabla_in}
  \end{aligned}
\end{equation}
where $E$, $F$ and $G$ are given by (\ref{Eq.EFG}) and 
$L$, $M_1$, $M_2$ and $L$ are defined as 
\begin{equation}
  \sff_{\triangle} 
  := 
  \begin{pmatrix}
    L & M_2 \\
    M_1 & N
  \end{pmatrix} 
  = 
  \begin{pmatrix}
    -\inner{\u{v}_1}{\nabla_1\u{n}} 
    & 
    -\inner{\u{v}_1}{\nabla_2\u{n}} 
    \\
    -\inner{\u{v}_2}{\nabla_1\u{n}} 
    & 
    -\inner{\u{v}_2}{\nabla_2\u{n}}
  \end{pmatrix}
  \label{Eq.LMN}
\end{equation}
in the \emph{second fundamental form} of $\triangle$. 
Note here that $M_1\neq M_2$ is possible in our case although 
the classical theory depends on the symmetry of $\sff$.
\par
Now we can define the Weingarten-type map 
$S_{\triangle}\colon T_{\triangle}\rightarrow T_{\triangle}$ 
as $S_{\triangle}= -\nabla \u{n}$ 
and the \emph{mean curvature} $H_{\triangle}$ 
and the \emph{Gauss curvature} $K_{\triangle}$ 
of $\triangle$ as its trace 
and determinant as in the classical case.
\par
The following result corresponds to (\ref{Eq.H&K_smooth}). 
\begin{Prop}
  The mean curvature $H_{\triangle}$ and the Gauss curvature $K_{\triangle}$ 
  have, respectively, the following representations: 
  \begin{equation}
    \begin{aligned}
      H_{\triangle} 
      & =  
      \frac{1}{2}
      \mathrm{tr}(
      \fff_{\triangle}^{-1}
      \sff_{\triangle}^{}
      ) 
      = 
      \frac{
	EN + GL - F (M_1 + M_2)
      }{
	2(EG - F^2)
      }, 
      \\
      K_{\triangle} 
      & = 
      \det(
      \fff_{\triangle}^{-1}
      \sff_{\triangle}^{}
      ) 
      = 
      \frac{
	LN - M_1M_2
      }{
	EG - F^2
      }. 
      \label{Eq.H&K_discrete}
    \end{aligned}
  \end{equation}
\end{Prop}
Since both the first fundamental form (\ref{Eq.EFG}) 
and the third one (\ref{Eq.c_ij}) are 
symmetric, while it is not always the case with
the second fundamental form (\ref{Eq.LMN}), 
the same identity as (\ref{Eq.KI-2HII+III=0}) 
cannot be expected. 
But the next proposition shows that 
the symmetry of $\sff_{\triangle}$ 
is the only obstruction for (\ref{Eq.KI-2HII+III=0}) 
to be valid. 
\begin{Prop}\label{Prop(KI-2HII+III=0)} 
  Let $\tff_{\triangle}$ be the \emph{third fundamental form} 
  of $\triangle$ defined as 
  \begin{equation}
    \tff_{\triangle} 
    := 
    \begin{pmatrix}
      c_{11} & c_{12} \\
      c_{21} & c_{22}
    \end{pmatrix} 
    = 
    \begin{pmatrix}
      \inner{\nabla_1\u{n}}{\nabla_1\u{n}} 
      & 
      \inner{\nabla_1\u{n}}{\nabla_2\u{n}} 
      \\
      \inner{\nabla_2\u{n}}{\nabla_1\u{n}} 
      & 
      \inner{\nabla_2\u{n}}{\nabla_2\u{n}}
    \end{pmatrix}. 
    \label{Eq.c_ij}
  \end{equation}
  Then 
  \begin{equation}
    K_{\triangle}\fff_{\triangle} 
    - 
    2H_{\triangle}\sff_{\triangle}
    + 
    \tff_{\triangle}
    = 
    \frac{M_1-M_2}{EG-F^2} 
    \begin{pmatrix}
      EM_1 - FL 
      & 
      EN - FM_2 
      \\
      FM_1 - GL
      & 
      FN - GM_2
    \end{pmatrix}. 
    \label{Eq.KI-2HII+III!=0}
  \end{equation}
  In particular,  the second fundamental form 
  $\sff_{\triangle}$ is symmetric if and only if 
  \[
  K_{\triangle}\fff_{\triangle} 
  - 
  2H_{\triangle}\sff_{\triangle} 
  + 
  \tff_{\triangle} 
  = 
  0. 
  \]
\end{Prop}
\begin{proof}
  A straightforward computation 
  using (\ref{Eq.nabla_in}) gives 
  \begin{align*}
    c_{11} 
    & = 
      \frac{EM_1^2 - 2FLM_1 + GL^2}{EG-F^2}, 
    \\
    c_{12} 
    & = 
      c_{21} 
      = 
      \frac{EM_1N - FLN - FM_1M_2 + GLM_2}{EG-F^2}, 
    \\
    c_{22} 
    & = 
      \frac{EN^2 - 2FM_2N + GM_2^2}{EG-F^2}. 
  \end{align*}
  This equalities combined with 
  (\ref{Eq.EFG}), (\ref{Eq.LMN}) and (\ref{Eq.H&K_discrete}) 
  yield the required equality. 
\end{proof}
On the other hand, 
exactly same equality as (\ref{Eq.dun*dun=Kdup*dup}) or 
(\ref{Eq.area_ratio}) 
is obtained. 
\begin{Prop}\label{Prop(Gauss)}
  The Gauss curvature $K_{\triangle}$ satisfies
  \begin{equation}
    \nabla_1\u{n}\times\nabla_2\u{n} 
    = 
    \frac{LN - M^2}{EG - F^2} 
    (\u{v}_1 \times \u{v}_2) 
    = 
    K_{\triangle}
    (\u{v}_1 \times \u{v}_2), 
    \label{Eq.dun*dun=Kv1v2}
  \end{equation}
  Thus, in particular, 
  the absolute value of the Gauss curvature 
  $K_{\triangle}$ is given by
  \[
  \lvert K_{\triangle} \rvert 
  = 
  \frac{
    \lvert 
    \nabla_1\u{n} \times \nabla_2\u{n} 
    \rvert
  }{
    \lvert 
    \u{v}_1 \times \u{v}_2
    \rvert
  }. 
  \]
\end{Prop}
\begin{proof}
  The proof again follows from a direct computation 
  using (\ref{Eq.nabla_in}) as follows: 
  \begin{align*}
    \nabla_1\u{n} \times \nabla_2\u{n}
    & = 
      \left(
      \frac{FM_1-GL}{EG-F^2}\u{v}_1
      + 
      \frac{FL-EM_1}{EG-F^2}\u{v}_2
      \right) 
    \\
    & \qquad \qquad 
      \times 
      \left(
      \frac{FN-GM_2}{EG-F^2}\u{v}_1 
      + 
      \frac{FM_2-EN}{EG-F^2}\u{v}_2 
      \right) 
    \\
    & = 
      \frac{
      \u{v}_1 \times \u{v}_2
      }{
      (EG-F^2)^2
      }
      \left\{
      (FM_1-GL)(FM_2-FN) 
      \right. 
    \\
    & \qquad \qquad 
      \left. 
      - 
      (FL-EM_1)(FN-GM_2)
      \right\} 
    \\
    & = 
      \frac{
      (LN-M_1M_2)(EG-F^2)
      }{
      (EG-F^2)^2
      }
      (\u{v}_1\times \u{v}_2) 
    \\
    & = 
      \frac{
      LN-M_1M_2
      }{
      EG-F^2
      }
      (\u{v}_1\times \u{v}_2) 
    \\
    & = 
      K_{\triangle}
      (\u{v}_1\times \u{v}_2), 
  \end{align*}
  as required. 
\end{proof}
\begin{figure}[htbp]
  \centering
  \includegraphics[width=7cm]
  {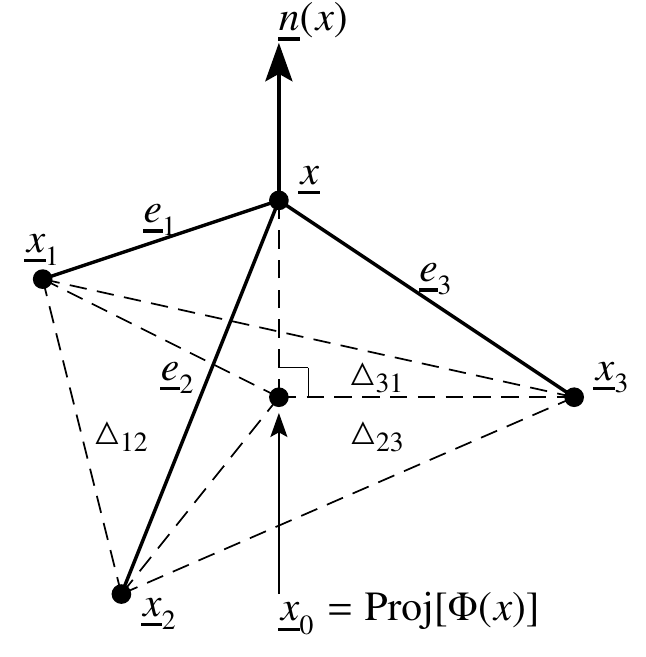}
  \caption{}
  \label{Figure(triangles)}
\end{figure}
Now we are ready to give the definitions of 
the mean curvature $H$ and the Gauss curvature $K$ 
of a discrete surface $\vPhi\colon X=(V,E)\to \mathbb{R}^3$. 
Idea is to define them as 
the area-weighted average of 
those of the three triangles around the vertex.
\par
Let $x\in V$ be a vertex, 
$E_x=\{e_1,e_2,e_3\}$ 
and $(\alpha,\beta)=(1,2)$, $(2,3)$ or $(3,1)$. 
If we choose the triangle 
$\triangle_{\alpha\beta}
= \triangle
(\u{x}_0,\u{x}_{\alpha},\u{x}_{\beta})$ 
as 
\[
\u{x}_0 
= 
\proj[\vPhi(x)], \quad 
\u{x}_{\alpha} 
= 
\vPhi(t(e_{\alpha}))~~ 
\text{and}~~
\u{x}_{\beta} 
= 
\vPhi(t(e_{\beta})), 
\]
(see Figure~\ref{Figure(triangles)}) then 
the first, second and third fundamental form of 
$\triangle_{\alpha\beta}$, are 
given as 
\begin{align*}
  \fff_{\alpha\beta} 
  & = 
    \begin{pmatrix}
      \inner{
	\nabla_{e_{\alpha}}\vPhi
      }{
	\nabla_{e_{\beta}}\vPhi
      }
      & 
      \inner{
	\nabla_{e_{\alpha}}\vPhi
      }{
	\nabla_{e_{\beta}}\vPhi
      } 
      \\
      \inner{
	\nabla_{e_{\beta}}\vPhi
      }{
	\nabla_{e_{\alpha}}\vPhi
      }
      & 
      \inner{
	\nabla_{e_{\beta}}\vPhi
      }{
	\nabla_{e_{\beta}}\vPhi
      }
    \end{pmatrix}, 
  \\
  \sff_{\alpha\beta} 
  & = 
    \begin{pmatrix}
      -\inner{
	\nabla_{e_{\alpha}}\vPhi
      }{
	\nabla_{e_{\alpha}}\u{n}
      }
      &  
      -\inner{
	\nabla_{e_{\alpha}}\vPhi
      }{
	\nabla_{e_{\beta}}\u{n}
      } 
      \\
      -\inner{
	\nabla_{e_{\beta}}\vPhi
      }{
	\nabla_{e_{\alpha}}\u{n}
      }
      & 
      -\inner{
	\nabla_{e_{\beta}}\vPhi
      }{
	\nabla_{e_{\beta}}\u{n}
      }
    \end{pmatrix}, 
  \\
  \tff_{\alpha\beta} 
  & = 
    \begin{pmatrix}
      \inner{
	\nabla_{e_{\alpha}}\u{n}
      }{
	\nabla_{e_{\alpha}}\u{n}
      }
      & 
      \inner{
	\nabla_{e_{\alpha}}\u{n}
      }{
	\nabla_{e_{\beta}}\u{n}
      } 
      \\
      \inner{
	\nabla_{e_{\beta}}\u{n}
      }{
	\nabla_{e_{\alpha}}\u{n}
      } 
      & 
      \inner{
	\nabla_{e_{\beta}}\u{n}
      }{
	\nabla_{e_j}\u{n}
      }
    \end{pmatrix}, 	
\end{align*}
respectively. 
Under this settings, we have already 
defined the mean curvature $H_{\triangle_{\alpha\beta}}$ 
and the Gauss curvature $K_{\triangle_{\alpha\beta}}$ of 
$\triangle_{\alpha\beta}$. 
Then 
\begin{Def}\label{Def(H&K)}
  For a discrete surface 
  $\vPhi\colon X=(V,E)\to \mathbb{R}^3$, 
  the \emph{mean curvature} $H(x)$ and 
  the \emph{Gauss curvature} $K(x)$ at $x\in V$ 
  are defined, respectively, as 
  \begin{align}
    H(x) 
    & := 
      \sum_{\alpha,\beta}
      \frac{\sqrt{\det \fff_{\alpha\beta}(x)}}{A(x)}
      H_{\triangle_{\alpha\beta}}(x), \label{Eq.disc_mean} 
    \\
    K(x) 
    & := 
      \sum_{\alpha,\beta}
      \frac{\sqrt{\det \fff_{\alpha\beta}(x)}}{A(x)}
      K_{\triangle_{\alpha\beta}}(x), \label{Eq.disc_gauss}
  \end{align}
  where the summations are taken over any 
  $(\alpha,\beta)\in \{(1,2),(2,3),(3,1)\}$ 
  such that the Weingarten-type map 
  $S_{\alpha\beta}\colon T_x \to T_x$ 
  is defined, also, 
  $A(x)$ is the denominator of {\upshape (\ref{Eq.n}):} 
  \[
  A(x) 
  := 
  \lvert 
  \u{e}_1\times\u{e}_2 + 
  \u{e}_2\times\u{e}_3 + 
  \u{e}_3\times\u{e}_1 
  \rvert, 
  \]
  twice the area of the triangle with 
  $\{\vPhi(t(e_1)), \vPhi(t(e_2)), \vPhi(t(e_3))\}$ 
  as its vertices. 
\end{Def}
\begin{Def}\label{Def(minimal)}
  A discrete surface is said to be \emph{minimal} 
  if its mean curvature vanishes at every vertex. 
\end{Def}
Here we give two observations, 
which are useful in practical computations 
of $H$ or $K$. 
The first asserts that 
we can forget about  the projection 
$\nabla_i\u{n}=\proj[\u{n}_i-\u{n}_0]$ as seen below. 
\begin{Lem}\label{Lem(sff)}
  The second fundamental form {\upshape (\ref{Eq.LMN})} of 
  $\triangle=\triangle(\u{x}_0,\u{x}_1,\u{x}_2)$ 
  satisfies 
  \[
  \sff_{\triangle} 
  = 
  \begin{pmatrix}
    -\inner{\u{v}_1}{\u{n}_1-\u{n}_0} 
    & 
    -\inner{\u{v}_1}{\u{n}_2-\u{n}_0} 
    \\
    -\inner{\u{v}_2}{\u{n}_1-\u{n}_0} 
    & 
    -\inner{\u{v}_2}{\u{n}_2-\u{n}_0}
  \end{pmatrix}. 
  \]
\end{Lem}

\begin{proof}
  The assertion is obvious 
  because $\u{v}_i=\u{x}_i-\u{x}_0$ ($i=1,2$) 
  lies on $T_{\triangle}$, 
  whereas $\nabla_i\u{n}=\proj[\u{n}_i-\u{n}_0]$ 
  is the orthogonal projection onto $T_{\triangle}$. 
\end{proof}

The second asserts that 
the curvatures at $x$ are equal 
to the corresponding curvatures 
of the triangle 
$\triangle(\u{x}_1, \u{x}_2, \u{x}_3)$ 
with the adjacent vertices $\u{x}_1$, $\u{x}_2$ 
and $\u{x}_3$ of $\u{x}$. 

\begin{Prop}\label{Prop(large_triangle)}
  Let $\vPhi\colon X=(V,E) \to \mathbb{R}^3$ be 
  a  discrete surface. 
  The mean curvature $H(x)$ and the Gauss curvature $K(x)$ 
  at $x\in V$ are represented, respectively, as 
  \begin{align}
    H(x) 
    & = 
      \frac{1}{2}
      \mathrm{tr}(
      \fff_{\triangle(x)}^{-1}
      \sff_{\triangle(x)}^{}
      ), \label{Eq.H_large_triangle} 
    \\
    K(x) 
    & = 
      \det(
      \fff_{\triangle(x)}^{-1}\sff_{\triangle(x)}^{}
      ), \label{Eq.K_large_triangle}
  \end{align}
  where $\fff_{\triangle(x)}$ and 
  $\sff_{\triangle(x)}$ are respectively 
  the first and second fundamental forms of 
  the triangle $\triangle(x)$ with vertices 
  $\{\vPhi(x_1),
  \vPhi(x_2),\vPhi(x_3)\}$, 
  which is actually given as 
  \begin{align*}
    \fff_{\triangle(x)}
    & = 
      \begin{pmatrix}
	\inner{
          \u{e}_2-\u{e}_1
	}{
          \u{e}_2-\u{e}_1
	} 
	& 
	\inner{
          \u{e}_2-\u{e}_1
	}{
          \u{e}_3-\u{e}_1
	} 
	\\
	\inner{
          \u{e}_3-\u{e}_1
	}{
          \u{e}_2-\u{e}_1
	} 
	& 
	\inner{
          \u{e}_3-\u{e}_1
	}{
          \u{e}_3-\u{e}_1
	}
      \end{pmatrix}, 
    \\
    \sff_{\triangle(x)} 
    & = 
      \begin{pmatrix}
	-\inner{
          \u{e}_2-\u{e}_1
	}{
          \u{n}(x_2)-\u{n}(x_1)
	} 
	& 
	-\inner{
          \u{e}_2-\u{e}_1
	}{
          \u{n}(x_3)-\u{n}(x_1)
	} 
	\\
	-\inner{
          \u{e}_3-\u{e}_1
	}{
          \u{n}(x_2)-\u{n}(x_1)
	} 
	& 
	-\inner{
          \u{e}_3-\u{e}_1
	}{
          \u{n}(x_3)-\u{n}(x_1)
	}
      \end{pmatrix}, 
  \end{align*}
  where $E_x=\{e_1,e_2,e_3\}$ and $x_i=t(e_i)$ for $i=1,2,3$. 
\end{Prop}

\begin{proof}
  Let $x\in V$ be fixed let $E_x=\{e_1,e_2,e_3\}$ and 
  $x_i:=t(e_i)$ for $i=1,2,3$. 
  Notice first that 
  $\vPhi(e_i)$ itself needs not lie on 
  the tangent plane $T_x$ at $\vPhi(x)$, 
  while does $\vPhi(e_j)-\vPhi(e_1)$ ($j=1,2$). 
  In the sequel, set 
  $\u{n}_0:=\u{n}(x)$, 
  $\u{n}_i:=\u{n}(x_i)$ ($i=1,2,3$) for simplicity 
  and let $(\alpha,\beta)=(1,2)$, $(2,3)$ or $(3,1)$. 
  Then we can write as 
  \begin{gather*}
    \begin{pmatrix}
      \nabla_{e_{\alpha}}\vPhi 
      & 
      \nabla_{e_{\beta}}\vPhi
    \end{pmatrix} 
    = 
    \begin{pmatrix}
      \u{e}_2 - \u{e}_1 
      & 
      \u{e}_3 - \u{e}_1
    \end{pmatrix}
    P_{\alpha\beta}, 
    \\
    \text{with}~
    P_{\alpha\beta} 
    = 
    \frac{1}{A(x)}
    \begin{pmatrix}
      +\inner{
	\u{n}_0
      }{
	\u{e}_{\alpha}\times(\u{e}_3-\u{e}_1)
      } 
      & 
      +\inner{
	\u{n}_0
      }{
	\u{e}_{\beta}\times(\u{e}_3-\u{e}_1)
      } 
      \\
      -\inner{
	\u{n}_0
      }{
	\u{e}_{\alpha}\times(\u{e}_2-\u{e}_1)
      } 
      & 
      -\inner{
	\u{n}_0
      }{
	\u{e}_{\beta}\times(\u{e}_2-\u{e}_1)
      } 
    \end{pmatrix}. 
  \end{gather*}
  Under this transformation of frames, 
  the first fundamental form $\fff_{\alpha\beta}$ 
  and second fundamental form $\sff_{\alpha\beta}$ 
  of the triangle 
  $\triangle_{\alpha\beta}
  = \triangle(
  \vPhi(x),
  \vPhi(x_{\alpha}),
  \vPhi(x_{\beta}))$ are transformed as
  \begin{align*}
    \fff_{\alpha\beta} 
    & = 
      \begin{pmatrix}
	\inner{
          \nabla_{e_{\alpha}}\vPhi
	}{
          \nabla_{e_{\alpha}}\vPhi
	} 
	& 
	\inner{
          \nabla_{e_{\alpha}}\vPhi
	}{
          \nabla_{e_{\beta}}\vPhi
	} 
	\\
	\inner{
          \nabla_{e_{\beta}}\vPhi
	}{
          \nabla_{e_{\alpha}}\vPhi
	} 
	& 
	\inner{
          \nabla_{e_{\beta}}\vPhi
	}{
          \nabla_{e_{\beta}}\vPhi
	}
      \end{pmatrix} 
    \\
    & = 
      {}^tP_{\alpha\beta}
      \begin{pmatrix}
	\lvert \u{e}_2-\u{e}_1 \rvert^2 
	& 
	\inner{\u{e}_2-\u{e}_1}{\u{e}_3-\u{e}_1} 
	\\
	\inner{\u{e}_2-\u{e}_1}{\u{e}_3-\u{e}_1} 
	& 
	\lvert \u{e}_3-\u{e}_1 \rvert^2
      \end{pmatrix}
          P_{\alpha\beta} 
    \\
    & = 
      {}^tP_{\alpha\beta}
      \fff_{\triangle(x)}
      P_{\alpha\beta}, 
    \\
    \sff_{\alpha\beta} 
    & = 
      \begin{pmatrix}
	-\inner{
          \nabla_{e_{\alpha}}\vPhi
	}{
          \nabla_{e_{\alpha}}\u{n}
	} 
	& 
	-\inner{
          \nabla_{e_{\alpha}}\vPhi
	}{
          \nabla_{e_{\beta}}\u{n}
	} 
	\\
	-\inner{
          \nabla_{e_{\beta}}\vPhi
	}{
          \nabla_{e_{\alpha}}\u{n}
	} 
	& 
	-\inner{
          \nabla_{e_{\beta}}\vPhi
	}{
          \nabla_{e_{\beta}}\u{n}
	}
      \end{pmatrix} 
    \\
    & = 
      {}^tP_{\alpha\beta}
      \begin{pmatrix}
	-\inner{\u{e}_2-\u{e}_1}{\u{n}_{\alpha}} 
	& 
	-\inner{\u{e}_2-\u{e}_1}{\u{n}_{\beta}} 
	\\
	-\inner{\u{e}_3-\u{e}_1}{\u{n}_{\alpha}} 
	& 
	-\inner{\u{e}_3-\u{e}_1}{\u{n}_{\beta}} 
      \end{pmatrix}, 
  \end{align*}
  respectively. Therefore we obtain
  \[
  \fff_{\alpha\beta}^{-1}\sff_{\alpha\beta}^{} 
  = 
  P_{\alpha\beta}^{-1}\fff_{\triangle(x)}
  \begin{pmatrix}
    -\inner{\u{e}_2-\u{e}_1}{\u{n}_{\alpha}} 
    & 
    -\inner{\u{e}_2-\u{e}_1}{\u{n}_{\beta}} 
    \\
    -\inner{\u{e}_3-\u{e}_1}{\u{n}_{\alpha}} 
    & 
    -\inner{\u{e}_3-\u{e}_1}{\u{n}_{\beta}} 
  \end{pmatrix}.
  \]
  Since 
  $\sqrt{\det\fff_{\alpha\beta}} 
  = A(x)\det P_{\alpha\beta}$, 
  it follows from the definition (\ref{Eq.disc_mean}) of 
  $H(x)$ that 
  \begin{align*}
    2H(x) 
    & = 
      \sum_{\alpha,\beta}
      \frac{\sqrt{\det\fff_{\alpha\beta}}}{A(x)}
      \mathrm{tr}
      (\fff_{\alpha\beta}^{-1}\sff_{\alpha\beta}^{}) 
    \\
    & = 
      \sum_{\alpha,\beta}
      \mathrm{tr}
      \left[
      \fff_{\triangle(x)}^{-1}
      (\det P_{\alpha\beta}^{}\cdot P_{\alpha\beta}^{-1})
      \begin{pmatrix}
	-\inner{\u{e}_2-\u{e}_1}{\u{n}_{\alpha}} 
	& 
	-\inner{\u{e}_2-\u{e}_1}{\u{n}_{\beta}} 
	\\
	-\inner{\u{e}_3-\u{e}_1}{\u{n}_{\alpha}} 
	& 
	-\inner{\u{e}_3-\u{e}_1}{\u{n}_{\beta}} 
      \end{pmatrix}
          \right] 
    \\
    & = 
      \mathrm{tr}
      \left[
      \fff_{\triangle(x)}^{-1}
      \sum_{\alpha,\beta}
      (\det P_{\alpha\beta}^{}\cdot P_{\alpha\beta}^{-1})
      \begin{pmatrix}
	-\inner{\u{e}_2-\u{e}_1}{\u{n}_{\alpha}} 
	& 
	-\inner{\u{e}_2-\u{e}_1}{\u{n}_{\beta}} 
	\\
	-\inner{\u{e}_3-\u{e}_1}{\u{n}_{\alpha}} 
	& 
	-\inner{\u{e}_3-\u{e}_1}{\u{n}_{\beta}} 
      \end{pmatrix}
          \right], 
  \end{align*}
  where the summation are taken over 
  $(\alpha,\beta)=(1,2)$, $(2,3)$ and $(3,1)$. 
  We complete the proof of (\ref{Eq.H_large_triangle}) 
  by proving 
  \[
  \mathrm{tr}(\sff_{\triangle(x)}) 
  = 
  \mathrm{tr}	
  \sum_{\alpha,\beta}
  (\det P_{\alpha\beta}^{}\cdot P_{\alpha\beta}^{-1})
  \begin{pmatrix}
    -\inner{\u{e}_2-\u{e}_1}{\u{n}_{\alpha}} 
    & 
    -\inner{\u{e}_2-\u{e}_1}{\u{n}_{\beta}} 
    \\
    -\inner{\u{e}_3-\u{e}_1}{\u{n}_{\alpha}} 
    & 
    -\inner{\u{e}_3-\u{e}_1}{\u{n}_{\beta}} 
  \end{pmatrix}, 
  \]
  which follows from a simple direct computation. 
  \par
  Our next task is to prove (\ref{Eq.K_large_triangle}). 
  It again follows from the definition (\ref{Eq.disc_gauss}) 
  of $K(x)$ that 
  \begin{align*}
    K(x) 
    & = 
      \sum_{\alpha,\beta}
      \frac{\sqrt{\det\fff_{\alpha\beta}}}{A(x)}
      \det
      (\fff_{\alpha\beta}^{-1}\sff_{\alpha\beta}^{})
      = 
      \sum_{\alpha,\beta}
      \frac{
      \det(\sff_{\alpha\beta}^{})
      }{
      A(x)^2\det P_{\alpha\beta}
      } 
    \\
    & = 
      \frac{1}{A(x)^2}
      \sum_{\alpha,\beta}
      \det
      \begin{pmatrix}
	-\inner{\u{e}_2-\u{e}_1}{\u{n}_{\alpha}} 
	& 
	-\inner{\u{e}_2-\u{e}_1}{\u{n}_{\beta}} 
	\\
	-\inner{\u{e}_3-\u{e}_1}{\u{n}_{\alpha}} 
	& 
	-\inner{\u{e}_3-\u{e}_1}{\u{n}_{\beta}} 
      \end{pmatrix} 
    \\
    & = 
      \frac{1}{A(x)^2}
      \begin{pmatrix}
	-\inner{\u{e}_2-\u{e}_1}{\u{n}_2-\u{n}_1} 
	& 
	-\inner{\u{e}_2-\u{e}_1}{\u{n}_3-\u{n}_1} 
	\\
	-\inner{\u{e}_3-\u{e}_1}{\u{n}_2-\u{n}_1} 
	& 
	-\inner{\u{e}_3-\u{e}_1}{\u{n}_3-\u{n}_1} 
      \end{pmatrix} 
    \\
    & = 
      \frac{1}{A(x)^2}\det(\sff_{\triangle(x)}).
  \end{align*}
  Thus we prove (\ref{Eq.K_large_triangle}) 
  because $A(x)^2=\det(\fff_{\triangle(x)})$. 
\end{proof}
We end this subsection 
by the following proposition, 
which corresponds to (\ref{Eq.charact_of_minimal}). 
\begin{Prop}
  Let $\vPhi\colon X=(V,E) \to \mathbb{R}^3$ 
  be a discrete surface, 
  let $x\in V$ be fixed and let $E_x=\{e_1,e_2,e_3\}$. 
  Then 
  the mean curvature $H(x)$ at $\vPhi(x)$ is 
  written as
  \begin{equation}
    \begin{aligned}
      2H(x)A(x)\u{n}(x) 
      & = 
      \sum_{\alpha,\beta}
      \left(
	\nabla_{e_{\beta}}\u{n}
	\times
	\nabla_{e_{\alpha}}\vPhi 
	- 
	\nabla_{e_{\alpha}}\u{n}
	\times
	\nabla_{e_{\beta}}\vPhi
      \right) 
      \\
      & = 
      \nabla_{e_2-e_1}\u{n}\times \nabla_{e_3-e_1}\vPhi 
      - 
      \nabla_{e_3-e_1}\u{n}\times \nabla_{e_2-e_1}\vPhi 
      \\
      & = 
      \nabla_{e_2-e_3}\u{n} \times \vPhi(e_1) 
      + 
      \nabla_{e_3-e_1}\u{n} \times \vPhi(e_2) 
      + 
      \nabla_{e_1-e_2}\u{n} \times \vPhi(e_3), 
    \end{aligned}
    \label{Eq.disc_min_eq}
  \end{equation}
  where the summation is taken over all 
  $(\alpha,\beta)=(1,2)$, $(2,3)$ and $(3,1)$, and 
  $\nabla_{e_i-e_j}\vPhi
  := \nabla_{e_i}\vPhi-\nabla_{e_j}\vPhi
  = \vPhi(e_i)-\vPhi(e_j)$ 
  as well as 
  $\nabla_{e_i-e_j}\u{n}
  := \nabla_{e_i}\u{n}-\nabla_{e_j}\u{n}$ 
  denote the directional derivatives 
  along $\vPhi(e_i)-\vPhi(e_j)$. 
\end{Prop}
\begin{proof}
  Set, for simplicity, $\u{e}_i=\vPhi(e_i)$ 
  and $\u{n}_i=\u{n}(t(e_i))$ ($i=1,2,3$) as usual. 
  As a consequence of Proposition~\ref{Prop(large_triangle)}, 
  we have 
  \begin{align*}
    H(x) 
    & = 
      \frac{1}{2A(x)^2}
      \Bigl\{
      - 
      \lvert \u{e}_2-\u{e}_1 \rvert^2
      \inner{\u{e}_3-\u{e}_1}{\u{n}_3-\u{n}_1} 
      - 
      \lvert \u{e}_3-\u{e}_1 \rvert^2
      \inner{\u{e}_2-\u{e}_1}{\u{n}_2-\u{n}_1} 
    \\
    & \qquad \qquad \qquad 
      + 
      \inner{\u{e}_2-\u{e}_1}{\u{e}_3-\u{e}_1}
      \bigl(
      \inner{\u{e}_2-\u{e}_1}{\u{n}_3-\u{n}_1}
      + 
      \inner{\u{e}_3-\u{e}_1}{\u{n}_2-\u{n}_1}
      \bigr)
      \Bigr\} 
    \\
    & = 
      \frac{-1}{2A(x)^2}
      \Bigl\{
      \lvert \u{e}_1-\u{e}_2\rvert^2
      \inner{\u{e}_3}{\u{n}_3}
      + 
      \lvert \u{e}_2-\u{e}_3\rvert^2
      \inner{\u{e}_1}{\u{n}_1} 
      +
      \lvert \u{e}_3-\u{e}_1\rvert^2
      \inner{\u{e}_2}{\u{n}_2} 
    \\
    & \qquad \qquad \qquad 
      + 
      \inner{\u{e}_1-\u{e}_2}{\u{e}_2-\u{e}_3}
      \bigl(
      \inner{\u{e}_3}{\u{n}_1} 
      + 
      \inner{\u{e}_1}{\u{n}_3}
      \bigr) 
    \\
    & \qquad \qquad \qquad 
      + 
      \inner{\u{e}_2-\u{e}_3}{\u{e}_3-\u{e}_1}
      \bigl(
      \inner{\u{e}_1}{\u{n}_2} 
      + 
      \inner{\u{e}_2}{\u{n}_1}
      \bigr) 
    \\
    & \qquad \qquad \qquad 
      + 
      \inner{\u{e}_3-\u{e}_1}{\u{e}_1-\u{e}_2}
      \bigl(
      \inner{\u{e}_2}{\u{n}_3} 
      + 
      \inner{\u{e}_3}{\u{n}_2}
      \bigr)\Bigr\}. 
  \end{align*}
  The terms involving $\u{n}_1$ are then summarized, 
  using 
  $\inner{\u{a}}{\u{c}}\u{b} 
  - \inner{\u{b}}{\u{c}}\u{a}
  = (\u{a}\times\u{b})\times\u{c}$ 
  for $\u{a},\u{b},\u{c}\in \mathbb{R}^3$, as 
  \[
  \inner{
    \u{e}_2-\u{e}_3
  }{
    (
    \u{e}_1\times\u{e}_2
    + \u{e}_2\times\u{e}_3 
    + \u{e}_3\times\u{e}_1
    )
    \times 
    \u{n}_1
  }
  = 
  A(x)\inner{\u{e}_2-\u{e}_3}{\u{n}(x)\times\u{n}_1}, 
  \]
  and similarly for $\u{n}_2$ and $\u{n}_3$. 
  Therefore, 
  \begin{align*}
    H(x) 
    & = 
      \frac{-1}{2A(x)}
      \Bigl\{
      \inner{\u{e}_2-\u{e}_3}{\u{n}(x)\times\u{n}_1} 
      + 
      \inner{\u{e}_3-\u{e}_1}{\u{n}(x)\times\u{n}_2} 
      + 
      \inner{\u{e}_1-\u{e}_2}{\u{n}(x)\times\u{n}_3}
      \Bigr\} 
    \\
    & = 
      \frac{-1}{2A(x)}
      \bigl\langle
      \u{n}(x), 
      \u{n}_1\times (\u{e}_2-\u{e}_3) 
      + 
      \u{n}_2\times (\u{e}_3-\u{e}_1) 
      + 
      \u{n}_3\times (\u{e}_1-\u{e}_2)
      \bigr\rangle 
    \\
    & = 
      \frac{1}{2A(x)}
      \bigl\langle
      \u{n}(x), 
      (\u{n}_2-\u{n}_3)\times \u{e}_1 
      + 
      (\u{n}_3-\u{n}_1)\times \u{e}_2 
      + 
      (\u{n}_1-\u{n}_2)\times \u{e}_3 
      \bigr\rangle. 
  \end{align*}
  Since 
  $\nabla_{e_{\alpha}-e_{\beta}}\u{n}
  = \nabla_{e_{\alpha}}\u{n}
  - \nabla_{e_{\beta}}\u{n}
  = \proj[\u{n}_{\alpha}-\u{n}_{\beta}]\in T_x$ is 
  the orthogonal projection onto the tangent plane 
  $T_x$ whose normal vector is $\u{n}(x)$, 
  so that 
  $\nabla_{e_{\alpha}-e_{\beta}}\u{n}
  \times \u{e}_{\gamma}$ 
  is parallel to $\u{n}(x)$ 
  for $(\alpha,\beta,\gamma)=(1,2,3)$, 
  $(2,3,1)$ or $(3,1,2)$, 
  we infer 
  \[
  2H(x)A(x)\u{n}(x) 
  = 
  \nabla_{e_2-e_3}\u{n}\times\u{e}_1 
  + 
  \nabla_{e_3-e_1}\u{n}\times\u{e}_2 
  + 
  \nabla_{e_1-e_2}\u{n}\times\u{e}_3. 
  \]
  The remaining two expressions in 
  (\ref{Eq.disc_min_eq}) are easily proved. 
\end{proof}
\begin{Cor}
  \label{Thm(charact_of_disc_min)}
  Let $X=(V,E)$ be a fixed graph, 
  $\vPhi_0:X=(V,E) \to \mathbb{R}^3$ 
  be a $3$-valent discrete surface with 
  $\u{n}_0\colon V \to \mathbb{R}^3$ 
  its oriented unit normal vector field, 
  and $H\colon V \to \mathbb{R}$ be a function. 
  Assume that 
  $\{\nabla_{e_2-e_1}\u{n}_0,
  \nabla_{e_3-e_1}\u{n}_0\}$ 
  is a pair of linearly independent vectors in $\mathbb{R}^3$, 
  for every $x\in V$, 
  where $E_x=\{e_1,e_2,e_3\}$. 
  If a $3$-valent discrete surface 
  $\vPhi\colon X=(V,E) \to \mathbb{R}^3$ 
  solves
  \begin{equation}
    2H(x)\u{m}(x) 
    = 
    \nabla_{e_2-e_3}\u{n}_0 
    \times 
    \vPhi(e_1) 
    + 
    \nabla_{e_3-e_1}\u{n}_0 
    \times 
    \vPhi(e_2) 
    +
    \nabla_{e_1-e_2}\u{n}_0 
    \times 
    \vPhi(e_3), 
    \label{Eq.prescribed_H}
  \end{equation}
  the prescribed mean curvature equation, 
  where $m\colon V\rightarrow \mathbb{R}^3$ is 
  the unnormalized normal vector field, given as 
  \[
  \u{m}(x)
  = \vPhi(e_1)\times \vPhi(e_2) 
  + \vPhi(e_2)\times \vPhi(e_3) 
  + \vPhi(e_3)\times \vPhi(e_1), 
  \]
  then, 
  after switching the orientation of edges $E_x=\{e_1,e_2,e_3\}$ 
  at each $x\in V$ if necessary, 
  the mean curvature of $\vPhi$ 
  coincides with $H(x)$ at each $x\in V$. 
\end{Cor}
\begin{proof}
  Let $\vPhi\colon X=(V,E)\rightarrow \mathbb{R}^3$ 
  solve (\ref{Eq.prescribed_H}). 
  Taking the inner product of (\ref{Eq.prescribed_H}) 
  with $\vPhi(e_2)-\vPhi(e_1)$, 
  which is perpendicular to $\u{m}(x)$, gives 
  \begin{align*}
    0 
    & = 
      \inner{
      \vPhi(e_2)-\vPhi(e_1)
      }{
      \nabla_{e_2-e_3}\u{n}_0 
      \times 
      \vPhi(e_1)
      } 
    \\
    & \qquad 
      + 
      \inner{
      \vPhi(e_2)-\vPhi(e_1)
      }{
      \nabla_{e_3-e_1}\u{n}_0 
      \times 
      \vPhi(e_2)
      } 
    \\
    & \qquad 
      + 
      \inner{
      \vPhi(e_2)-\vPhi(e_1)
      }{
      \nabla_{e_1-e_2}\u{n}_0 
      \times 
      \vPhi(e_3)
      } 
    \\
    & = 
      \inner{
      \nabla_{e_2}\u{n}_0 
      - 
      \nabla_{e_3}\u{n}_0
      }{
      \vPhi(e_1) 
      \times 
      \vPhi(e_2)
      } 
    \\
    & \qquad 
      - \inner{
      \nabla_{e_3}\u{n}_0 
      - 
      \nabla_{e_1}\u{n}_0
      }{
      \vPhi(e_2) 
      \times 
      \vPhi(e_1)
      } 
    \\
    & \qquad 
      + \inner{
      \nabla_{e_1}\u{n}_0 
      - 
      \nabla_{e_2}\u{n}_0
      }{
      \vPhi(e_3) 
      \times 
      (\vPhi(e_2)
      - 
      \vPhi(e_1))
      } 
    \\
    & = 
      - \inner{
      \nabla_{e_1}\u{n}_0
      }{
      \vPhi(e_1)\times \vPhi(e_2) 
      + \vPhi(e_2)\times \vPhi(e_3) 
      + \vPhi(e_3)\times \vPhi(e_1)
      } 
    \\
    & \qquad 
      + \inner{
      \nabla_{e_2}\u{n}_0
      }{
      \vPhi(e_1)\times \vPhi(e_2) 
      + \vPhi(e_2)\times \vPhi(e_3) 
      + \vPhi(e_3)\times \vPhi(e_1)
      } 
    \\
    & = 
      \inner{
      \nabla_{e_2}\u{n}_0 
      - 
      \nabla_{e_1}\u{n}_0
      }{
      \u{m}(x)
      } 
    \\
    & = 
      \inner{
      \nabla_{e_2-e_1}\u{n}_0
      }{
      \u{m}(x)
      }. 
  \end{align*}
  In a similar way, 
  taking the inner product with 
  $\vPhi(e_3)-\vPhi(e_1)$ gives 
  $\inner{\nabla_{e_3-e_1}\u{n}_0}{\u{m}(x)}=0$. 
  Since our assumption guarantees that 
  $\{\nabla_{e_2-e_1}\u{n}_0,\nabla_{e_3-e_1}\u{n}_0\}$ 
  spans the tangent plane to $\vPhi_0$ at $\vPhi_0(x)$, 
  we conclude that $\u{m}(x)$ is perpendicular to the plane, 
  or equivalently, parallel to $\u{n}_0(x)$. 
  The assumption that $\vPhi$ is a $3$-valent discrete surface 
  guarantees that $\u{m}(x)\neq \u{0}$ for every $x\in V$. 
\end{proof}
Notice that if we choose $H=0$ 
as a function $H\colon V \to \mathbb{R}$ 
in Corollary~\ref{Thm(charact_of_disc_min)}, 
then the equation (\ref{Eq.prescribed_H}) 
is linear with respect to $\vPhi$, 
so that, is always solvable, 
although a solution is not possibly 
a $3$-valent discrete surface. 
Several examples obtained by solving 
such equations will be actually provided 
in Section \ref{Section(mackay_min)}. 
%
\subsection{Variational approach} 
\label{subsec:variational}
Recall that both the mean curvature and the Gauss curvature 
of a smooth surface is formulated also 
by a variational approach of its area functional 
(\ref{Eq.first_var_of_A}) and (\ref{Eq.second_var_of_A}). 
This subsection is devoted to derive a variation formula 
for the \emph{area functional} defined as 
\begin{equation}
  \mathcal{A}[\vPhi] 
  := 
  \sum_{x\in V}
  \lvert 
  \vPhi(e_{x,1})\times\vPhi(e_{x,2}) + 
  \vPhi(e_{x,2})\times\vPhi(e_{x,3}) + 
  \vPhi(e_{x,3})\times\vPhi(e_{x,1}) 
  \rvert
  \label{Eq.disc_area}
\end{equation}
of a $3$-valent discrete surface 
$\vPhi\colon X=(V,E)\rightarrow \mathbb{R}^3$, 
where $E_x=\{e_{x,1},e_{x,2},e_{x,3}\}$ 
is the set of edges with origin $x$ 
such that the order of $e_i$'s is chosen to 
match the orientation. 
\par
As in the preceding subsection, 
we focus on a triangle 
$\triangle=\triangle(\u{x}_0,\u{x}_1,\u{x}_2)$ 
with vertices 
$\{\u{x}_1,\u{x}_2,\u{x}_3\}\subseteq \mathbb{R}^3$, 
to each $\u{x}_i$ of which the oriented unit normal vector 
$\u{n}_i$ is assumed to be assigned, 
to derive the variation formulas of its area: 
\begin{equation}
  A(\u{x}_1,\u{x}_2,\u{x}_3) 
  := 
  \lvert (\u{x}_2-\u{x}_1) 
  \times 
  (\u{x}_3-\u{x}_1) \rvert. 
  \label{Eq.1area}
\end{equation}
For any triplet $\u{u}=\{\u{u}_i\mid i=1,2,3\}$ 
of vectors in $\mathbb{R}^3$, 
we consider the variation of 
$\triangle(\u{x}_1,\u{x}_2,\u{x}_3)$ 
with $\u{u}$ as the variation vector field, 
that is, 
a $1$-parameter family of triangles 
$\triangle(
\u{x}_1(t),
\u{x}_2(t),
\u{x}_3(t)
)$ with 
\begin{equation}
  \u{x}_i(t) 
  = 
  \u{x}_i + t \u{u}_i
  \quad 
  (i=1,2,3)
  \label{Eq.n_variation}
\end{equation}
of its vertices for $t\in \mathbb{R}$. 
\begin{Lem}\label{Prop(var_of_triangle)}
  The first variation of area {\upshape (\ref{Eq.1area})} 
  of the triangle 
  $\triangle=\triangle(\u{x}_1,\u{x}_2,\u{x}_3)$ 
  with respect to $\u{u}=\{\u{u}_i\mid 
  i=1,2,3\}$ is given as 
  \begin{equation}
    \left.
      \frac{d}{dt}
    \right\rvert_{t=0}
    A(\u{x}_1(t),\u{x}_2(t),\u{x}_3(t)) 
    = 
    \inner{\u{V}^1}{\u{u}_1} 
    + \inner{\u{V}^2}{\u{u}_2} 
    + \inner{\u{V}^3}{\u{u}_3}, 
    \label{Eq.general_first_var_of_triangle}
  \end{equation}
  where, for $(i,j,k)=(1,2,3)$, $(2,3,1)$ or 
  $(3,1,2)$, $\u{V}^i$ is obtained by 
  rotating $\u{x}_j-\u{x}_k$ by $90^{\circ}$ 
  on the plane on which 
  $\triangle$ lie 
  {\upshape(}in the counterclockwise direction 
  as viewed facing the normal vector of $\triangle${\upshape)}. 
  If in particular $\u{u}_i$ is the unit normal 
  vector $\u{n}_i$ which is assigned to $\u{x}_i$, 
  it follows that 
  \begin{equation}
    \frac{1}{A(\u{x}_1,\u{x}_2,\u{x}_3)}
    \left.
      \frac{d}{dt}
    \right\rvert_{t=0}
    A(\u{x}_1(t),\u{x}_2(t),\u{x}_3(t)) 
    = 
    -2H_{\triangle}, 
    \label{Eq.first_var_of_triangle}
  \end{equation}
  where $H_{\triangle}$ is the mean curvature 
  of $\triangle$. 
  Moreover, if we consider a normal variation 
  of $\triangle$ 
  {\upshape(}i.e.\ $\u{u}_i=\u{n}_i$, 
  $i=1,2,3${\upshape)}, then we have 
  \begin{equation}
    \frac{1}{A(\u{x}_1,\u{x}_2,\u{x}_3)}
    \left.
      \frac{d^2}{dt^2}
    \right\rvert_{t=0} 
    A(
    \u{x}_1(t),
    \u{x}_2(t),
    \u{x}_3(t)
    )
    = 
    2K_{\triangle} 
    + 
    \mathrm{tr}(
    \fff_{\triangle}^{-1}
    (\tff_{\triangle}'-\tff_{\triangle}^{})
    ), 
    \label{Eq.second_var_of_triangle}
  \end{equation}
  where $K_{\triangle}$ is the Gauss curvature 
  of $\triangle$, 
  $\fff_{\triangle}$ and $\tff_{\triangle}$ 
  are, respectively, the first fundamental form {\upshape (\ref{Eq.EFG})} 
  and the third fundamental form {\upshape (\ref{Eq.c_ij})} 
  of $\triangle$ 
  with respect to the frame 
  $\{\u{u}_2-\u{u}_1,\u{u}_3-\u{u}_1\}$, and 
  \begin{equation}
    \tff_{\triangle}' 
    := 
    \begin{pmatrix}
      c_{22}' & c_{23}' \\ 
      c_{32}' & c_{33}'
    \end{pmatrix} 
    = 
    \begin{pmatrix}
      \inner{\u{n}_2-\u{n}_1}{\u{n}_2-\u{n}_1} 
      & 
      \inner{\u{n}_2-\u{n}_1}{\u{n}_3-\u{n}_1} 
      \\
      \inner{\u{n}_3-\u{n}_1}{\u{n}_2-\u{n}_1} 
      & 
      \inner{\u{n}_2-\u{n}_0}{\u{n}_3-\u{n}_1} 
    \end{pmatrix}.
  \end{equation}
  In particular, 
  if $\nabla_i\u{n}=\u{n}_i-\u{n}_1$ {\upshape(}$i=2,3${\upshape)}, 
  in other words, 
  $\triangle(
  \u{x}_1(t),
  \u{x}_2(t),
  \u{x}_3(t)
  )$ 
  is parallel to 
  $\triangle(\u{x}_1,\u{x}_2,\u{x}_3)$ 
  for some/any $t>0$, 
  then we have 
  \begin{equation}
    \frac{1}{A(\u{x}_1,\u{x}_2,\u{x}_3)}
    \left.
      \frac{d^2}{dt^2}
    \right\rvert_{t=0} 
    A(
    \u{x}_1(t),
    \u{x}_2(t),
    \u{x}_3(t)
    )
    = 
    2K_{\triangle}. 
  \end{equation}
\end{Lem}

\begin{proof}
  We discuss with the frame 
  $\{\u{v}_2=\u{u}_2-\u{u}_1, 
  \u{v}_3=\u{u}_3-\u{u}_1\}$ 
  and set 
  \[
  g_{ij} 
  := 
  \inner{\u{v}_i}{\u{v}_j}, \quad 
  b_{ij} 
  := 
  - \inner{\u{v}_i}{\nabla_i\u{u}} 
  = 
  - \inner{\u{v}_i}{\u{u}_i-\u{u}_1}
  \quad 
  (i,j=2,3).
  \]
  By (\ref{Eq.n_variation}), 
  \begin{align*}
    g_{ij}(t) 
    :={} & 
           \inner{
           \u{x}_i(t)-\u{x}_0(t)
           }{
           \u{x}_j(t)-\u{x}_0(t)
           } 
    \\
    ={} & 
          g_{ij} 
          - 
          t b_{ij} 
          - 
          t b_{ji} 
          + 
          t^2\inner{\u{u}_i-\u{u}_1}{\u{u}_j-\u{u}_1}. 
  \end{align*}
  Note that the coefficient 
  $\inner{\u{u}_i-\u{u}_1}{\u{u}_j-\u{u}_1}$ 
  of $t^2$ is $c_{ij}'$ in the case 
  $\u{u}_i=\u{n}_i$. \par
  We first consider the case that $\u{u}_i=\u{n}_i$ 
  for $i=1,2,3$. 
  Then, by using $g_{ij}=g_{ji}$ 
  and 
  $2H_{\triangle}
  = g^{11}b_{11}+g^{22}b_{22}+g^{12}b_{12}+g^{21}b_{21}$, 
  where $g^{ij}$ is the $(i,j)$-component 
  of the inverse matrix of 
  $\fff_{\triangle}
  = (g_{ij})_{i,j}$, 
  \begin{equation}
    \begin{aligned}
      \det(g_{ij}(t))_{i,j} 
      & 
      \begin{aligned}
	= 
	g_{22}(t)g_{33}(t) 
	- 
	g_{23}(t)g_{32}(t) 
      \end{aligned} 
      \\
      &  
      \begin{aligned}
	= 
	\det(g_{ij})_{i,j} 
	& 
	- 2t
	(
	g_{22}b_{11} 
	+ 
	g_{11}b_{22} 
	- 
	g_{12}b_{21} 
	- 
	g_{21}b_{12}
	) 
	\\
	& 
	+ t^2
	(
	g_{11}c_{22}' 
	+ 
	g_{22}c_{11}' 
	- 
	g_{12}c_{21}' 
	- 
	g_{21}c_{12}'
	) 
	\\
	& 
	+ t^2 
	(
	4b_{11}b_{22} 
	- 
	2b_{12}b_{21} 
	- 
	b_{12}^2 
	- 
	b_{21}^2
	)
      \end{aligned} 
      \\
      & 
      \begin{aligned}
	= 
	\det(g_{ij})_{i,j} 
	\left\{
          1 - 4t H_{\triangle} 
          + 
          t^2 
          \left(
            \sum_{i,j=1}^2g^{ij}c_{ij}' 
            + 
            \frac{4\det\sff_{\sym}}{\det(g_{ij})_{i,j}} 
          \right)
	\right\}, 
      \end{aligned}
    \end{aligned}
    \label{Eq.det(gij(t))}
  \end{equation}
  where $\sff_{\sym}:=(\sff_{\triangle}+\sff_{\triangle}^T)/2$ 
  is the symmetrized matrix of the 
  second fundamental form. 
  Using 
  $\sqrt{1+t\lambda+t^2\mu} 
  = 1 + (\lambda/2)t 
  + (\mu/2 - \lambda^2/8)t^2 
  + O(t^3)$ 
  for $\lvert t\rvert\ll1$, 
  \[
  \frac{
    \sqrt{\det(g_{ij}(t))_{i,j}}
  }{
    \sqrt{\det(g_{ij})_{i,j}}
  }
  = 
  1 - 2t H_{\triangle} 
  + 
  t^2 
  \left(
    \sum_{i,j=1}^2\frac{g^{ij}c_{ij}'}{2} 
    + 
    \frac{2\det\sff_{\sym}}{\det(g_{ij})_{i,j}} 
    - 
    2H_{\triangle}^2
  \right) 
  + 
  O(t^3). 
  \]
  Since 
  $A(
  \u{x}_0(t),
  \u{x}_1(t),
  \u{x}_2(t)
  ) 
  = 
  \sqrt{\det(g_{ij}(t))_{i,j}}$, 
  this expansion 
  shows (\ref{Eq.first_var_of_triangle}). 
  \par
  To get (\ref{Eq.second_var_of_triangle}) 
  we continue with more computation 
  of the term involving $t^2$. 
  Using the equality 
  \[
  \sum_{i,j=1}^2g^{ij}c_{ij} 
  = 
  4H_{\triangle}^2 
  - 
  2K_{\triangle} 
  + 
  \frac{(M_1-M_2)^2}{EG-F^2}, 
  \]
  which follows from (\ref{Eq.KI-2HII+III!=0}), we have 
  \begin{align*}
    \sum_{i,j=1}^2\frac{g^{ij}c_{ij}}{2} 
    + 
    \frac{2\det\sff_{\sym}}{\det(g_{ij})_{i,j}} 
    & = 
      2H_{\triangle}^2 
      - 
      K_{\triangle} 
      + 
      \frac{(M_1-M_2)^2}{2(EG-F^2)} 
      + 
      \frac{4LN-(M_1+M_2)^2}{2(EG-F^2)} 
    \\
    & = 
      2H_{\triangle}^2 
      - 
      K_{\triangle} 
      + 
      \frac{2(LN-M_1M_2)}{EG-F^2} 
    \\
    & = 
      2H_{\triangle}^2 
      + 
      K_{\triangle}, 
  \end{align*}
  where the last equality follows from 
  the definition of $K_{\triangle}$. 
  Thus we infer 
  \[
  \frac{
    \sqrt{\det(g_{ij}(t))_{i,j}}
  }{
    \sqrt{\det(g_{ij})_{i,j}}
  }
  = 
  1 - 2t H_{\triangle} 
  + 
  t^2 
  \left\{
    K_{\triangle} 
    + 
    \frac{1}{2}
    \mathrm{tr}
    (\fff_{\triangle}^{-1}
    (\tff_{\triangle}'-\tff_{\triangle}^{})
    )
  \right\}
  + 
  O(t^3). 
  \]
  This proves (\ref{Eq.second_var_of_triangle}). 
  The latter assertion of the proposition 
  immediately follows from 
  (\ref{Eq.second_var_of_triangle}). \par
  We then consider for a general variation 
  vector field $\u{u}$. 
  A similar computation as 
  (\ref{Eq.general_first_var_of_triangle}) 
  shows 
  \[
  \frac{\sqrt{\det(g_{ij}(t))_{i,j}}
  }{
    \sqrt{\det(g_{ij})_{i,j}}
  } 
  = 
  1 - \frac{t}{\det(g_{ij})_{i,j}}
  (g_{33}b_{22}+g_{22}b_{33}-g_{23}b_{32}-g_{32}b_{23}) 
  + O(t^2). 
  \]
  Therefore 
  \begin{align*}
    \left.\frac{d}{dt}\right\rvert_{t=0}
    \sqrt{\det(g_{ij}(t))_{i,j}}
    & = \frac{1}{\sqrt{\det(g_{ij})_{i,j}}}
      \Bigl\{
      \inner{\u{v}_3}{\u{v}_3}
      \inner{\u{v}_2}{\u{u}_2-\u{u}_1}
      + 
      \inner{\u{v}_2}{\u{v}_2}
      \inner{\u{v}_3}{\u{u}_3-\u{u}_1}
    \\
    & 
      \qquad \qquad \qquad \quad 
      - 
      \inner{\u{v}_2}{\u{v}_3}
      \left(
      \inner{\u{v}_2}{\u{u}_3-\u{u}_1}
      + 
      \inner{\u{v}_3}{\u{u}_2-\u{u}_1}
      \right)
      \Bigr\}
    \\
    & = 
      \inner{\u{V}^1}{\u{u}_1} 
      + 
      \inner{\u{V}^2}{\u{u}_2} 
      + 
      \inner{\u{V}^3}{\u{u}_3},  
  \end{align*}
  where $\u{V}^i$ contains 
  neither of $\u{u}_j$ and 
  has the following expression: 
  \begin{align*}
    V^i 
    & = 
      \frac{1}{A(\u{x}_1,\u{x}_2,\u{x}_3)}
      \left\{
      \inner{\u{x}_j-\u{x}_k}{\u{x}_k-\u{x}_i}
      (\u{x}_i-\u{x}_j)
      - 
      \inner{\u{x}_j-\u{x}_k}{\u{x}_i-\u{x}_j}
      (\u{x}_k-\u{x}_i)
      \right\} 
    \\ 
    & = 
      \frac{1}{A(\u{x}_1,\u{x}_2,\u{x}_3)}
      (\u{x}_j-\u{x}_k)\times 
      \left[
      (\u{x}_i-\u{x}_j)\times 
      (\u{x}_k-\u{x}_i)
      \right],  
  \end{align*}
  where $(i,j,k)=(1,2,3)$, $(2,3,1)$ 
  or $(3,1,2)$. 
  Since $(\u{x}_i-\u{x}_j)\times 
  (\u{x}_k-\u{x}_i)$ divided by 
  $A(\u{x}_1,\u{x}_2,\u{x}_3)$ 
  is the (reversed) unit normal vector 
  of $\triangle$, 
  this expression shows that $V^i$ is 
  perpendicular to both 
  the normal vector and 
  $\u{x}_j-\u{x}_k$, 
  proving (\ref{Eq.general_first_var_of_triangle}). 
\end{proof}
\begin{Prop}[general first variation formula for $\mathcal{A}$]
  Let $\vPhi\colon X=(V,E)\rightarrow \mathbb{R}^3$ 
  be a $3$-valent discrete surface 
  and $\u{u}=\{\u{u}_x\}_{x\in V}$ 
  be any vector field on the surface, 
  that is, each $\u{u}_x$ is a vector in 
  $\mathbb{R}^3$ assigned to the vertex $x\in V$. 
  \begin{enumerate}
  \item[{\upshape (\rnum{1})}] 
    The first variation formula for $\vPhi+t\u{u}$ 
    {\upshape(}$t\in \mathbb{R}${\upshape)} 
    is given of the form 
    \[
    \left.\frac{d}{dt}\right\rvert_{t=0}
    \mathcal{A}[\vPhi+t\u{u}] 
    = 
    \sum_{x\in V}
    \inner{
      \u{V}_{x_1}^{e_1} 
      + \u{V}_{x_2}^{e_2} 
      + \u{V}_{x_3}^{e_3}
    }{
      \u{u}_x}, 
    \]
    where $E_x=\{e_1,e_2,e_3\}$ and $t(e_i)=x_i$. 
  \item[{\upshape (\rnum{2})}] 
    The vector $\u{V}_{x_1}^{e_1}$ is obtained by 
    rotating $\u{x}_{12}-\u{x}_{13}$ 
    by $90^{\circ}$ on the tangent plane $T_{x_1}$ at $x_1$ 
    in the counterclockwise direction as viewed 
    facing the normal vector of $T_{x_1}$ 
    {\upshape (see Figure~\ref{Figure(Vxe)})}. 
  \item[{\upshape (\rnum{3})}] 
    When the variation around a vertex $x\in V$ is 
    given by the normal deformation, 
    i.e.\ $\u{u}_{x_i}=\u{n}_{x_i}$ for 
    the adjacent vertices $x_i\in V$ 
    {\upshape(}$i=1,2,3${\upshape)} of $x$, 
    then 
    \[
    \inner{\u{V}_x^{e_1}}{\u{n}_{x_1}} 
    + 
    \inner{\u{V}_x^{e_2}}{\u{n}_{x_2}} 
    + 
    \inner{\u{V}_x^{e_3}}{\u{n}_{x_3}} 
    \]
    is equal to the mean curvature $-2H(x) A(x)$, 
    where $A(x):=\lvert (\u{e}_2-\u{e}_1)\times (\u{e}_3-\u{e}_1)$ 
    is the area element at $x\in V$.
  \end{enumerate}
  \begin{figure}[htpb]
    \centering
    \includegraphics[width=8cm]{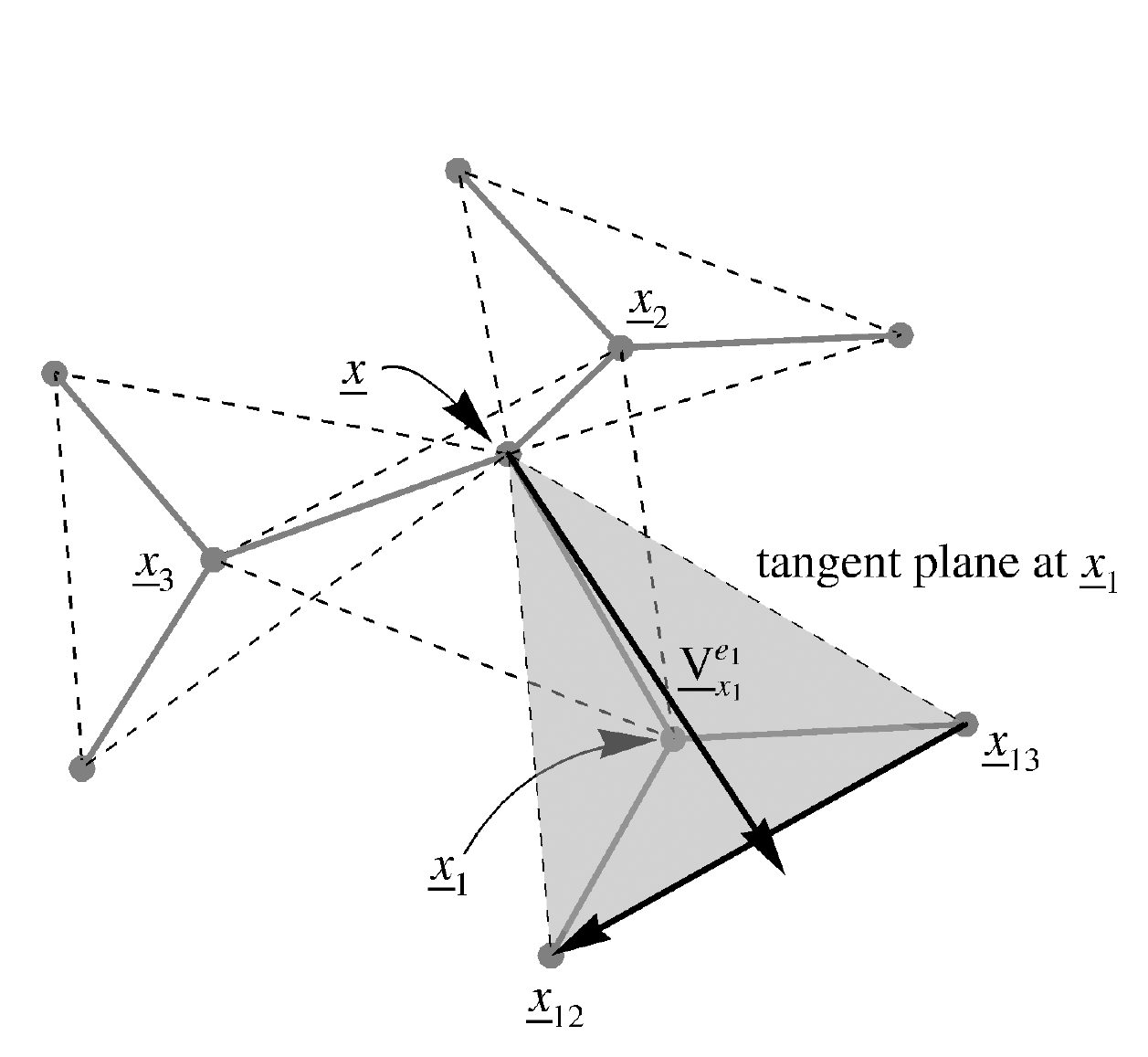}
    \caption{Thick segments are the edges of 
      the surface, and the four dashed triangles are 
      on the tangent planes at $\u{x}$, $\u{x}_i$ 
      ($i=1,2,3$). 
      $\u{V}_{x_1}^{e_1}$ is obtained by 
      rotating $\u{x}_{12}-\u{x}_{13}$ by 
      $90^{\circ}$ on the tangent plane at $\u{x}_1$, 
      on which the gray-hued triangle lies. 
    }
    \label{Figure(Vxe)}
  \end{figure}
\end{Prop}
Now we restate the results 
on the variation formula 
in terms of $3$-valent discrete surfaces. 
\begin{Thm}[normal variation formula for $\mathcal{A}$]
  \label{Thm(variation_formula)}
  Let $\vPhi\colon X=(V,E)\rightarrow \mathbb{R}^3$ 
  be a $3$-valent discrete surface 
  with $\u{n}\colon V\rightarrow \mathbb{R}^3$ its 
  oriented unit normal vector field. 
  The normal variation 
  $\vPhi+t\u{n}$ {\upshape(}$t\in \mathbb{R}${\upshape)} 
  of $\vPhi$ gives the following variation formulas: 
  \begin{align*}
    \left.
    \frac{d}{dt} 
    \right\rvert_{t=0}
    \mathcal{A}
    [\vPhi+t\u{n}]
    & = 
      -2\sum_{x\in V} 
      H(x)A(x), 
    \\
    \left. 
    \frac{d^2}{dt^2} 
    \right\rvert_{t=0} 
    \mathcal{A} 
    [\vPhi+t\u{n}] 
    & = 
      \sum_{x\in V} 
      \left\{ 
      2K(x) 
      + 
      \mathrm{tr}
      (\fff_{\triangle(x)}^{-1}(
      \tff_{\triangle(x)}' 
      - 
      \tff_{\triangle(x)}^{})
      \right\}A(x), 
  \end{align*}
  where $\triangle(x)$ is 
  the triangle with 
  $\{
  \vPhi(t(e_{x,1})), 
  \vPhi(t(e_{x,2})), 
  \vPhi(t(e_{x,3}))
  \}$ 
  as its vertices for $x\in V$, 
  $E_x=\{e_{x,1},e_{x,2},e_{x,3}\}$. 
  If, in particular, 
  $\vPhi+t\u{n}$ 
  is actually a family of parallel surfaces 
  {\upshape(}in the sense that each tangent plane 
  of $\vPhi+t\u{n}$ is parallel 
  to the corresponding one of 
  $\vPhi+t'\u{n}$ for any 
  $t,t'\in \mathbb{R}${\upshape)}, 
  then the second variation is given as 
  \[
  \mathcal{A}[\vPhi+t\u{n}] 
  = 
  \sum_{x\in V}
  \left\{
    1-2t H(x)+t^2 K(x)
  \right\}
  A(x), 
  \]
  which is so-called the Steiner formula. 
\end{Thm}
\begin{proof}
  The former assertion is immediate from 
  Proposition~\ref{Prop(var_of_triangle)}. \par
  Let us consider the case that 
  $\vPhi_t=\vPhi+t\u{n}$ 
  gives a family of parallel surfaces for $t\in \mathbb{R}$. 
  Then, since 
  $\nabla_{e_i-e_1}\vPhi_t 
  = \nabla_{e_i-e_1}\vPhi
  + t \nabla_{e_i-e_1}\u{n}$ for $i=2,3$, 
  we obtain 
  \begin{align*}
    \nabla_{e_2-e_1}\vPhi_t 
    \times 
    \nabla_{e_3-e_1}\vPhi_t 
    & = 
      \nabla_{e_2-e_1}\vPhi 
      \times 
      \nabla_{e_3-e_1}\vPhi 
    \\
    & \qquad - 
      t
      (
      \nabla_{e_3-e_1}\u{n}\times \nabla_{e_2-e_1}\vPhi 
      - 
      \nabla_{e_2-e_1}\u{n}\times \nabla_{e_3-e_1}\vPhi 
      ) 
    \\
    & \qquad + 
      t^2
      (\nabla_{e_3-e_1}\u{n}\times\nabla_{e_2-e_1}\u{n}) 
    \\
    & = 
      A(x)\u{n}(x) 
      - 
      2t H(x)A(x)\u{n}(x) 
      + 
      t^2K(x)A(x)\u{n}(x) 
    \\
    & = (1-2t H(x)+t^2K(x))A(x)\u{n}(x), 
  \end{align*}
  which implies 
  the unit normal vector $\u{n}_t$ 
  of $\vPhi_t$ does not change 
  for sufficiently small $\lvert t\rvert\ll1$. 
  The area element $A_t(x)$ of $\vPhi_t$ 
  is then computed as 
  \[
  A_t(x) 
  = 
  \lvert 
  \nabla_{e_2-e_1}\vPhi_t
  \times 
  \nabla_{e_3-e_1}\vPhi_t
  \rvert 
  = 
  (1-2t H(x)+t^2K(x))A(x), 
  \]
  as required. 
\end{proof}
\begin{Rem}
  Unlike the classical surface theory, 
  a tangential first variation 
  of a discrete surface may not vanish, 
  so that a surface with $H=0$ 
  is not always an extremum of area $\mathcal{A}$. 
\end{Rem}
%
\subsection{Harmonic and minimal surfaces}
\label{harmonic_minimal}
The area of a smooth regular surface 
$\s{p}\colon \Omega \to \mathbb{R}^3$, 
where $\Omega\subseteq \mathbb{R}^2$ is a domain, 
is dominated by its Dirichlet energy: 
\begin{equation}
  \int_{\Omega} 
  \sqrt{
    \lvert \partial_u\s{p} \rvert^2 
    \lvert \partial_v\s{p} \rvert^2 
    - 
    \inner{\partial_u\s{p}}{\partial_v\s{p}}^2
  }\,dudv 
  \leq 
  \frac{1}{2} 
  \int_{\Omega} 
  \left(
    \lvert \partial_u\s{p} \rvert^2 
    + 
    \lvert \partial_v\s{p} \rvert^2
  \right)\,dA, 
  \label{Eq.area<energy}
\end{equation}
and the equality holds if and only if $\s{p}$ 
is conformal in the sense that 
$\lvert \partial_u\s{p}\rvert^2
= 
\lvert \partial_v\s{p}\rvert^2$ 
and 
$\inner{\partial_u\s{p}}{\partial_v\s{p}}=0$. 
To solve the Plateau Problem, 
Douglas and Rad\'o (1930s) 
came up with the idea to minimize 
the Dirichlet energy 
instead of the area functional itself 
for several advantageous reasons. 
In our settings, the corresponding 
Dirichlet energy is given as 
the sum of square norm of the edges. 
A (periodic) realization of a graph which minimizes 
such an energy is called 
a \emph{harmonic realization} in \cite{MR1783793} 
or an \emph{equilibrium placement} 
in \cite{Delgado-Friedrichs-OKeeffe}. 
An elementary result, 
related with our settings, 
corresponding to (\ref{Eq.area<energy}) 
is described as follows. 
\begin{Prop}
  Let $\triangle(\u{x}_1,\u{x}_2,\u{x}_3)$ 
  be a triangle with vertices 
  $\{\u{x}_1,\u{x}_2,\u{x}_3\}\subseteq \mathbb{R}^3$ 
  and let $A$ be its area. 
  For any point $\u{x}\in \mathbb{R}^3$, 
  it follows 
  \[
  \frac{4\sqrt{3}}{3}A
  \leq 
  \lvert \u{x}_1-\u{x} \rvert^2 
  + 
  \lvert \u{x}_2-\u{x} \rvert^2 
  + 
  \lvert \u{x}_3-\u{x} \rvert^2. 
  \]
  The equality holds if and only if 
  $\triangle(\u{x}_1,\u{x}_2,\u{x}_3)$ 
  is an equilateral triangle 
  and $\u{x}$ is located at its barycenter. 
\end{Prop}
\begin{Def}
  Let $X=(V,E,m)$ be a weighted graph with weight 
  $m\colon E \to (0,\infty)$ 
  satisfying $m(e)=m(\bar{e})$. 
  A discrete surface 
  $\vPhi\colon X=(V,E,m) \to \mathbb{R}^3$ 
  is said to be \emph{harmonic with weight $m$} if 
  it is a harmonic realization with weight $m$, 
  that is, if it satisfies 
  \begin{equation}
    m(e_{x,1})
    \vPhi(e_{x,1}) 
    + 
    m(e_{x,2})
    \vPhi(e_{x,2}) 
    + 
    m(e_{x,3})
    \vPhi(e_{x,3}) 
    = 
    \u{0}
    \label{Eq.disc_harm}
  \end{equation}
  for every vertex $x\in V$, where 
  $E_x=\{e_{x,1},e_{x,2},e_{x,3}\}$. 
\end{Def}
Exact representation of $H$ and $K$ in the case of discrete harmonic surfaces is given as follows. 
\begin{Prop}\label{Prop(harmonic_H&K)}
  Let $X=(V,E,m)$ be a weighted graph with weight 
  $m\colon E \to (0,\infty)$ 
  satisfying $m(e)=m(\bar{e})$, and 
  $\vPhi\colon X=(V,E,m) \to \mathbb{R}^3$ 
  be a $3$-valent discrete harmonic surface, 
  $x\in V$ be fixed and $E_x=\{e_1,e_2,e_3\}$. 
  Then 
  the mean curvature $H(x)$ 
  and the Gauss curvature $K(x)$ 
  are, respectively, written as 
  \begin{align}
    H(x) 
    & = 
      \frac{m_1+m_2+m_3}{2A(x)^2}
      \sum_{(\alpha,\beta,\gamma)}
      \frac{
      \inner{\u{e}_{\alpha}}{\u{e}_{\beta}}
      (
      \inner{\u{e}_{\alpha}}{\u{n}_{\beta}}
      + 
      \inner{\u{e}_{\beta}}{\u{n}_{\alpha}}
      )
      }{
      m_{\gamma}
      }, \label{Eq.harmonic_H} 
    \\
    K(x) 
    & = 
      - \frac{m_1+m_2+m_3}{2A(x)^2}
      \sum_{(\alpha,\beta,\gamma)}
      \frac{
      \inner{\u{e}_{\alpha}}{\u{n}_{\beta}}
      \inner{\u{e}_{\beta}}{\u{n}_{\alpha}}
      }{
      m_{\gamma}
      }, \label{Eq.harmonic_K}
  \end{align}
  where $m_i=m(e_i)$, 
  $A(x)
  = \lvert \u{e}_1\times\u{e}_2
  + \u{e}_2\times\u{e}_3 
  + \u{e}_3\times\u{e}_1 \rvert$, 
  $\u{e}_i=\nabla_{e_i}\vPhi=\vPhi(e_i)\in T_x$ 
  is a tangent vector at $\vPhi(x)$, 
  $\u{n}_i=\u{n}(t(e_i))$ is the oriented unit normal vector 
  at each adjacent vertex of $\vPhi(x)$, for $i=1,2,3$, 
  and the summations are taken over any 
  $(\alpha,\beta,\gamma)\in \{(1,2,3),(2,3,1),(3,1,2)\}$. 
\end{Prop}
\begin{proof}
  We first make the following observations 
  which are easily proved from (\ref{Eq.disc_harm}): 
  \begin{enumerate}
  \item[(\rnum{1})] 
    Every $\vPhi(e_i)$ lies on 
    the tangent plane $T_x$ at $\vPhi$, 
    so that 
    $\u{e}_i
    =\nabla_{e_i}\vPhi
    =\vPhi(e_i)\in T_x$ 
    for $i=1,2,3$. 
  \item[(\rnum{2})] 
    $m_3^{-1}(\u{e}_1\times \u{e}_2) 
    = m_1^{-1}(\u{e}_2\times \u{e}_3) 
    = m_2^{-1}(\u{e}_3\times \u{e}_1)$ 
    and is parallel to $\u{n}(x)$. 
  \end{enumerate}
  Let $(\alpha,\beta)=(1,2)$, $(2,3)$ or $(3,1)$ 
  be fixed. 
  The first fundamental form $\fff_{\alpha\beta}$ and 
  the second fundamental form $\sff_{\alpha\beta}$ 
  the triangle 
  $\triangle_{\alpha\beta} 
  = \triangle
  (\vPhi(x),
  t(\u{e}_{\alpha}),
  t(\u{e}_{\beta}))$ 
  are, respectively, written as 
  \[
  \fff_{\alpha\beta} 
  = 
  \begin{pmatrix}
    \inner{\u{e}_{\alpha}}{\u{e}_{\alpha}} 
    & 
    \inner{\u{e}_{\alpha}}{\u{e}_{\beta}} 
    \\
    \inner{\u{e}_{\beta}}{\u{e}_{\alpha}} 
    & 
    \inner{\u{e}_{\beta}}{\u{e}_{\beta}} 
  \end{pmatrix}, \quad 
  \sff_{\alpha\beta} 
  = 
  \begin{pmatrix}
    0 
    & 
    -\inner{\u{e}_{\alpha}}{\u{n}_{\beta}} 
    \\
    -\inner{\u{e}_{\beta}}{\u{n}_{\alpha}} 
    & 
    0 
  \end{pmatrix}
  \]
  because 
  $\inner{\u{e}_{\alpha}}{\u{n}_{\alpha}} 
  = 0
  = \inner{\u{e}_{\beta}}{\u{n}_{\beta}}$ by (\rnum{2}). 
  Then we have 
  \begin{align}
    H_{\triangle_{\alpha\beta}} 
    & = 
      \frac{
      \inner{\u{e}_{\alpha}}{\u{e}_{\beta}}
      (
      \inner{\u{e}_{\alpha}}{\u{n}_{\beta}} 
      + 
      \inner{\u{e}_{\beta}}{\u{n}_{\alpha}}
      )
      }{
      2(
      \lvert \u{e}_{\alpha}\rvert^2 
      \lvert \u{e}_{\beta}\rvert^2 
      - 
      \inner{\u{e}_{\alpha}}{\u{e}_{\beta}}^2
      )
      }, \label{Eq.harmonic_Hab}
    \\
    K_{\triangle_{\alpha\beta}} 
    & = 
      - 
      \frac{
      \inner{\u{e}_{\alpha}}{\u{n}_{\beta}}
      \inner{\u{e}_{\beta}}{\u{n}_{\alpha}}
      }{
      \lvert \u{e}_{\alpha}\rvert^2 
      \lvert \u{e}_{\beta}\rvert^2 
      - 
      \inner{\u{e}_{\alpha}}{\u{e}_{\beta}}^2
      }. \label{Eq.harmonic_Kab}
  \end{align}
  Here we note that 
  \[
  \lvert \u{e}_{\alpha}\rvert^2 
  \lvert \u{e}_{\beta}\rvert^2 
  - 
  \inner{\u{e}_{\alpha}}{\u{e}_{\beta}}^2 
  = 
  \det\fff_{\alpha\beta} 
  = 
  \lvert \u{e}_{\alpha}\times\u{e}_{\beta}\rvert^2 
  = 
  \frac{A(x)^2m_{\gamma}}{m_1+m_2+m_3}, 
  \]
  where $\gamma\neq\alpha,\beta$. 
  The desired expressions are now immediately obtained 
  from 
  \begin{align*}
    \frac{\sqrt{\det\fff_{\alpha\beta}(x)}}{A(x)}
    H_{\triangle_{\alpha\beta}}
    & = 
      \frac{1}{2A(x)\sqrt{\det\fff_{\alpha\beta}(x)}}
      \inner{\u{e}_{\alpha}}{\u{e}_{\beta}}
      (
      \inner{\u{e}_{\alpha}}{\u{n}_{\beta}}
      + 
      \inner{\u{e}_{\beta}}{\u{n}_{\alpha}}
      ) 
    \\
    & = 
      \frac{m_1+m_2+m_3}{2A(x)^2}\cdot 
      \frac{\inner{\u{e}_{\alpha}}{\u{e}_{\beta}}
      (
      \inner{\u{e}_{\alpha}}{\u{n}_{\beta}}
      + 
      \inner{\u{e}_{\beta}}{\u{n}_{\alpha}}
      )}{m_{\gamma}}, 
    \\
    \frac{\sqrt{\det\fff_{\alpha\beta}(x)}}{A(x)}
    K_{\triangle_{\alpha\beta}} 
    & = 
      \frac{m_1+m_2+m_3}{2A(x)^2}\cdot 
      \frac{
      \inner{\u{e}_{\alpha}}{\u{n}_{\beta}}
      \inner{\u{e}_{\beta}}{\u{n}_{\alpha}}
      }{
      m_{\gamma}
      }. 
  \end{align*}
\end{proof}
A discrete harmonic surface needs not be 
minimal in the sense of Definition~\ref{Def(minimal)}, 
but we can provide a sufficient condition 
for a harmonic surface 
to be minimal, 
which is corresponding to 
the conformality of graphs. 
\begin{Thm}\label{Thm(min_harm)}
  Let $X=(V,E,m)$ be a weighted graph with weight 
  $m\colon E\rightarrow (0,\infty)$ 
  satisfying $m(e)=m(\bar{e})$. 
  A $3$-valent harmonic discrete surface 
  $\vPhi:X=(V,E,m) \to \mathbb{R}^3$ 
  is minimal if 
  \begin{equation}
    \inner{\vPhi(e_1)}{\vPhi(e_2)} 
    = 
    \inner{\vPhi(e_2)}{\vPhi(e_3)} 
    = 
    \inner{\vPhi(e_3)}{\vPhi(e_1)}
    \label{Eq.conformal}
  \end{equation}
  holds at every $x\in V$, where $E_x=\{e_1,e_2,e_3\}$. 
  Moreover, if $m\colon E \to (0,\infty)$ 
  is constant, 
  then the condition {\upshape (\ref{Eq.conformal})} 
  is equivalent to 
  \[
  \lvert \vPhi(e_1)\rvert 
  = 
  \lvert \vPhi(e_2)\rvert 
  = 
  \lvert \vPhi(e_3)\rvert. 
  \]
\end{Thm}
\begin{proof}
  We use the same notation as in 
  Proposition~\ref{Prop(harmonic_H&K)}. 
  We then sort (\ref{Eq.harmonic_H}) 
  by terms involving the common $\u{n}_{\alpha}$ 
  to compute 
  \begin{align*}
    H(x) 
    & = 
      \frac{m_1+m_2+m_3}{2A(x)^2m_1m_2m_3}
      \sum_{(\alpha,\beta,\gamma)}
      m_{\alpha}m_{\beta}
      \inner{\u{e}_{\alpha}}{\u{e}_{\beta}}
      (
      \inner{\u{e}_{\alpha}}{\u{n}_{\beta}}
      + 
      \inner{\u{e}_{\beta}}{\u{n}_{\alpha}}
      ) 
    \\
    & = 
      \frac{m_1+m_2+m_3}{2A(x)^2}
      \sum_{(\alpha,\beta,\gamma)}
      \left\{
      m_{\alpha}m_{\beta}
      \inner{\u{e}_{\alpha}}{\u{e}_{\beta}}
      \inner{\u{e}_{\beta}}{\u{n}_{\alpha}} 
      + 
      m_{\gamma}m_{\alpha}
      \inner{\u{e}_{\gamma}}{\u{e}_{\alpha}}
      \inner{\u{e}_{\gamma}}{\u{n}_{\alpha}}
      \right\} 
    \\
    & = 
      \frac{m_1+m_2+m_3}{2A(x)^2}
      \sum_{(\alpha,\beta,\gamma)}
      m_{\alpha}
      \Bigl\langle
      \inner{\u{e}_{\alpha}}{\u{e}_{\beta}}
      m_{\beta}\u{e}_{\beta} 
      + 
      \inner{\u{e}_{\gamma}}{\u{e}_{\alpha}}
      m_{\gamma}\u{e}_{\gamma}, 
      \u{n}_{\alpha}
      \Bigr\rangle, 
  \end{align*}
  which equals zero provided 
  (\ref{Eq.conformal}); 
  $\inner{\u{e}_1}{\u{e}_2}
  = \inner{\u{e}_2}{\u{e}_3}
  = \inner{\u{e}_3}{\u{e}_1}$ 
  holds because 
  $m_{\beta}\u{e}_{\beta}+m_{\gamma}\u{e}_{\gamma}
  = -m_{\alpha}\u{e}_{\alpha}$ 
  is perpendicular to $\u{n}_{\alpha}$. \par
  Moreover, if the weight $m\colon E \to (0,\infty)$ 
  is constant, 
  then the equation (\ref{Eq.disc_harm}) becomes 
  $\u{e}_1+\u{e}_2+\u{e}_3=\u{0}$, 
  which gives
  \begin{align*}
    \lvert \u{e}_{\alpha}\rvert^2 
    & = 
      - 
      \inner{\u{e}_{\alpha}}{\u{e}_{\beta}} 
      - 
      \inner{\u{e}_{\gamma}}{\u{e}_{\alpha}}, 
    \\
    \lvert \u{e}_{\beta}\rvert^2 
    & = 
      - 
      \inner{\u{e}_{\beta}}{\u{e}_{\gamma}} 
      - 
      \inner{\u{e}_{\alpha}}{\u{e}_{\beta}} 
  \end{align*}
  after taking the inner product with 
  $\u{e}_{\alpha}$ and $\u{e}_{\beta}$. 
  This shows 
  $\lvert \u{e}_{\alpha}\rvert=\lvert\u{e}_{\beta}\rvert$ 
  if and only if 
  $\inner{\u{e}_{\gamma}}{\u{e}_{\alpha}}
  = \inner{\u{e}_{\beta}}{\u{e}_{\gamma}}$. 
\end{proof}
%
\section{Several examples}\label{Section(example)}
\subsection{Plane graphs}
\label{subsec:plane_graphs}
A $3$-valent discrete surface 
$\vPhi\colon X=(V,E)\rightarrow \mathbb{R}^3$ 
is said to be a \emph{plane} if 
its image $\vPhi(X)$ lies on 
a plane in $\mathbb{R}^3$. 
Since the second fundamental form of a plane 
vanishes identically, 
independently of the choice of its side at each point, 
so do both its mean curvature and Gauss curvature. 
Since its third fundamental form again vanishes, 
the second variation of the area functional 
also vanishes. 
%
\subsection{Sphere-shaped graphs} 
\label{subsec:sphere_shaped}
\begin{Prop}
  Let $X=(V,E)$ be a finite graph, 
  $\mathbb{S}^2(r)\subseteq \mathbb{R}^3$ be 
  the round sphere with radius $r>0$ and 
  with center at the origin, and 
  $\vPhi\colon X=(V,E)\rightarrow \mathbb{S}^2(r)$ 
  be a $3$-valent discrete surface with the property that 
  \begin{equation}
    \vPhi(x) 
    = 
    r\u{n}(x)
    \label{Eq.phi=rn} 
  \end{equation}
  for every vertex $\u{x}\in V$, 
  where $\u{n}(x)$ is the oriented unit normal vector 
  at $x\in V$. Then the mean curvature $H$ 
  and the Gauss curvature $K$ of $\vPhi$ 
  are given, respectively, as 
  \begin{equation}
    H(x) 
    = 
    -\frac{1}{r}, 
    \quad 
    K(x) 
    = 
    \frac{1}{r^2}
    \label{Eq.H(x)=-1/r&K(x)=1/r^2}
  \end{equation}
  regardless of $x\in V$. 
\end{Prop}
\begin{proof}
  Let $x\in V$ be fixed and let $E_x=\{e_1,e_2,e_3\}$. 
  The assumption (\ref{Eq.phi=rn}) implies for $i=2,3$ that 
  \[
  \nabla_{e_i-e_1}\u{n} 
  = 
  \nabla_{e_i}\u{n} 
  - 
  \nabla_{e_1}\u{n} 
  = 
  \u{n}(t(e_i)) 
  - 
  \u{n}(t(e_1)), 
  \]
  which is parallel to both 
  $\nabla_{e_2-e_1}\vPhi 
  = \vPhi(e_2)-\vPhi(e_1)$ 
  and 
  $\nabla_{e_3-e_1}\vPhi 
  = \vPhi(e_3)-\vPhi(e_1)$ 
  with factor $r$. 
  Therefore the first $\fff_{\triangle(x)}$, 
  second $\sff_{\triangle(x)}$ and 
  third $\tff_{\triangle(x)}$ of 
  the triangle $\triangle(x)$ with vertices 
  $\{\vPhi(t(e_1)),
  \vPhi(t(e_2)),
  \vPhi(t(e_3))\}$ 
  satisfy 
  \begin{equation}
    \fff_{\triangle(x)} 
    = 
    -r\sff_{\triangle(x)} 
    = 
    r^2\tff_{\triangle(x)}, 
    \label{Eq.I=-rII=r^2III}
  \end{equation}
  which proves (\ref{Eq.H(x)=-1/r&K(x)=1/r^2}). 
\end{proof}
\begin{Rem}
  Since {\upshape (\ref{Eq.I=-rII=r^2III})} 
  implies that $\sff_{\triangle(x)}$ is symmetric, 
  it follows
  \[
  K(x)\fff_{\triangle(x)} 
  - 
  2H(x)\sff_{\triangle(x)} 
  + 
  \tff_{\triangle(x)} 
  = 
  0. 
  \]
  Moreover, since 
  $\tff_{\triangle(x)}'=\tff_{\triangle(x)}$, 
  \[
  A_t(x) 
  = 
  (1-2tH(x)+t^2K(x))A(x),\quad 
  \text{for $\lvert t\rvert\ll 1$}
  \]
  holds as well, where 
  $A_t(x)$ {\upshape(}$t\in \mathbb{R}${\upshape)} 
  stands for area of 
  the normal variation of $\triangle(x)$. 
\end{Rem}

\begin{Cor}
  {\upshape(1)} a regular hexahedron, 
  {\upshape(2)} a regular dodecahedron 
  and 
  {\upshape(3)} a regular truncated icosahedron 
  {\upshape(}fullerene $C_{60}${\upshape)} are all 
  $3$-valent discrete surfaces 
  with constant curvatures: 
  \[
  H(x) 
  = 
  -\frac{1}{r}, 
  \quad 
  K(x) 
  = 
  \frac{1}{r^2}, 
  \]
  where $r>0$ is 
  the radius of the round sphere 
  on which these surfaces lie. 
\end{Cor}
\begin{proof}
  It is easily seen that all satisfy (\ref{Eq.phi=rn}).
\end{proof}
%
\subsection{Carbon nanotubes}\label{Section(swnt)}
In this section we will calculate the mean curvature 
and the Gauss curvature of 
a carbon nanotube $\cnt(\lambda,c)$ 
which is, as will be more precisely defined below, 
a regular hexagonal lattice 
wound on the right circular cylinder. 
\par
First, 
let 
$
H(\u{u},\u{\xi})
$
be the \emph{regular hexagonal lattice} 
which is the planer graph 
with vertices 
$\{\u{u}\}\cup\{\u{u}+\u{\xi}_i\}_{i=1}^3$ 
and with (unoriented) edges 
$\{(\u{u},\u{u}+\u{\xi}_i)\}_{i=1}^3$ 
extended by translations via 
$\u{a}_1:=\u{\xi}_2-\u{\xi}_1$ and 
$\u{a}_2:=\u{\xi}_3-\u{\xi}_1$. 
For an arbitrary positive number $\lambda>0$ called the \emph{scale factor}, 
we set the regular hexagonal lattice $X(\lambda)$ by 
\begin{equation}
  X(\lambda) 
  := 
  H(\u{u},\u{\xi}),
  \quad 
  \u{u}
  := 
  \begin{pmatrix}
    0 \\ 0
  \end{pmatrix}, \quad 
  \u{\xi}
  := 
  \lambda
  \begin{pmatrix}
    -\sqrt{3}/2 \\ -1/2
  \end{pmatrix}.
  \label{Eq.Xlambda}
\end{equation}
Let $\u{a}_1:=\rho_{2\pi/3}\u{\xi}-\u{\xi}$ 
and $\u{a}_2:=\rho_{-2\pi/3}\u{\xi}-\u{\xi}$ 
be its lattice vector. 
Note then that a vertex of the hexagonal lattice 
$X(\lambda)$ can be represented as 
\[
\u{\xi} 
= 
\alpha_1\u{a}_1 + \alpha_2\u{a}_2, 
\quad 
\text{
  $(\alpha_1,\alpha_2)
  \in \mathbb{Z}\times\mathbb{Z}$
  ~ or~ 
  $(\alpha_1+1/3,\alpha_2+1/3)
  \in \mathbb{Z}\times \mathbb{Z}$}
\]
and that 
$\u{\xi} = \alpha_1\u{a}_1+\alpha_2\u{a}_2$ and 
$\u{\eta} = \beta_1\u{a}_1+\beta_2\u{a}_2$ 
are mutually adjacent if and only if 
one of the following three conditions is satisfied. 
\begin{enumerate}
\item[{\upshape (\rnum{1})}] 
  $\alpha_1 - \beta_1 = \pm 1/3$, 
  $\alpha_2 - \beta_2 = \pm 1/3$, 
\item[{\upshape (\rnum{2})}] 
  $\alpha_1 - \beta_1 = \mp 2/3$, 
  $\alpha_2 - \beta_2 = \pm 1/3$, 
\item[{\upshape (\rnum{3})}] 
  $\alpha_1 - \beta_1 = \pm 1/3$, 
  $\alpha_2 - \beta_2 = \mp 2/3$, 
\end{enumerate}
where the double-sign corresponds in the same order. 
\begin{Def}\label{Def(cnt)}
  For any pair of integers 
  $c=(c_1,c_2)\in \mathbb{Z}\times\mathbb{Z}$ 
  satisfying $c_1>0$ and $c_2\geq 0$, 
  called a \emph{chiral index} 
  and $\lambda>0$, 
  a \emph{carbon nanotube} 
  $\cnt(\lambda,c)$ is 
  a $3$-valent discrete surface 
  $\vPhi_{\lambda,c}
  \colon X(\lambda)=(V(\lambda),E(\lambda))
  \rightarrow\mathbb{R}^3$ 
  defined by 
  the map
  \[
  \mathbb{R}^2\rightarrow \mathbb{R}^3~;~ 
  \begin{pmatrix}
    x \\ y
  \end{pmatrix}
  \mapsto 
  \begin{pmatrix}
    r(\lambda,c)
    \cos\dfrac{x}{r(\lambda,c)} 
    \\[2ex] 
    r(\lambda,c)
    \sin\dfrac{x}{r(\lambda,c)} 
    \\[1.5ex] 
    y 
  \end{pmatrix}
  \]
  composed with the counterclockwise rotation 
  \[
  \rho_{-\theta(\lambda,c)}
  = 
  \frac{\sqrt{3}}{2L_0(c)}
  \begin{pmatrix}
    2c_1+c_2 & \sqrt{3}c_2 
    \\ 
    -\sqrt{3}c_2 & 2c_1+c_2
  \end{pmatrix}
  \]
  of angle $-\theta(\lambda,c)$, 
  where $\theta(\lambda, c)\in [0,\pi/2)$ is 
  the vector angle between 
  $\u{c}:=c_1\u{a}_1+c_2\u{a}_2$, 
  called the \emph{chiral vector} 
  and $(1,0)^T$, 
  and $r(\lambda,c):=\lvert \u{c}\rvert/(2\pi)$, 
  called the \emph{radius}. 
  More precisely, 
  $\cnt(\lambda,c)$ is the graph 
  in $\mathbb{R}^3$ with 
  $\vPhi_{\lambda,c}
  (V(\lambda))$ 
  as the set of vertices and with 
  \[
  \{\vPhi_{\lambda,c}(t(e))
  - 
  \vPhi_{\lambda,c}(o(e))
  \mid
  e\in E(\lambda)\}
  \]
  as the set of edges. 
\end{Def}
\begin{Rem}
  (1) Let $X(\lambda,c)=(V(\lambda,c),E(\lambda,c))$ 
  be defined as a fundamental region 
  with respect to the translation in $\u{c}$-direction 
  acting on the regular hexagonal lattice 
  $X(\lambda)=(V(\lambda),E(\lambda))$. 
  In other words, $X(\lambda,c)$ is obtained by identifying 
  $\u{x}$ with $\u{x}+n\u{c}$ and $\u{e}$ with $\u{e}+n\u{e}$, 
  respectively, for any 
  $n\in \mathbb{Z}$, $\u{x}\in V(\lambda)$, 
  and $\u{e}\in E(\lambda)$. 
  Then $\cnt(\lambda,c)$ is isomorphic to 
  $X(\lambda,c)$ as an abstract graph; 
  indeed, $\u{x}$ and $\u{x}+n\u{c}$, 
  where $n\in \mathbb{Z}$, 
  are mapped by $\vPhi_{\lambda,c}$ 
  to the same point in $\mathbb{R}^3$. 
  \par
  (2) $\cnt(\lambda,c)$ has period 
  in the direction $\u{t}:=t_1\u{a}_1+t_2\u{a}_2$, where 
  $d(c):=\gcd(c_1+2c_2,2c_1+c_2)$ and 
  \[
  (t_1,t_2) 
  := 
  \left(
    -\frac{c_1+2c_2}{d(c)}, 
    \frac{2c_1+c_2}{d(c)}
  \right), 
  \]
  that is, the image of $\vPhi_{\lambda,c}$ 
  is invariant under a translation in $\u{t}$-direction 
  acting on $X(\lambda)$ (See \cite{Naito:2016}).
\end{Rem}
\begin{figure}[htbp]
  \centering
  \includegraphics[width=8cm]
  {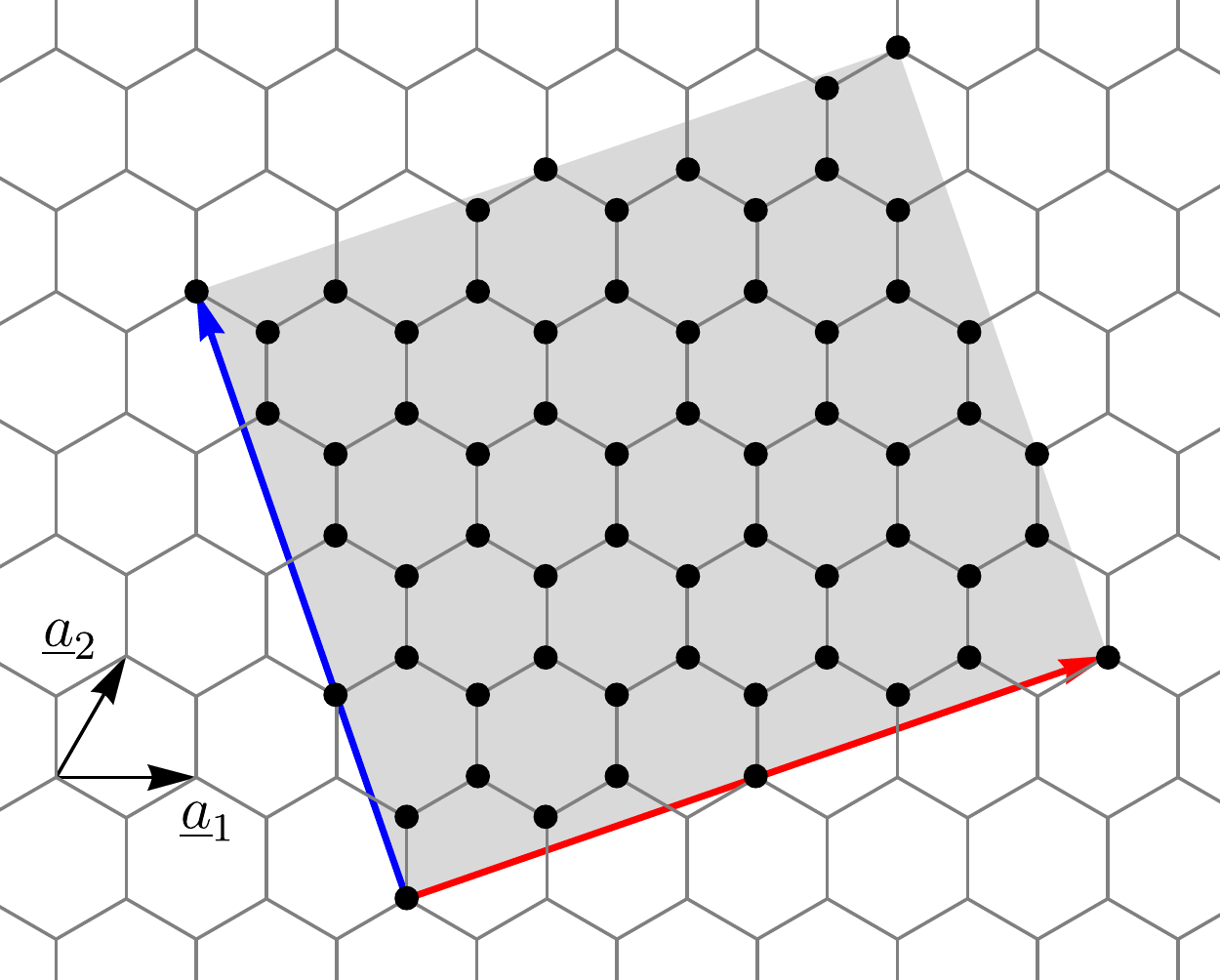}
  \caption{
    $(c_1, c_2) = (4, 2)$, and $(t_1, t_2) = (-4, 5)$. 
    The rectangle is wound on the cylinder 
    along the red line. 
  }
  \includegraphics[width=8cm]
  {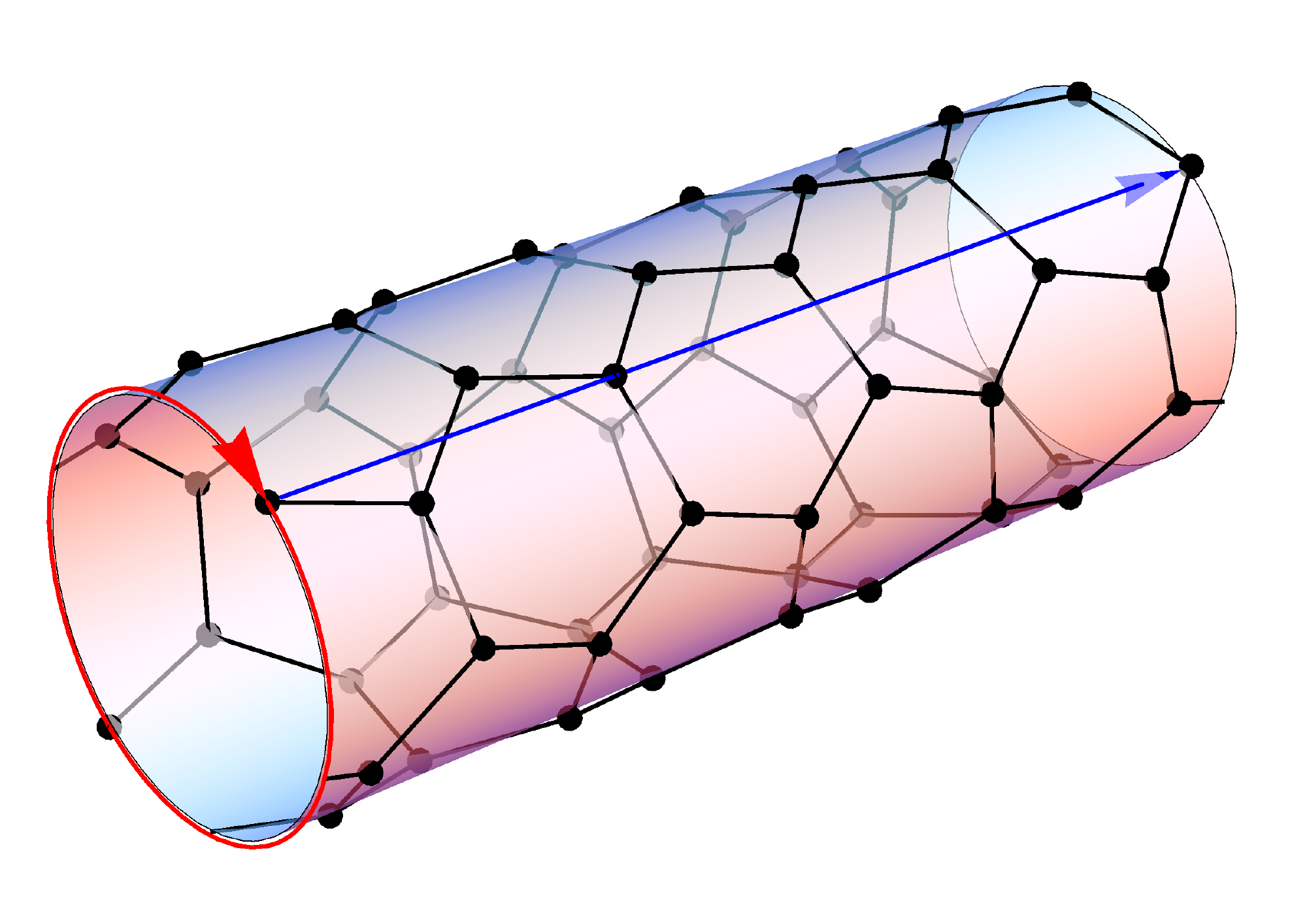}
  \caption{$\cnt(\lambda, (4,2))$}
\end{figure}
Now we come to grips with the calculation 
of the curvatures of $\cnt(\lambda,c)$. 
In what follows, fix $\lambda>0$ and $c=(c_1,c_2)$, 
which are sometimes abbreviated, 
such as $r=r(\lambda,c)$ or $L_0=L_0(c)$. 
\begin{Prop}\label{Prop(vrep)}
  For any 
  $(\alpha_1,\alpha_2)
  \in \mathbb{Z}\times \mathbb{Z}$, 
  a vertex 
  $\u{x}(\alpha_1,\alpha_2) = 
  \vPhi_{\lambda,c}
  (\alpha_1\u{a}_1+\alpha_2\u{a}_2)$ 
  of $\cnt(\lambda,c)$ 
  is represented as 
  \begin{equation}
    \u{x}(\alpha_1,\alpha_2) 
    = 
    R(\phi(\alpha_1,\alpha_2))\u{x}(0,0) 
    - 
    \sqrt{3}r
    \psi(\alpha_1,\alpha_2)(0,0,1)^T, 
    \label{Eq.x(alpha)}
  \end{equation}
  where $\u{x}(0,0)=(r,0,0)^T$, 
  \[
  R(\tau) 
  := 
  \begin{pmatrix}
    \cos\varphi & -\sin\varphi & 0 
    \\
    \sin\varphi & \cos\varphi & 0 
    \\
    0 & 0 & 1
  \end{pmatrix}
  \]
  and 
  \begin{gather}
    (C_1,C_2) 
    := 
    \left(
      \frac{3\pi c_1}{L_0(c)^2},
      \frac{3\pi c_2}{L_0(c)^2}
    \right), 
    \quad 
    (T_1,T_2) 
    := 
    \left(
      \frac{3\pi t_1d(c)}{L_0(c)^2}, 
      \frac{3\pi t_2d(c)}{L_0(c)^2}
    \right), \label{Eq.C1C2T1T2} 
    \\
    \phi(\alpha_1,\alpha_2) 
    = T_2\alpha_1 - T_1\alpha_2, \quad 
    \psi(\alpha_1,\alpha_2)
    = C_2\alpha_1 - C_1\alpha_2. \notag
  \end{gather}
\end{Prop}
\begin{Lem}\label{Lem(0only)}
  The value of the mean curvature as well as 
  the Gauss curvature  of $\cnt(\lambda,c)$ at 
  $\u{x}(\alpha_1,\alpha_2)$ 
  coincide with 
  those at $\u{x}(0,0)$, respectively. 
\end{Lem}
\begin{proof}
  For any $\alpha_1\u{a}_1+\alpha_2\u{a}_2\in V(\lambda)$, 
  let $T(\alpha_1,\alpha_2)\subseteq X(\lambda)$ be the tree 
  consisting of $\alpha_1\u{a}_1+\alpha_2\u{a}_2$ as well as 
  all the chains of length $2$ 
  starting with $\alpha_1\u{a}_1+\alpha_2\u{a}_2$, 
  which completely determines 
  the curvatures at $\u{x}(\alpha_1,\alpha_2)$. 
  Then $T(\alpha_1,\alpha_2)$ is obtained from 
  the translation either 
  (\rnum{1}) of $T(0,0)$ or 
  (\rnum{2}) of $T(-1/3,-1/3)$ 
  by $\beta_1\u{a}_1+\beta_2\u{a}_2$ 
  for some 
  $(\beta_1,\beta_2) 
  \in \mathbb{Z}\times \mathbb{Z}$. \par
  In the case of (\rnum{1}), 
  $\u{x}(T(\alpha_1,\alpha_2))
  \subseteq \cnt(\lambda,c)$ 
  is the rotation of angle $\phi(\beta_1,\beta_2)$ 
  around $z$-axis composed with 
  the translation by 
  $(0,0,-\sqrt{3}r\psi(\beta_1,\beta_2))^T$ of 
  $\u{x}(T(0,0))$. 
  Thus the mean curvature as well as 
  the Gauss curvature at $\u{x}(\alpha_1,\alpha_2)$ clearly 
  coincides with those at $\u{x}(0,0)$, respectively. 
  \par
  While, in the case of (\rnum{2}), 
  $\u{x}(T(\alpha_1,\alpha_2))$ is the rotation of 
  $\u{x}(T(\beta_1,\beta_2))$ of angle $180^{\circ}$ 
  around the axis which intersects orthogonally with $z$-axis 
  and through the point 
  \[
  \frac{1}{2}(\u{x}(\alpha_1,\alpha_2) 
  + \u{x}(\beta_1,\beta_2))
  = 
  \frac{1}{2} 
  (
  \u{x}(\beta_1-1/3,\beta_2-1/3) 
  + 
  \u{x}(\beta_1,\beta_2)
  ). 
  \]
  Therefore 
  again the curvatures  
  at $\u{x}(\alpha_1,\alpha_2)$ coincide with 
  those at $\u{x}(\alpha_1,\alpha_2)$ and thus coincides with 
  those at $\u{x}(0,0)$. 
\end{proof}

Let 
\[
\u{x}_1
:= 
\u{x}(-1/3,-1/3),\quad 
\u{x}_2
:= 
\u{x}(+2/3,-1/3),\quad 
\u{x}_3
:= 
\u{x}(-1/3,+2/3)	
\]
in the sequel. 
A normal vector of $\cnt(\lambda,c)$ is 
computed as follows. 
\begin{Prop}\label{Prop(cnt_normal)}
  For $(\alpha_1,\alpha_2)\in \mathbb{Z}\times\mathbb{Z}$, 
  the outer unit normal vector 
  $\u{n}(\alpha_1,\alpha_2)$ at $\u{x}(\alpha_1,\alpha_2)$ 
  is given as the rotation 
  $\u{n}(\alpha_1,\alpha_2)=R(\phi(\alpha_1,\alpha_2))\u{n}_0$ 
  of the outer unit normal vector $\u{u}_0=\u{n}(0,0)$ at 
  $\u{x}(0,0)$, 
  where 
  $\u{n}_0=\u{m}_0/\lvert \u{m}_0 \rvert$, 
  $\u{m}_0 
  = \u{x}_1\times\u{x}_2 
  + \u{x}_2\times\u{x}_3 
  + \u{x}_3\times\u{x}_1$ 
  having the coordinates
  \[
  \u{m}_0 
  = 
  2\sqrt{3}r^2
  \begin{pmatrix}
    C_1\cos\dfrac{C_2}{2}\sin\dfrac{T_2}{2} 
    - 
    C_2\cos\dfrac{C_1}{2}\sin\dfrac{T_1}{2} 
    \\[1.5ex]
    -C_1\sin\dfrac{C_2}{2}\sin\dfrac{T_2}{2} 
    + 
    C_2\sin\dfrac{C_1}{2}\sin\dfrac{T_1}{2} 
    \\[1.5ex]
    \dfrac{2}{\sqrt{3}}
    \sin\dfrac{T_1}{2}\sin\dfrac{T_2}{2}
    \sin\dfrac{T_1+T_2}{2}
  \end{pmatrix} 
  =: 
  2\sqrt{3}r^2
  \begin{pmatrix}
    m_x(c) \\
    m_y(c) \\
    \dfrac{2}{\sqrt{3}}m_z(c)
  \end{pmatrix}. 
  \]
  While for $(\beta_1+1/3,\beta_2+1/3)\in \mathbb{Z}\times\mathbb{Z}$, the outer unit normal vector 
  $\u{n}(\beta_1,\beta_2)$ 
  at $\u{x}(\beta_1,\beta_2)$ 
  is given as the rotation 
  $\u{n}(\beta_1,\beta_2)
  = R(\phi(\beta_1+1/3,\beta_2+1/3))\u{n}_1$ 
  of the outer unit normal vector 
  $\u{n}_1=\u{n}(-1/3,-1/3)$ at 
  $\u{x}(-1/3,-1/3)$, where 
  $\u{n}_1=\u{m}_1/\lvert\u{m}_1\rvert$, 
  $\u{m}_1$ satisfies 
  $\lvert \u{m}_1\rvert=\lvert \u{m}_0\rvert$ 
  and has the coordinates 
  \[
  \u{m}_1 
  = 2\sqrt{3}r^2
  \begin{pmatrix}
    C_1\sin\dfrac{T_2}{2}\cos\dfrac{T_2}{2} 
    - 
    C_2\sin\dfrac{T_1}{2}\cos\dfrac{T_1}{2} 
    \\[1.5ex]
    -C_1\sin^2\dfrac{T_2}{2} 
    - 
    C_2\sin^2\dfrac{T_1}{2} 
    \\[1.5ex]
    -\dfrac{2}{\sqrt{3}}\sin\dfrac{T_1}{2}\sin\dfrac{T_2}{2}
    \sin\dfrac{T_1+T_2}{2}
  \end{pmatrix} 
  =: 
  2\sqrt{3}r^2
  \begin{pmatrix}
    m_{1,x}(c) \\
    m_{1,y}(c) \\
    -\dfrac{2}{\sqrt{3}}m_{1,z}(c) 
  \end{pmatrix}. 
  \]
\end{Prop}
Using Proposition~\ref{Prop(vrep)}, 
Lemma~\ref{Lem(0only)} 
and Proposition~\ref{Prop(cnt_normal)}, 
we can obtain the following result. 
\begin{Thm}\label{Thm(H&K_of_cnt)}
  The carbon nanotube $\cnt(\lambda,c)$ 
  with scale factor $\lambda>0$ 
  and with chiral index $c=(c_1,c_2)$ 
  has the constant mean curvature 
  \[
  H(\lambda,c) 
  = 
  -\frac{m_x(c)}{2r(\lambda,c)}
  \cdot
  \frac{
    m_x(c)^2+m_y(c)^2+(8/3)m_z(c)^2
  }{
    (m_x(c)^2+m_y(c)^2+(4/3)m_z(c)^2)^{3/2}
  }
  \]
  and has the constant Gauss curvature 
  \[
  K(\lambda,c)
  = 
  \frac{
    4m_z(c)^2(m_x(c)^2+m_y(c)^2)
  }{
    3r(\lambda,c)^2
    (m_x(c)^2+m_y^2(c)+(4/3)m_z(c)^2)^2
  }, 
  \]
  where $r(\lambda,c)=\lvert \u{c}\rvert/(2\pi)>0$ is the radius 
  of the circular cylinder 
  on which $\cnt(\lambda,c)$ winds, and 
  \begin{align*}
    m_x(c) 
    & = 
      C_1\cos\frac{C_2}{2}\sin\frac{T_2}{2}
      - 
      C_2\cos\frac{C_1}{2}\sin\frac{T_1}{2}, 
    \\
    m_y(c) 
    & = 
      -C_1\sin\frac{C_2}{2}\sin\frac{T_2}{2}
      + 
      C_2\sin\frac{C_1}{2}\sin\frac{T_1}{2}, 
    \\
    m_z(c) 
    & = 
      \sin\frac{T_1}{2}\sin\frac{T_2}{2}
      \sin\frac{T_1+T_2}{2}, 
    \\
    (C_1,C_2) 
    & = 
      \left(
      \frac{3\pi c_1}{L_0(c)^2}, 
      \frac{3\pi c_2}{L_0(c)^2}
      \right), 
    \\
    (T_1,T_2)
    & = 
      \left(
      -\frac{3\pi(c_1+2c_2)}{L_0(c)^2}, 
      \frac{3\pi(2c_1+c_2)}{L_0(c)^2} 
      \right), 
    \\
    L_0(c) 
    & = 
      \frac{\lvert \u{c}\rvert}{\lambda} 
      = 
      \sqrt{3(c_1^2+c_1c_2+c_2^2)}. 
  \end{align*}
  If, in particular, $c_1=c_2$, then, $m_z(c)=0$, so that 
  $H(\lambda,c)$ and $K(\lambda,c)$ respectively have 
  the following representations
  \[
  H(\lambda,c) 
  = 
  -\frac{1}{2r(\lambda,c)}
  \cos\frac{C_1}{2}, 
  \quad 
  K(\lambda,c) 
  = 
  0. 
  \]
\end{Thm}
\begin{proof}
  By Lemma~\ref{Lem(0only)}, 
  it suffices to determine the value of the mean curvature 
  and the Gauss curvature of $\cnt(\lambda,c)$ 
  only at $\u{x}(0,0)$. 
  We will make use of Proposition~\ref{Prop(large_triangle)} 
  rather than 
  calculate according to 
  the original definition 
  (Definition~\ref{Def(H&K)}). 
  Namely, we choose the frame 
  \[
  \u{v}_1 
  := 
  \u{x}_2-\u{x}_1, 
  \quad 
  \u{v}_2 
  := 
  \u{x}_3-\u{x}_1
  \]
  to determine the first fundamental form $\fff$ 
  and the second fundamental form $\sff$ 
  of $\cnt(\lambda,c)$ at $\u{x}(0,0)$ 
  The results are follows. 
  \[
  \fff 
  = 
  \begin{pmatrix}
    \lvert \u{v}_1\rvert^2 & 
    \inner{\u{v}_1}{\u{v}_2} 
    \\
    \inner{\u{v}_2}{\u{v}_1} & 
    \lvert \u{v}_2\rvert^2
  \end{pmatrix}, 
  \quad 
  \sff 
  = 
  \begin{pmatrix}
    -\inner{\u{v}_1}{\u{n}_2-\u{n}_1} & 
    -\inner{\u{v}_1}{\u{n}_3-\u{n}_1}  
    \\
    -\inner{\u{v}_2}{\u{n}_2-\u{n}_1} & 
    -\inner{\u{v}_2}{\u{n}_3-\u{n}_1} 
  \end{pmatrix}, 
  \]
  where $\u{n}_i$ is the outer unit normal vector 
  at $\u{x}_i$ ($i=1,2,3$) and 
  \begin{align*}
    \lvert \u{v}_1\rvert^2 
    & = 
      r^2
      \left(
      4\sin^2\frac{T_2}{2}+3C_1^2
      \right), 
    \\
    \lvert \u{v}_2\rvert^2 
    & = 
      r^2
      \left(
      4\sin^2\frac{T_1}{2}+3C_2^2
      \right), 
    \\
    \inner{\u{v}_1}{\u{v}_2}
    & = 
      -r^2
      \left(
      4\sin\frac{T_1}{2}\sin\frac{T_2}{2}\cos\frac{T_1+T_2}{2} + 
      3C_1C_2
      \right), 
    \\
    -\inner{\u{v}_1}{\u{n}_2-\u{n}_1}
    & = 
      -\frac{8\sqrt{3}r^3m_x}{\lvert \u{m}_0 \rvert}
      \sin^2\frac{T_2}{2}, 
    \\
    -\inner{\u{v}_2}{\u{n}_3-\u{n}_1}
    & = 
      -\frac{8\sqrt{3}r^3m_x}{\lvert \u{m}_0 \rvert}
      \sin^2\frac{T_1}{2}, 
    \\
    -\inner{\u{v}_1}{\u{n}_3-\u{n}_1}
    & = 
      \frac{8\sqrt{3}r^3}{\lvert \u{m}_0\rvert}
      \sin\frac{T_1}{2}\sin\frac{T_2}{2}
      \left(
      m_x
      \cos\frac{T_1+T_2}{2}
      - 
      m_y
      \sin\frac{T_1+T_2}{2}
      \right) 
    \\
    -\inner{\u{v}_2}{\u{n}_2-\u{n}_1}
    & = 
      \frac{8\sqrt{3}r^3}{\lvert \u{m}_0 \rvert}
      \sin\frac{T_1}{2}\sin\frac{T_2}{2}
      \left(
      m_x
      \cos\frac{T_1+T_2}{2}
      + 
      m_y
      \sin\frac{T_1+T_2}{2}
      \right). 
  \end{align*}
  The rest of the proof 
  is also a direct computation of 
  \[
  H(\lambda,c) 
  = 
  \frac{1}{2}\mathrm{tr}(\fff^{-1}\sff), 
  \quad 
  K(\lambda,c) 
  = 
  \det(\fff^{-1}\sff). 
  \]
\end{proof}
\subsection{Mackay-like crystals}\label{Section(mackay_min)}
The Schwarzian surface of type $P$ and $D$ and Gyroid 
are well known example of triply periodic minimal surfaces 
in $\mathbb{R}^3$. 
Examples of $3$-valent discrete surfaces which is claimed to
lie on the Schwarzian surface of $P$-type
have been known since 1990s. 
In 1991, A.\ L.\ Mackay and 
H.\ Terrones~\cite{Mackay-Terrones:1991} 
proposed a $sp^2$-bonding (hence a $3$-valent discrete surface)
carbon crystal, now called the \emph{Mackay crystal}, 
consists of $6$- and $8$-membered rings 
(see Figure~\ref{Figure(mackay_crys)}). 
There is no proof for it actually lies on the Schwarzian surface 
but it has the same symmetries as those of 
the Schwarzian surface of type $P$.
T.\ Lenosky et al.~\cite{Lenosky:1992} introduced 
another $3$-valent discrete surfaces consisting of 
$6$- and $7$-membered rings. 
More recently, the first and second author 
et al.~\cite{Tagami:2014} 
systematically investigated carbon crystals of 
$sp^2$-bonding (i.e.\ $3$-valent) with the octahedral symmetry 
and the dihedral symmetry, 
and listed up all the possible structures 
with small number of vertices. 
They are called {\it Mackay-like crystals}. \par
In \cite{Tagami:2014},  a standard realization 
of a Mackay-like crystal  in $\mathbb{R}^3$
is obtained by solving 
a solvable linear system 
and that the lattice vectors 
form an orthogonal basis 
all with equal length. 
Thus the Gauss curvature as well as 
the mean curvature of them are computed explicitly. 
Please see Figure~\ref{Figure(K_or_mackay-like)} 
for the Gauss curvature 
and Figure~\ref{Figure(H_for_mackay-like)} 
for the mean curvature 
how the curvatures are distributed on them. 
\begin{figure}[htbp]
  \centering
  \begin{tabular}{cc}
    \includegraphics[width=150pt]{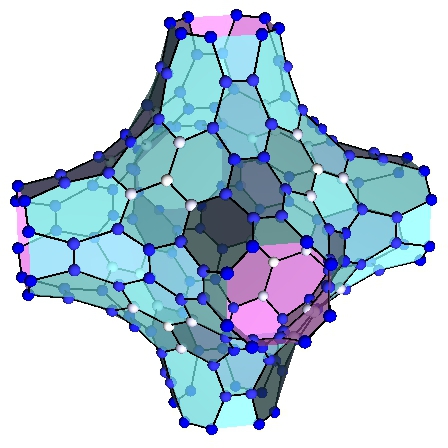}
    &
      \includegraphics[width=150pt]{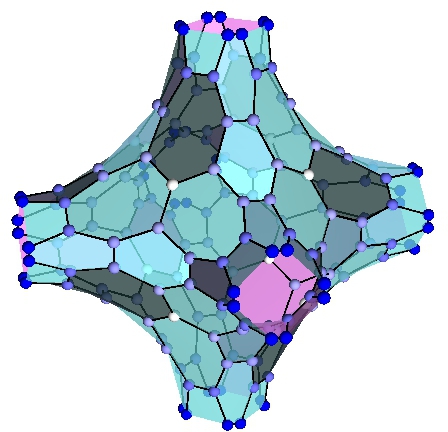}
      \cr
      \includegraphics[width=150pt]{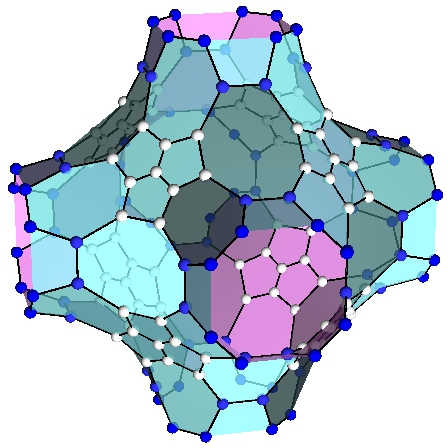}
    &
      \includegraphics[width=150pt]{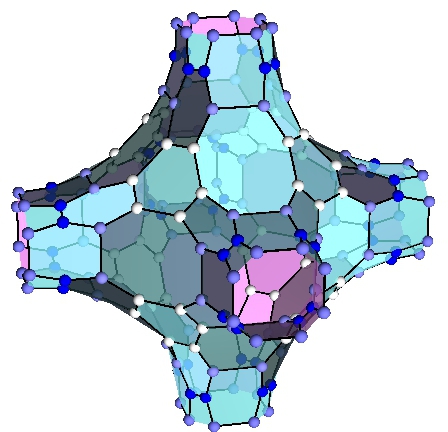}
  \end{tabular}
  \caption{
    The Gauss curvature of Mackay-like crystals founded in \cite{Tagami:2014}.
    The Gauss curvature attain smallest (negative, largest absolutely) values at the most blue points
    in the respective pictures.
    while zero at the white points. The color on the faces are linearly
    interpolated between vertices.
  }
  \label{Figure(K_or_mackay-like)}
\end{figure}
\begin{figure}[htbp]
  \centering
  \begin{tabular}{cc}
    \includegraphics[width=150pt]{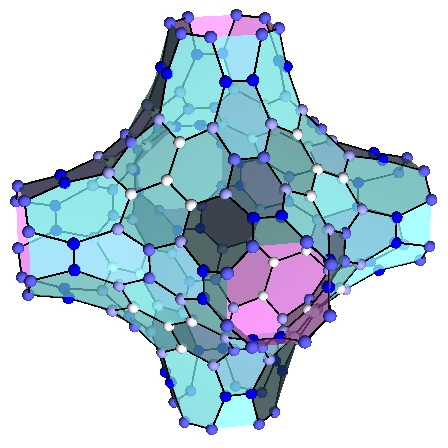}
    &
      \includegraphics[width=150pt]{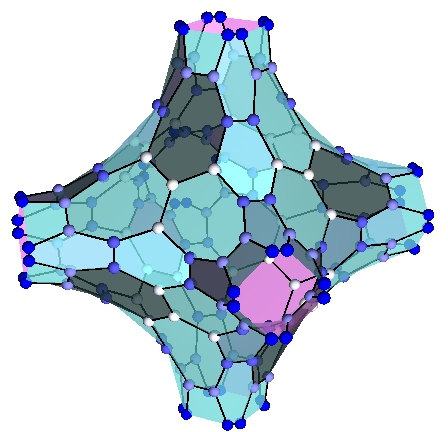}
      \cr
      \includegraphics[width=150pt]{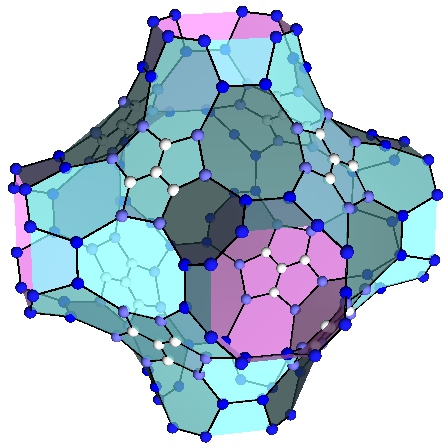}
    &
      \includegraphics[width=150pt]{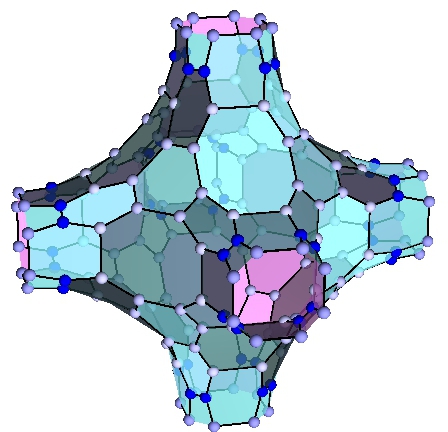}
  \end{tabular}
  \caption{
    The mean curvature of Mackay-like crystals founded in \cite{Tagami:2014}.
    The mean curvature attain smallest (negative, largest absolutely) values at the most blue points
    in the respective pictures.
    while zero at the white points. The color on the faces are linearly
    interpolated between vertices.
  }
  \label{Figure(H_for_mackay-like)}
\end{figure}
\begin{figure}[htbp]
  \centering
  \includegraphics[scale=0.6]
  {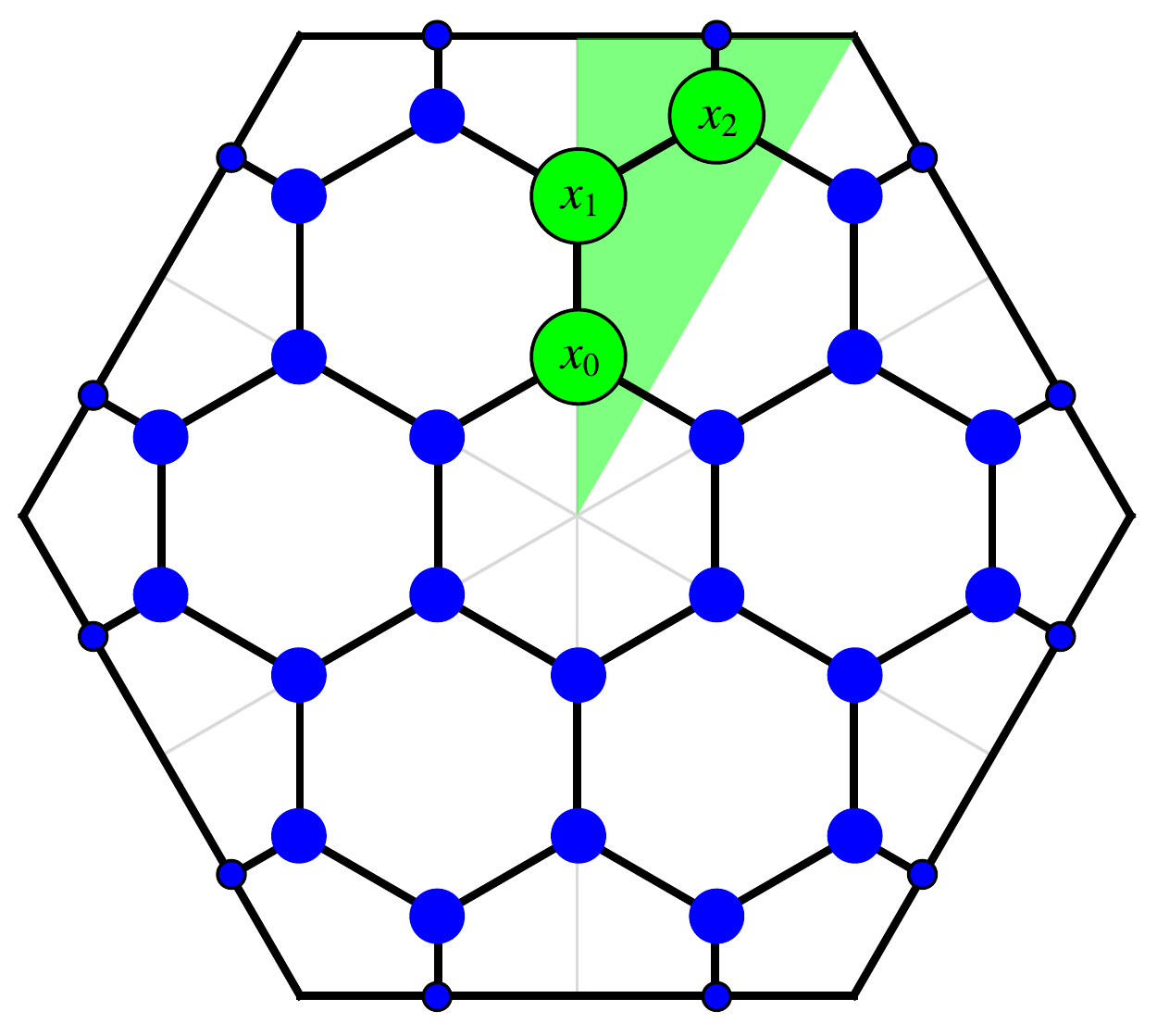}
  \caption{The fundamental region 
    for the octahedral symmetry
  }
  \label{Figure(hex_dom_of_P)}
\end{figure}
The classical Mackay crystal~\cite{Mackay-Terrones:1991} 
has, as an abstract graph, the octahedral symmetry, 
whose fundamental region is shown in 
{\upshape Figure~\ref{Figure(hex_dom_of_P)}}. 
It has further symmetry 
and its smallest patch is 
a subgraph with three vertices 
lying on the green-hued domain in 
{\upshape Figure~\ref{Figure(hex_dom_of_P)}}. 
\par
A standard realization 
$\vPhi_0\colon X=(V,E)\rightarrow \mathbb{R}^3$ 
for the classical Mackay crystal, 
whose lattice vectors $\{\u{e}_x,\u{e}_y,\u{e}_z\}$ 
form an orthogonal frame of $\mathbb{R}^3$, say, 
$\u{e}_x=(2,0,0)^T$, 
$\u{e}_y=(0,2,0)^T$ 
and 
$\u{e}_z=(0,0,2)^T$. 
Then $\vPhi_0$ is the unique solution of the system 
\[
\left\{
  \begin{aligned}
    0 
    & = 
    \vPhi_0(x_1) 
    + 
    R_{xy}(L(\vPhi_0(x_0))) 
    + 
    L(\vPhi_0(x_0)) 
    - 
    3\vPhi_0(x_0), 
    \\
    0 
    & = 
    \vPhi_0(x_0) 
    + 
    \vPhi_0(x_2) 
    + 
    R_{xy}(\vPhi_0(x_2)) 
    - 
    3\vPhi_0(x_1), 
    \\
    0 
    & = 
    \vPhi_0(x_1) 
    + 
    L(\vPhi_0(x_2)) 
    + 
    R_{z1}(\vPhi_0(x_2)) 
    - 
    3\vPhi_0(x_2), 
  \end{aligned}
\right.
\]
where 
$R_{xy}(x,y,z):=(y,x,z)$, 
$L(x,y,z):=(1-z,1-y,1-x)$, 
which is the reflection over the line 
through $(1/2,1/2,1/2)$ and $(0,1/2,1)$, 
and $R_{z1}(x,y,z):=(x,y,2-z)$. 
Its figure is shown in the left-side of 
{\upshape Figure~\ref{Figure(mackay_crys)}}. \par
A $3$-valent discrete minimal surface of the Mackay crystal 
is constructed by deforming the standard realization 
$\vPhi_0$  by
solving {\upshape(\ref{Eq.prescribed_H})} 
on the smallest patch $\{x_0,x_1,x_2\}$ 
{\upshape(}see {\upshape Figure~\ref{Figure(hex_dom_of_P)}}{\upshape)} 
so as to have symmetry with respect to 
folding along the boundary of the green-hued domain. 
The system is as follows: 
\begin{equation}
  \left\{
    \begin{aligned}
      0 
      & = 
      \proj_{x_0}(\u{n}_{50}-\u{n}_{10}) 
      \times 
      (\vPhi(x_1)-\vPhi(x_0)) 
      \\
      & \qquad 
      + \proj_{x_0}(\u{n}_{10}-\u{n}_{1}) 
      \times 
      (R_{xy}(L(\vPhi(x_0)))-\vPhi(x_0)) 
      \\
      & \qquad 
      + \proj_{x_0}(\u{n}_{1}-\u{n}_{50}) 
      \times 
      (L(\vPhi(x_0))-\vPhi(x_0)), 
      \\
      0 
      & = 
      \proj_{x_1}(\u{n}_{2}-\u{n}_{52}) 
      \times 
      (\vPhi(x_0)-\vPhi(x_1)) 
      \\
      & \qquad 
      + \proj_{x_1}(\u{n}_{52}-\u{n}_{0}) 
      \times 
      (\vPhi(x_2)-\vPhi(x_1)) 
      \\
      & \qquad 
      + \proj_{x_1}(\u{n}_{0}-\u{n}_{2}) 
      \times 
      (R_{xy}(\vPhi(x_2))-\vPhi(x_1)), 
      \\
      0 
      & = 
      \proj_{x_2}(\u{n}_{3}-\u{n}'_{2}) 
      \times 
      (\vPhi(x_1)-\vPhi(x_2)) 
      \\
      & \qquad 
      + \proj_{x_2}(\u{n}'_{2}-\u{n}_{1}) 
      \times 
      (L(\vPhi(x_2))-\vPhi(x_2)) 
      \\
      & \qquad 
      + \proj_{x_2}(\u{n}_{1}-\u{n}_{3}) 
      \times 
      (R_{z1}(\vPhi(x_2))-\vPhi(x_2)), 
      \\
      0 
      & = 
      \vPhi(x_0)-R_{xy}(\vPhi(x_0)), 
      \\
      0 
      & = 
      \vPhi(x_1)-R_{xy}(\vPhi(x_1)), 
    \end{aligned}
  \right.
  \label{Eq.minimal_mackay}
\end{equation}
where $\u{n}_i$ stands for the unit normal vector 
of $\vPhi_0$ at $x_i$, 
$\u{n}_2'$ for the one at $R_{z1}(x_2)$, 
and $\proj_{x_i}$ for the orthogonal projection 
onto the tangent plane $T_{x_i}$ of $\vPhi_0$. 
The actual coordinates are given as follows:
\begin{align*}
  \vPhi(x_0) 
  & = 
    \frac{1}{18190160132}
    \begin{pmatrix}
      7635077341+4959792\sqrt{187} 
      \\
      7635077341+4959792\sqrt{187} 
      \\
      12286671541-10842192\sqrt{187}
    \end{pmatrix}, 
  \\
  \vPhi(x_1) 
  & = 
    \frac{1}{18190160132}
    \begin{pmatrix}
      6537796891+8687376\sqrt{187} 
      \\
      6537796891+8687376\sqrt{187} 
      \\
      15629549191-22198320\sqrt{187}
    \end{pmatrix}, 
  \\
  \vPhi(x_2) 
  & = 
    \frac{1}{18190160132}
    \begin{pmatrix}
      3663967141+18450096\sqrt{187} 
      \\
      8262094741+2829744\sqrt{187} 
      \\
      17302810741-27882576\sqrt{187} 
      \\
    \end{pmatrix}. 
\end{align*}
By symmetry of $\vPhi_0$, 
the first three equations in 
{\upshape (\ref{Eq.minimal_mackay})} 
actually gives $7$ linearly independent equations, 
which is less than the number of unknown variables 
by $2$. 
The last two equations in 
{\upshape (\ref{Eq.minimal_mackay})} 
is needed for $\vPhi$ to have 
the same symmetry as the Mackay crystal. 
The figure of the minimal discrete surface $\vPhi$ is shown in 
the right-side of {\upshape Figure~\ref{Figure(mackay_crys)}}. 
\begin{figure}[htbp]
  \begin{minipage}{.5\textwidth}
    \centering
    \includegraphics[bb=0 0 350 350, scale=0.5]
    {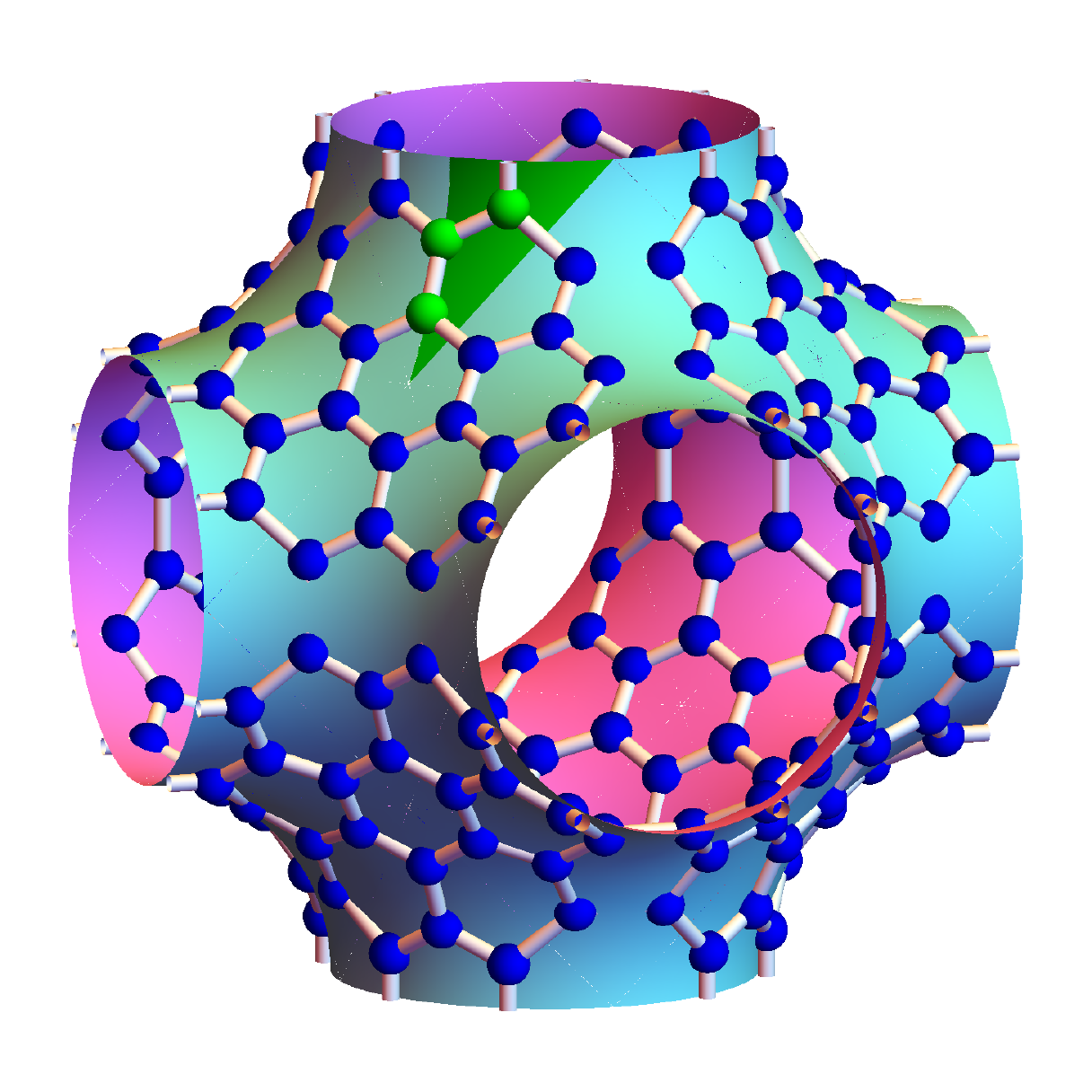}
  \end{minipage}
  \begin{minipage}{.49\textwidth}
    \centering
    \includegraphics[bb=0 0 350 350, scale=0.5]
    {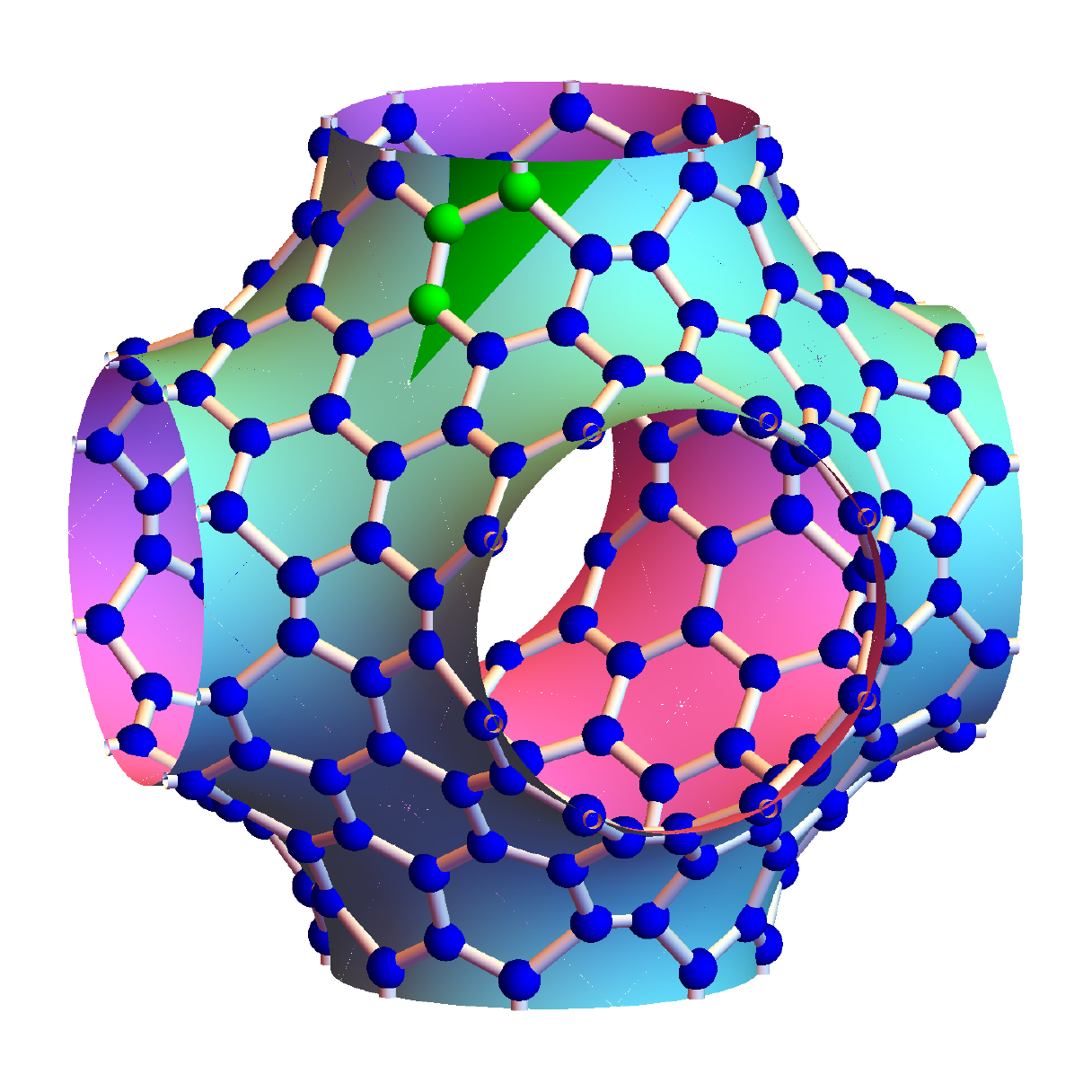}
  \end{minipage}
  \caption{
    Cyan-hued surfaces are both 
    Schwarzian surface of type $P$.
    Left: The classical Mackay crystal which is 
    standardly realized. 
    Right: A $3$-valent minimal discrete discrete surface which has 
    same maximal symmetry as the standardly realized 
    Mackay crystal. }
  \label{Figure(mackay_crys)}
\end{figure}
%
\subsection{$K_4$ lattice}
\label{subsec:k4}
The $K_4$ lattice is the maximal abelian cover of 
the complete 3-valent graph (the $K_4$ graph)
(c.f.\ \cite{MR1783793,MR1500145,MR2375022}). 
The standard realization of the $K_4$ lattice 
has edges with uniform length, so the mean curvature of it 
vanishes at every vertex 
(by Proposition~\ref{Thm(min_harm)}) 
while its Gauss curvature is positive 
at every vertex. 
Thus the eigenvalues of 
the Weingarten-type map 
(the principal curvatures) 
are not real numbers. 
%
\section{Goldberg-Coxeter construction}\label{Section(GC)}
For a given $3$-valent discrete surface, 
we would like to construct 
a sequence of its subdivisions 
which gets finer and finer, 
and converges to a smooth surface. 
An idea for this is the Goldberg-Coxeter construction, 
which can be applied to $3$-valent 
abstract planer graph to increase only $6$-membered rings. 
The Goldberg-Coxeter construction is a generalization of simplicial subdivision considered by 
Goldberg \cite{1937104}, 
and originally discussed by M.\ Deza and M.\ Dutour \cite{MR2429120,MR2035314}.
Here, we recall the definition of it for readers' convenience, 
and calculate subdivisions of the regular hexagonal lattice.
\begin{Def}[\!{\cite[Section~2.1]{MR2429120}}]
  Let $X=(V,E)$ be a $3$-valent planer graph 
  and $k>0$, $\ell\geq 0$ be integers. 
  The graph 
  $\GC_{k,\ell}(X)$ is built in the following steps. 
  \begin{enumerate}
  \item[{\upshape (1)}] 
    Take the dual graph $X^{\ast}$ of $X$. 
    Since $X$ is $3$-valent, $X^{\ast}$ is 
    a triangulation, namely, 
    a planer graph whose faces are all triangles. 
  \item[{\upshape (2)}] 
    Every triangle in $X^{\ast}$ is subdivided 
    into another set of faces in accordance with 
    {\upshape Figure~\ref{Figure(GC_constr)}}. 
    If we obtain a face which are not triangle, 
    then it can be glued with other neighboring 
    non-triangle faces to form triangles. 
  \item[{\upshape (3)}] 
    By duality, the triangulation of {\upshape (2)} 
    is transformed into $\GC_{k,\ell}(X)$. 
  \end{enumerate}
\end{Def}
\begin{figure}[H]
  \centering
  \includegraphics[scale=0.65]
  {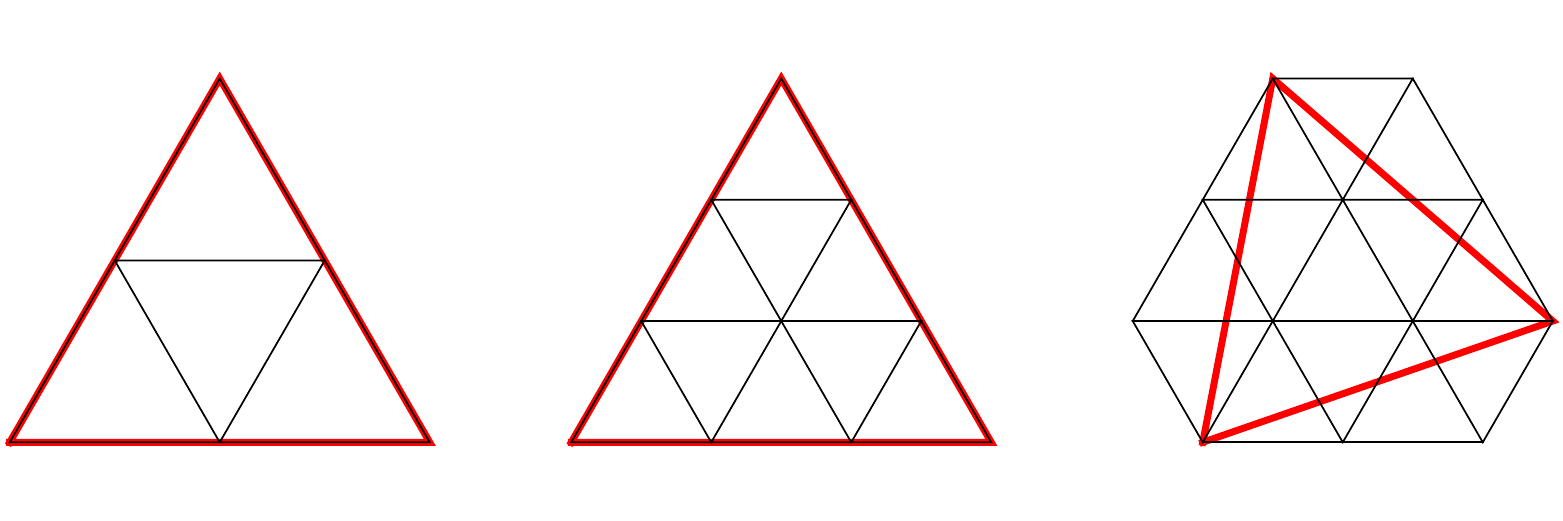} 
  \\[-10pt]
  $(k,\ell)=(2,0)$ \hspace{1cm}
  $(k,\ell)=(3,0)$ \hspace{1cm}
  $(k,\ell)=(2,1)$ 
  \caption{The red triangles are those of $X^{\ast}$}
  \label{Figure(GC_constr)}
\end{figure}
\begin{Ex}
  \label{ex:00}
  The following figures (Figure \ref{fig:gc0}) give the steps of  
  the construction for $\GC_{2,0}(X)$ of a regular 
  hexagonal lattice $X$ (i.e.\ $X(\lambda)$ in Section \ref{Section(swnt)}). 
\end{Ex}
\begin{figure}[H]
  \centering
  \begin{tabular}{ccccc}
    {\upshape (a)} 
    & {\upshape (b)} 
    & {\upshape (c)} 
    & {\upshape (d)} 
    & {\upshape (e)} 
    \\
    \includegraphics[bb=0 0 110 110,scale=0.65]
    {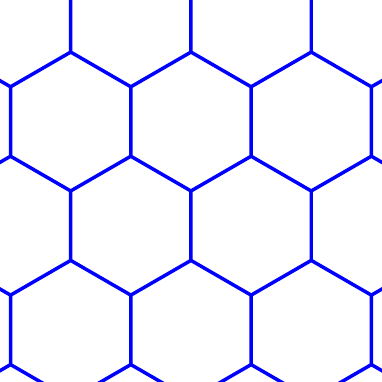}
    & 
      \includegraphics[bb=0 0 110 110,scale=0.65]
      {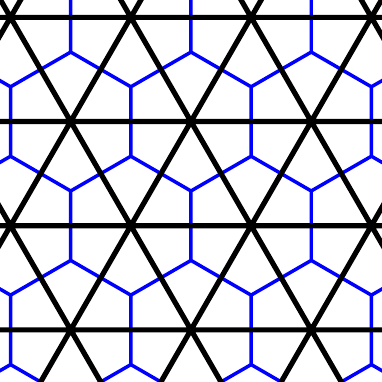}
    & 
      \includegraphics[bb=0 0 110 110,scale=0.65]
      {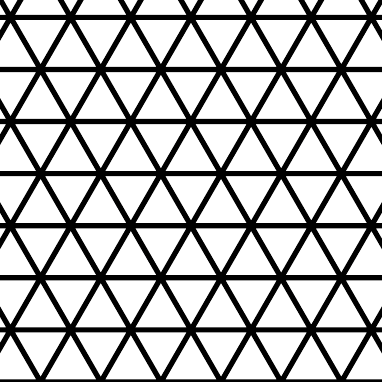}
    & 
      \includegraphics[bb=0 0 110 110,scale=0.65]
      {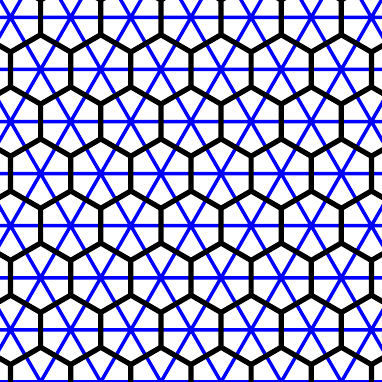}
    & 
      \includegraphics[bb=0 0 110 110,scale=0.65]
      {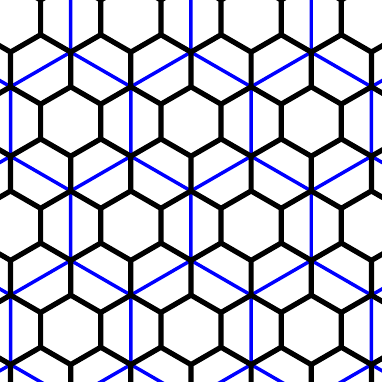}
  \end{tabular}
  \caption{
    {\upshape (a)}~$X$, 
    {\upshape (b)}~$X$ and its dual $X^{\ast}$, 
    {\upshape (c)}~$X^{\ast}$ and its subdivision, 
    {\upshape (d)}~the subdivision and 
    its dual $\GC_{2,0}(X)$, 
    {\upshape (e)}~$X$ and $\GC_{2,0}(X)$. 
  }
  \label{fig:gc0}
\end{figure}
\begin{Ex}
  Here are the figures of $\GC_{k,\ell}(X)$ 
  of a regular hexagonal lattice $X$ 
  for several $(k,\ell)$. 
\end{Ex}
\begin{figure}[H]
  \centering
  \begin{tabular}{ccccc}
    \includegraphics[bb=0 0 110 110,scale=0.65]
    {figures/gc_hex/gc_2_0_3.pdf}
    & 
      \includegraphics[bb=0 0 110 110,scale=0.65]
      {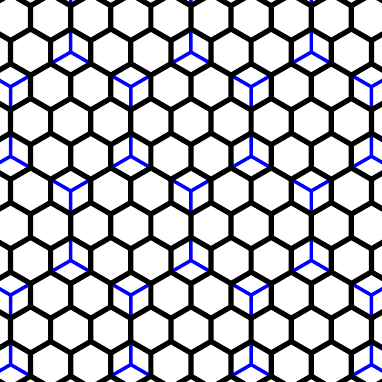}
    & 
      \includegraphics[bb=0 0 110 110,scale=0.65]
      {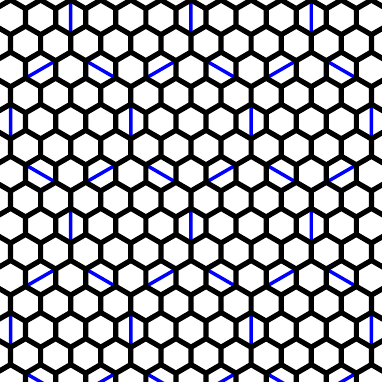}
    & 
      \includegraphics[bb=0 0 110 110,scale=0.65]
      {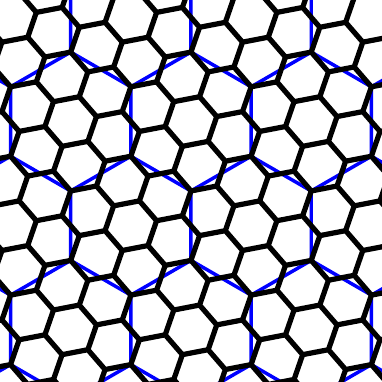}
    & 
      \includegraphics[bb=0 0 110 110,scale=0.65]
      {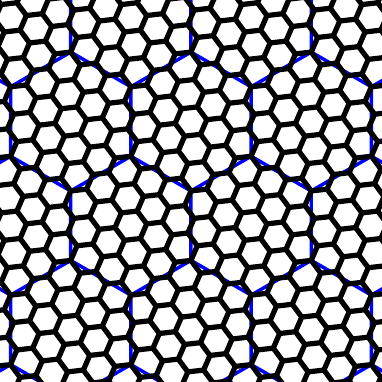} 
    \\
    $\GC_{2,0}(X)$ 
    & 
      $\GC_{3,0}(X)$ 
    & 
      $\GC_{4,0}(X)$ 
    & 
      $\GC_{2,1}(X)$ 
    & 
      $\GC_{3,2}(X)$
  \end{tabular}
  \caption{In each figure, blue graph is original $X$.}
  \label{fig:gc1}
\end{figure}
A basic result on the iterating subdivisions 
is stated as follows. 
\begin{Thm}[{\!\cite[Theorem~2.1.1]{MR2429120}}]
  For any $3$-valent planer graph $X$, it follows 
  \[
  \GC_{z_2}(\GC_{z_1}(X)) 
  = 
  \GC_{z_2z_1}(X), 
  \]
  where $\GC_z(X):=\GC_{k,\ell}(X)$ 
  for $z=k+\ell\omega$, $\omega=(1+\sqrt{3}i)/2$. 
\end{Thm}
The rest of this section 
is devoted to the computation 
of the actual coordinates of 
$\GC_{k,\ell}(X)$ of a hexagonal lattice $X$. 
To this end, we fix the following notation: \par
Let $\rho_{\theta}$ be the 
counterclockwise rotation in $\mathbb{R}^2$ 
of angle $\theta$ around the origin. 
For given $\u{u},\u{\xi}\in \mathbb{R}^2$, 
set $\u{\xi}_1:=\u{\xi}$, 
$\u{\xi}_2:=\rho_{2\pi/3}\u{\xi}$ and 
$\u{\xi}_3:=\rho_{-2\pi/3}\u{\xi}$. 
Recall the regular hexagonal lattice
$
H(\u{u},\u{\xi})
$
as in Section \ref{Section(swnt)}, 
and 
let 
$
T(\u{u},\u{a})
$
be the triangular lattice
which is also the planer graph 
with vertices $\{\u{v},\u{a}_1+\u{v},\u{a}_2+\u{v}\}$ 
and with (unoriented) edges 
$\{(\u{v},\u{v}+\u{a}_1),
(\u{v},\u{v}+\u{a}_2),
(\u{v}+\u{a}_1,\u{v}+\u{a}_2)\}$ 
extended by translations via 
$\u{a}_1$ and $\u{a}_2$. 
\begin{Lem}\label{Lem(dual&dual)}
  \begin{enumerate}
  \item[{\upshape (1)}] 
    The dual lattice $X^{\ast}$ 
    of $X=H(\u{u},\u{\xi})$ is given as 
    $X^{\ast}=T(\u{v},\u{a})$, where, for example, 
    \begin{align*}
      \u{v} 
      & = 
        \u{u}-\u{\xi}_3 
        = 
        \u{u}-\rho_{-2\pi/3}\u{\xi}, 
      \\
      \u{a} 
      & = 
        \rho_{\pi/3}(\u{\xi}_2-\u{\xi}_1) 
        = 
        -\u{\xi}-\rho_{\pi/3}\u{\xi} 
        = 
        -\sqrt{3}\rho_{\pi/6}\u{\xi}. 
    \end{align*}
  \item[{\upshape (2)}] 
    The dual lattice $X^{\ast}$ of $X=T(\u{v},\u{a})$ 
    is given as 
    $X^{\ast}=H(\u{u},\u{\xi})$, where, for example, 
    \begin{align*}
      \u{u} 
      & = 
        \u{v}+\frac{1}{3}(\u{a}+\rho_{\pi/3}\u{a}), 
      \\
      \u{\xi} 
      & = 
        \frac{1}{3}(\rho_{\pi/3}\u{a}-\u{a}) 
        = 
        (\rho_{-2\pi/3}-\mathrm{Id})^{-1}\u{a} 
        = 
        -\frac{1}{\sqrt{3}}\rho_{-\pi/6}\u{a}. 
    \end{align*}
  \end{enumerate}
\end{Lem}
\begin{proof}
  The proof is immediate by noting that 
  the respective base vertices $\u{v}$ and $\u{u}$ 
  of the dual lattices $X^{\ast}$ are barycenters of 
  a hexagon and a triangle of $X$, respectively. 
\end{proof}
\begin{Prop}\label{Prop(GC_of_hex)}
  Let $X=H(\u{u},\u{\xi})$ be a hexagonal lattice, 
  $k>0$ and $l\geq 0$ be integers. 
  Then $\GC_{k,\ell}(X)$ is the hexagonal lattice 
  $H(\u{w},\u{\zeta})$ with 
  \begin{align*}
    \u{w} 
    & = 
      \u{u}-\rho_{-2\pi/3}\u{\xi} 
      - 
      \frac{1}{3(k^2+k\ell+\ell^2)}
      \left\{
      (k+2\ell)\u{\xi} 
      + 
      (2k+\ell)\rho_{\pi/3}\u{\xi} 
      + 
      (k-\ell)\rho_{2\pi/3}\u{\xi}
      \right\}, 
    \\
    \u{\zeta} 
    & = 
      \frac{1}{3(k^2+k\ell+\ell^2)}
      \left\{
      (2k+\ell)\u{\xi} 
      + 
      (k-\ell)\rho_{\pi/3}\u{\xi} 
      - 
      (k+2\ell)\rho_{2\pi/3}\u{\xi}
      \right\}
  \end{align*}
\end{Prop}
\begin{proof}
  The dual graph $X^{\ast}=T(\u{v},\u{a})$ 
  as is given in Lemma~\ref{Lem(dual&dual)}. 
  The $(k,\ell)$-subdivision $(X^{\ast})_{k,\ell}$ 
  of $X^{\ast}$ is by definition $T(\u{v},\u{b})$ 
  for some $\u{b}\in \mathbb{R}^2$. 
  Since the set of vertices of $(X^{\ast})_{k,\ell}$ is 
  given as 
  $\{\u{v}+\beta_1\u{b}_1+\beta_2\u{b}_2\mid 
  (\beta_1,\beta_2)\in \mathbb{Z}\times\mathbb{Z}\}$, 
  where $\u{b}_1=\u{b}$ and $\u{b}_2=\rho_{\pi/3}\u{b}$, 
  the definition of subdivision implies 
  \begin{align*}
    \u{a}_1 
    & = 
      k\u{b}_1+\ell\u{b}_2, 
    \\
    \u{a}_2 
    & = 
      \rho_{\pi/3}\u{a} 
      = 
      k\rho_{\pi/3}\u{b}_1 
      + 
      \ell\rho_{\pi/3}\u{b}_2 
      = 
      -\ell\u{b}_1 + (k+\ell)\u{b}_2, 
  \end{align*}
  so that 
  \begin{equation}
    \begin{aligned}
      \u{b}_1 
      & = 
      \frac{1}{k^2+k\ell+\ell^2}
      \left(
	(k+\ell)\u{a}_1-\ell\u{a}_2
      \right), 
      \\
      \u{b}_2 
      & = 
      \frac{1}{k^2+k\ell+\ell^2}
      \left(
	\ell\u{a}_1+k\u{a}_2
      \right). 
    \end{aligned}
    \label{Eq.b1b2}
  \end{equation}
  So far we have 
  \[
  (X^{\ast})_{k,\ell} 
  = 
  T(\u{v},\u{b}) 
  = 
  T\left(
    \u{u}-\rho_{-2\pi/3}\u{\xi},~ 
    \frac{1}{k^2+k\ell+\ell^2}
    \left(
      (k+\ell)\u{a}_1-\ell\u{a}_2
    \right)
  \right). 
  \]
  The dual of $(X^{\ast})_{k,\ell}=T(\u{v},\u{b})$, 
  say $H(\u{w},\u{\zeta})$, is $\GC_{k,\ell}(X)$ 
  and we already know from Lemma~\ref{Lem(dual&dual)} 
  how to find $(\u{w},\u{\zeta})$ 
  from $(\u{v},\u{b})$. 
\end{proof}
\begin{Ex}
  Table \ref{tab:1} is a table for $X=H((0,0),(-\sqrt{3}/2,1/2))$. 
  \begin{table}[H]
    \centering
    \caption{}
    \label{tab:1}
    {\renewcommand\arraystretch{1.3}
      \begin{tabular}{c|c|c|c}
        $(k, \ell)$ & $\u{w}$ 
        & $\u{\zeta}$ 
        & 
          $\lvert \u{\zeta}\rvert/\lvert\u{\xi}\rvert$
          \cr \hline
          $(2, 0)$ 
        & $(0, 1/2)$
        & $(-\sqrt{3}/4, 1/4)$
        & $1/2$
          \cr \hline
          $(3, 0)$ 
        & $(0, 2/3)$
        & $(-\sqrt{3}/6, 1/6)$
        & $1/3$
          \cr \hline
          $(4, 0)$ 
        & $(0, 3/4)$
        & $(-\sqrt{3}/8, 1/8)$
        & $1/4$
          \cr \hline
          $(2, 1)$ 
        & $(-\sqrt{3}/14, 9/14)$
        & $(-\sqrt{3}/7, 2/7)$
        & $ 1/\sqrt{7}$
          \cr \hline
          $(3, 1)$ 
        & $(-\sqrt{3}/26, 19/26)$
        & $(-3\sqrt{3}/26, 5/26)$
        & $ 1/\sqrt{13}$
          \cr 
      \end{tabular}
    }
  \end{table}
\end{Ex}
\section{Convergence of the GC-subdivisions}
\label{sec:subdivision}
This section provides several examples of 
the convergence of GC-subdivisions. 
We start with a simple observation 
on the convergence of general sequence of 
$3$-valent discrete surfaces. 
\par
GC-subdivisions of a dodecahedron, a hexahedron, or a tetrahedron are called 
\emph{Goldberg polyhedra}.
In particular, 
fullerene $C_{60}$ (a truncated octahedron) is $\GC_{1,1}$ of a dodecahedron
(see \cite{Hart}).
\subsection{Convergence theorem}
\label{subsec:convergence}
Here, a \emph{GC-subdivision} of discrete surfaces is 
a discrete surface, 
which is a GC-construction as an abstract graph and 
is embedded as suitable way.
\begin{Prop}
  Let 
  $\{
  \vPhi_k\colon X_k=(V_k,E_k)\rightarrow \mathbb{R}^3
  \}_{k=1}^{\infty}$ 
  be a sequence of $3$-valent discrete surfaces 
  with the following properties. 
  \begin{enumerate}
  \item[{\upshape (\rnum{1})}] 
    The sequence of sets of points 
    $\{\vPhi_k(V_k)\}_{k=1}^{\infty}$ converges to 
    a smooth surface $M$ in $\mathbb{R}^3$ 
    in the Hausdorff topology. 
  \item[{\upshape (\rnum{2})}] 
    For any $p\in M$, 
    the unit normal vector $\u{n}_k(x_k)$ 
    of $\vPhi_k$ at $x_k\in V_k$ 
    converges to the unit normal vector $\u{n}(p)$ of $M$ at $p$, 
    independently of the choice of $\{x_k\}_{k=1}^{\infty}$ 
    with $\vPhi_k(x_k)\rightarrow p$ as $k\rightarrow \infty$. 
  \item[{\upshape (\rnum{3})}] 
    The Weingarten map 
    $S_k\colon T_{x_k}\rightarrow T_{x_k}$ 
    of $\vPhi_k$ converges to 
    the Weingarten map 
    $S\colon T_pM\rightarrow T_pM$ 
    of $M$ in the following sense: 
    for $\{x_k\}_{k=1}^{\infty}$ 
    with $\vPhi_k(x_k)\rightarrow p$ as $k\rightarrow \infty$ 
    and for 
    $\{\u{v}_k\in T_{x_k}\}_{k=1}^{\infty}$ 
    converging to some $\u{v}\in T_pM$, 
    it follows 
    \[
    S_k(\u{v}_k)\rightarrow S(\u{v})
    \]
    in $\mathbb{R}^3$ as $k\rightarrow \infty$. 
  \end{enumerate}
  Then both the mean curvature $H_k(x_k)$ 
  and the Gauss curvature $G_k(x_k)$ of $\vPhi_k$ 
  respectively converge to the mean curvature $H(p)$ 
  and the Gauss curvature $G(p)$ of $M$ 
  for $\{x_k\}_{k=1}^{\infty}$ with 
  $\vPhi_k(x_k)\rightarrow p$ as $k\rightarrow \infty$. 
\end{Prop}
\begin{proof}
  Let $p\in M$ be a point and $\{x_k\}_{k=1}^{\infty}$ 
  be a sequence of points $x_k\in V_k$ 
  such that $\vPhi_k(x_k)$ converges to 
  $p$ in $\mathbb{R}^3$. 
  For any tangent vector $\u{v}\in T_pM$, 
  as is easily seen using (\rnum{2}), 
  it follows that the sequence $\{\u{v}_k\}_{k=1}^{\infty}$, 
  where $\u{v}_k$ is the orthogonal projection 
  of $\u{v}$ onto $T_{x_k}$, 
  converges to $\u{v}$. 
  If we take a pair of linearly independent vectors 
  $\{\u{v},\u{w}\}\subseteq T_pM$ 
  so that $\u{v}\times\u{w}$ has the same direction 
  as $\u{n}(p)$, then, 
  the vectors $\{\u{v}_k,\u{w}_k\}\subseteq T_{x_k}$ 
  which are respectively obtained from 
  $\{\u{v},\u{w}\}\subseteq T_pM$ 
  as in the above manner are also linearly independent 
  as well as $\u{v}_k\times\u{w}_k$ has 
  the same direction as $\u{n}_k(x_k)$ 
  for sufficiently large $k\in \mathbb{N}$. 
  Then, by (\rnum{3}), 
  \[
  \begin{pmatrix}
    \inner{\u{v}_k}{\u{v}_k} 
    & 
    \inner{\u{v}_k}{\u{w}_k} 
    \\
    \inner{\u{w}_k}{\u{v}_k} 
    & 
    \inner{\u{w}_k}{\u{w}_k}
  \end{pmatrix}^{-1}
  \begin{pmatrix}
    \inner{\u{v}_k}{S_k(\u{v}_k)}
    & 
    \inner{\u{v}_k}{S_k(\u{w}_k)} 
    \\
    \inner{\u{w}_k}{S_k(\u{v}_k)}
    & 
    \inner{\u{w}_k}{S_k(\u{w}_k)} 
  \end{pmatrix}, 
  \]
  whose trace is equal to $H_k(x_k)$ 
  (resp.\ determinant is equal to 
  $G_k(x_k)$), 
  converges, as $k\to\infty$, to 
  \[
  \begin{pmatrix}
    \inner{\u{v}}{\u{v}} 
    & 
    \inner{\u{v}}{\u{w}} 
    \\
    \inner{\u{w}}{\u{v}} 
    & 
    \inner{\u{w}}{\u{w}}
  \end{pmatrix}^{-1}
  \begin{pmatrix}
    \inner{\u{v}}{S(\u{v})}
    & 
    \inner{\u{v}}{S(\u{w})} 
    \\
    \inner{\u{w}}{S(\u{v})}
    & 
    \inner{\u{w}}{S(\u{w})} 
  \end{pmatrix}, 
  \]
  whose trace is equal to $H(p)$ 
  (resp.\ determinant is equal to 
  $G(p)$). 
\end{proof}
The following examples shows that 
the condition of the preceding proposition is optimal 
in the most general settings. 
\begin{Ex}
  Let $X_k$ be the regular hexagonal lattice 
  in the plane with exception at 
  a vertex, say, $(0,0)$, which is located 
  at $(0,0,h_k)$, where $h_k>0$. 
  If the distance of adjacent vertices becomes 
  small with order $1/k$, then 
  \begin{enumerate}
  \item[{\upshape (\rnum{1})}] 
    $X_k$ does not converge to the plane 
    in the Hausdorff sense 
    unless $h_k$ converges to $0$ 
    as $k\rightarrow \infty$. 
  \item[{\upshape (\rnum{2})}] 
    The normal vector does not converge 
    provided $kh_k$ is bounded away from $0$ 
    as $k\rightarrow \infty$. 
  \item[{\upshape (\rnum{3})}] 
    The Weingarten map does not converge 
    provided $k^2h_k$ is bounded away from $0$ 
    as $k\rightarrow \infty$
  \end{enumerate}
\end{Ex}
\subsection{Convergence of carbon nanotubes}
\label{subsec:convergence:CNT}
%
Here we consider a sequence of subdivisions 
of a carbon nanotube $\cnt(\lambda,c)$ 
via Goldberg-Coxeter construction.
Namely, 
GC-subdivisions of a carbon nanotube 
are GC-construction of the regular hexagonal lattice
and then are rolled up to a tube with suitable radius.
For convenience, 
we put no assumptions on 
the index $c=(c_1,c_2)$ 
other than $c\neq 0$. 
Even then (unless $c_1>0$ and $c_2\geq 0$), 
$\cnt(\lambda,c)$ is well-defined 
in the exactly same manner, 
although two carbon nanotubes with different 
indexes may have just the same structure. 
%
\begin{Prop}
  $\GC_{k,\ell}(X(\lambda))=H(\u{w},\u{\zeta})$ 
  of $X(\lambda)=H(\u{0},\u{\xi})$ 
  as in {\upshape (\ref{Eq.Xlambda})} satisfies 
  \begin{equation}
    \lvert \u{\zeta} \rvert 
    = 
    \frac{
      1
    }{
      \sqrt{k^2+k\ell+\ell^2}
    }
    \lvert \u{\xi} \rvert, \quad 
    \frac{
      \inner{\u{\zeta}}{\u{\xi}}
    }{
      \lvert \u{\zeta} \rvert 
      \lvert \u{\xi} \rvert
    } 
    = 
    \frac{
      2k+\ell
    }{
      2\sqrt{k^2+k\ell+\ell^2}
    }. 
    \label{Eq.|z|&<z,x>}
  \end{equation}
  In particular, 
  the angle between $\u{\zeta}$ and $\u{\xi}$ 
  is same as that between the chiral vector 
  $\u{c}=k\u{a}_1+\ell\u{a}_2$ in $X(\lambda)$ 
  and $\u{e}_1=(1,0)^T$. 
\end{Prop}
\begin{proof}
  The expressions (\ref{Eq.|z|&<z,x>}) 
  are consequences of straightforward computation 
  using Proposition~\ref{Prop(GC_of_hex)}. 
  Recall that the chiral vector 
  $\u{c}=k\u{a}_1+\ell\u{a}_2$ in $X(\lambda)$ 
  is given as 
  \[
  \u{c} 
  = 
  \frac{\sqrt{3}}{2}\lambda
  \begin{pmatrix}
    2k+\ell \\
    \sqrt{3}\ell
  \end{pmatrix}. 
  \]
  Since $\lvert \u{c}\rvert=\lambda\sqrt{3(k^2+k\ell+\ell^2)}$, 
  a simple computation shows that 
  $\inner{\u{c}}{\u{e}_1}/(\lvert\u{c}\rvert\lvert\u{e}_1\rvert)$ 
  is exactly same as the latter of (\ref{Eq.|z|&<z,x>}). 
\end{proof}
Not all the vertices of $X(\lambda)=H(\u{0},\u{\xi})$ 
are vertices of a subdivision $\GC_{k,\ell}(X(\lambda))$, 
while, as is stated in 
Proposition~\ref{Prop(lattice_inv_under_subdiv)} 
below, in either case, 
the chiral vector is, as position vector, 
always a vertex of $\GC_{k,\ell}(X(\lambda))$. 
This observation suggests that  
any $\GC_{k,\ell}(X(\lambda))$ 
is nothing less than a subdivision of $X(\lambda)$ 
in consideration of $\cnt(\lambda,c)$. 
\begin{Prop}\label{Prop(lattice_inv_under_subdiv)}
  The lattice vectors $\u{a}_1$, $\u{a}_2$ of 
  $X(\lambda)$ belong to 
  the set of vertices of $\GC_{k,\ell}(X(\lambda))$ 
  for any $k>0$, $\ell\geq 0$. 
  In particular, 
  the chiral vector $\u{c}=c_1\u{a}_1+c_2\u{a}_2$ 
  is also a vertex of $\GC_{k,\ell}(X(\lambda))$ 
  for any $k>0$, $\ell\geq 0$, 
  and the chiral index with respect to 
  the lattice vectors of $\GC_{k,\ell}(X(\lambda))$ is given as 
  $(kc_1-\ell c_2, \ell c_1+(k+\ell)c_2)$. 
\end{Prop}
\begin{proof}
  Since the lattice vectors of $\GC_{k,\ell}(X(\lambda))$ 
  are just $\u{b}_1$, $\u{b}_2$ of (\ref{Eq.b1b2}) 
  and $\u{a}_1=k\u{b}_1+\ell\u{b}_2$, 
  $\u{a}_2=-\ell\u{b}_1+(k+\ell)\u{b}_2$, 
  the former assertion follows. 
  \[
  \u{c} 
  = 
  c_1\u{a}_1+c_2\u{a}_2 
  = 
  (kc_1-\ell c_2)\u{b}_1 
  + 
  (\ell c_1+(k+\ell)c_2)\u{b}_2
  \]
  proves the latter assertion. 
\end{proof}
\begin{Def}
  Let $\cnt(\lambda,c)$ be the carbon nanotube 
  with chiral index $c=(c_1,c_2)$. 
  Then $\cnt(\mu,d)$ is said to be 
  a \emph{$(k,\ell)$-subdivision}
  of $\cnt(\lambda,c)$ if 
  there exist $k>0$ and $\ell\geq 0$ such that 
  \begin{align*}
    \mu 
    & = 
      \frac{\lambda}{\sqrt{k^2+k\ell+\ell^2}}, 
    \\
    d 
    & = 
      \left(
      kc_1-\ell c_2, 
      \ell c_1 + (k+\ell)c_2
      \right). 
  \end{align*}
  If this is the case, we write as 
  \[
  \cnt(\mu,d) 
  = 
  \GC_{k,\ell}(\cnt(\lambda,c)). 
  \]
\end{Def}
\begin{Ex}
  Here are simple examples of $(k,\ell)$-subdivisions. 
  \begin{align*}
    & 
      \GC_{1,0}(\cnt(\lambda,c)) 
      = 
      \cnt(\lambda,c), 
    \\
    & 
      \GC_{k,0}(\cnt(\lambda,(c_1,c_2))) 
      = 
      \cnt
      \left(
      \frac{\lambda}{k},(kc_1,kc_2)
      \right), 
    \\
    & 
      \GC_{k,k}(\cnt(\lambda,(c_1,0))) 
      = 
      \cnt
      \left(
      \frac{\lambda}{\sqrt{3}k},(kc_1,kc_1)
      \right), 
    \\
    & 
      \GC_{k,k}(\cnt(\lambda,(c_1,c_1))) 
      = 
      \cnt\left(\frac{\lambda}{\sqrt{3}k},(0, 3kc_1)\right). 
  \end{align*}
\end{Ex}
%
%
\begin{Thm}\label{Thm(conv_of_cnt)}
  Let $\{\cnt(\lambda^{(n)},c^{(n)})\}_{n=1}^{\infty}$ 
  be a sequence of strictly monotone subdivisions 
  $\cnt(\lambda,c) = \cnt(\lambda^{(1)},c^{(1)})$ 
  in the sense that 
  $\cnt(\lambda^{(n+1)},c^{(n+1)})$ is 
  a $(k_n,\ell_n)$-subdivision of $\cnt(\lambda^{(n)},c^{(n)})$ 
  for some $k_n\geq 2$ and $\ell_n\geq 0$ for each $n\in \mathbb{N}$. 
  Then 
  \begin{equation}
    \lim_{n\to\infty}H^{(n)} 
    = 
    -\frac{1}{2r(\lambda,c)}, 
    \quad 
    \lim_{n\to\infty}K^{(n)} 
    = 
    0, 
    \label{Eq.limH&limK}
  \end{equation}
  where $H^{(n)}=H(\lambda^{(n)},c^{(n)})$ and 
  $K^{(n)}=K(\lambda^{(n)},c^{(n)})$ are, respectively, 
  the mean curvature and the Gauss curvature 
  of $\cnt(\lambda^{(n)},c^{(n)})$ and 
  \[
  r(\lambda,c) 
  = 
  \frac{\lambda\sqrt{3(c_1^2+c_1c_2+c_2^2)}}{2\pi}
  \]
  is the radius of $\cnt(\lambda,c)$. 
\end{Thm}
To prove Theorem~\ref{Thm(conv_of_cnt)}, 
we need the following lemma. 
\begin{Lem}\label{Lem(omegad1+d2)}
  If $\cnt(\mu,d)=\GC_{k,\ell}(\cnt(\lambda,c))$, 
  then 
  \begin{align*}
    \mu 
    & = 
      \frac{\lambda}{\lvert k+\bar{\omega}\ell\rvert}, 
    \\
    \omega d_1+d_2 
    & = (\omega c_1+c_2)(k+\bar{\omega}\ell), 
  \end{align*}
  where $\omega=(1+\sqrt{3}i)/2$, $c=(c_1,c_2)$ 
  and $d=(d_1,d_2)$. 
\end{Lem}
\begin{proof}[Proof of Theorem~\ref{Thm(conv_of_cnt)}]
  Note that 
  any $(k,\ell)$-subdivision $\GC_{k,\ell}(\cnt(\lambda,c))$ 
  of $\cnt(\lambda,c)$ does not change 
  the radius $r=r(\lambda,c)>0$. 
  From the consequence of Theorem~\ref{Thm(H&K_of_cnt)}, 
  by noting that $m_x(c)\geq 0$ for any 
  $c=(c_1,c_2)
  \in \mathbb{Z}\times\mathbb{Z}\setminus\{(0,0)\}$, 
  we further compute 
  \begin{align*}
    H^{(n)} 
    & = 
      -\frac{1}{2r}
      \left\{
      \frac{
      m_x(c^{(n)})
      }{
      (
      m_x(c^{(n)})^2+m_y(c^{(n)})^2
      + (4/3)m_z(c^{(n)})^2
      )^{1/2}
      } 
      \right. 
    \\
    & \qquad \qquad 
      \left.
      - 
      \frac{
      (4/3)m_z(c^{(n)})^2
      m_x(c^{(n)})
      }{
      (
      m_x(c^{(n)})^2+m_y(c^{(n)})^2
      + (4/3)m_z(c^{(n)})^2
      )^{3/2}
      }
      \right\} 
    \\
    & = 
      -\frac{1}{2r}
      \left\{
      \frac{
      1
      }{
      (
      1+m_y(c^{(n)})^2m_x(c^{(n)})^{-2}
      + (4/3)m_z(c^{(n)})^2m_x(c^{(n)})^{-2}
      )^{1/2}
      } 
      \right. 
    \\
    & \qquad \qquad 
      \left.
      + 
      \frac{
      (4/3)m_z(c^{(n)})^2
      m_x(c^{(n)})^{-2}
      }{
      (
      1+m_y(c^{(n)})^2m_x(c^{(n)})^{-2}
      + (4/3)m_z(c^{(n)})^2m_x(c^{(n)})^{-2}
      )^{3/2}
      }
      \right\}, 
  \end{align*}
  and 
  \begin{align*}
    0\leq 
    K^{(n)} 
    & = 
      \frac{
      4m_z(c^{(n)})
      (m_x(c^{(n)})^2+m_y(c^{(n)})^2)
      }{
      3r^2
      (
      m_x(c^{(n)})^2+m_y(c^{(n)})^2
      + (4/3)m_z(c^{(n)})^2
      )^2
      } 
    \\
    & \leq 
      \frac{4}{3r^2}\cdot
      \frac{
      m_z(c^{(n)})^2
      }{
      m_x(c^{(n)})^2+m_y(c^{(n)})^2
      } 
    \\
    & = 
      \frac{4}{3r^2}\cdot
      \frac{
      m_z(c^{(n)})^2m_x(c^{(n)})^{-2}
      }{
      1+m_y(c^{(n)})^2m_x(c^{(n)})^{-2}
      }. 
  \end{align*}
  Hence, to obtain (\ref{Eq.limH&limK}), 
  it suffices to see 
  \begin{equation}
    \lim_{n\to\infty}
    \frac{m_y(c^{(n)})^2}{m_x(c^{(n)})^2} 
    = 
    0, 
    \quad 
    \text{and} 
    \quad 
    \lim_{n\to\infty}
    \frac{m_z(c^{(n)})^2}{m_x(c^{(n)})^2} 
    = 
    0. 
    \label{Eq.lim_of_m/m}
  \end{equation}
  It follows by Lemma~\ref{Lem(omegad1+d2)} 
  that 
  \[
  \omega c^{(n+1)}_1 
  + 
  c^{(n+1)}_2 
  = 
  (
  \omega c^{(n)}_1 
  + 
  c^{(n)}_2
  )
  (k_n+\bar{\omega}\ell_n)
  \]
  for some $k_n\geq 2$ and $\ell_n\geq0$, 
  which implies 
  $\lvert \omega c^{(n)}_1+c^{(n)}_2\rvert
  \geq (\sqrt{2})^{n-1}$.  
  Since $L_0(c^{(n)})=\lvert \omega c^{(n)}_1+c^{(n)}_2\rvert$ 
  is the length of chiral vectors divided by $\lambda^{(n)}$, 
  $C^{(n)}_i$ and $T^{(n)}_i$ ($i=1,2$) 
  corresponding to (\ref{Eq.C1C2T1T2}) 
  are estimated as 
  \begin{gather*}
    \lvert C^{(n)}_i \rvert 
    = 
    \left\lvert 
      \frac{3\pi c^{(n)}_i}{L_0(c^{(n)})}
    \right\rvert 
    \leq 
    \frac{3\pi}{L_0(c^{(n)})}\quad (i=1,2), 
    \\
    \lvert T^{(n)}_1 \rvert 
    = 
    \left\lvert 
      \frac{3\pi(c^{(n)}_1+2c^{(n)}_2)}{L_0(c^{(n)})^2}
    \right\rvert 
    \leq 
    \frac{9\pi}{L_0(c^{(n)})}, 
    \quad 
    \lvert T^{(n)}_2 \rvert 
    = 
    \left\lvert 
      \frac{3\pi(2c^{(n)}_1+c^{(n)}_2)}{L_0(c^{(n)})^2}
    \right\rvert 
    \leq 
    \frac{9\pi}{L_0(c^{(n)})}, 
  \end{gather*}
  all of which converge to $0$ as $n$ tends to infinity. \par
  We are now ready to prove (\ref{Eq.lim_of_m/m}). 
  If either $C^{(n)}_1=0$ or $C^{(n)}_2=0$ 
  (then $T^{(n)}_1\neq 0$ and $T^{(n)}_2\neq 0$ 
  in either case), 
  then $m_y(c^{(n)})/m_x(c^{(n)})=0$ as well as 
  \[
  \frac{m_z(c^{(n)})}{m_x(c^{(n)})} 
  = 
  \frac{
    \sin(T^{(n)}_2/2)\sin((T^{(n)}_1+T^{(n)}_2)/2)
  }{
    -C^{(n)}_2\cos(C^{(n)}_1/2)
  }, 
  \]
  or 
  \[
  \frac{m_z(c^{(n)})}{m_x(c^{(n)})} 
  = 
  \frac{
    \sin(T^{(n)}_1/2)\sin((T^{(n)}_1+T^{(n)}_2)/2)
  }{
    C^{(n)}_1\cos(C^{(n)}_2/2)
  }, 
  \]
  respectively, 
  whose absolute values are both arbitrarily small 
  for any sufficiently large $n\in \mathbb{N}$. 
  If either $T^{(n)}_1=0$, $T^{(n)}_2=0$ 
  or $T^{(n)}_1+T^{(n)}_2=0$ 
  (then $C^{(n)}_1\neq 0$ and $C^{(n)}_2\neq 0$ 
  in either case), 
  then this time $m_z(c^{(n)})/m_x(c^{(n)})=0$ as well as 
  \[
  \frac{m_y(c^{(n)})}{m_x(c^{(n)})} 
  = 
  \tan\frac{C^{(n)}_2}{2}, 
  \quad 
  = 
  \tan\frac{C^{(n)}_1}{2}, 
  \quad 
  \text{or}
  \quad 
  = 
  -\tan\frac{C^{(n)}_1}{2}, 
  \]
  respectively, 
  whose absolute values are also both arbitrarily small. 
  If neither $C^{(n)}_i$, $T^{(n)}_i$ ($i=1,2$) 
  nor $T^{(n)}_1+T^{(n)}_2$ 
  are zero, 
  \begin{align*}
    \frac{m_y(c^{(n)})}{m_x(c^{(n)})} 
    = 
    \frac{
    -T^{(n)}_2
    \frac{\sin(C^{(n)}_2/2)}{C^{(n)}_2}
    \frac{\sin(T^{(n)}_2/2)}{T^{(n)}_2}
    + 
    +T^{(n)}_1
    \frac{\sin(C^{(n)}_1/2)}{C^{(n)}_1}
    \frac{\sin(T^{(n)}_1/2)}{T^{(n)}_1}
    }{
    \frac{T^{(n)}_2}{C^{(n)}_2}
    \cos\frac{C^{(n)}_2}{2}
    \frac{\sin(T^{(n)}_2/2)}{T^{(n)}_2}
    - 
    \frac{T^{(n)}_1}{C^{(n)}_1}
    \cos\frac{C^{(n)}_1}{2}
    \frac{\sin(T^{(n)}_1/2)}{T^{(n)}_1}
    }, 
  \end{align*}
  as well as 
  \begin{align*}
    \frac{m_z(c^{(n)})}{m_x(c^{(n)})} 
    = 
    \frac{
    \frac{\sin(T^{(n)}_1/2)}{T^{(n)}_1}
    \frac{\sin(T^{(n)}_2/2)}{T^{(n)}_2}
    \frac{
    \sin((T^{(n)}_1+T^{(n)}_1)/2)
    }{
    T^{(n)}_1+T^{(n)}_2
    }
    }{
    \frac{C^{(n)}_1}{T^{(n)}_1+T^{(n)}_2}
    \cos\frac{C^{(n)}_2}{2}
    \frac{\sin(T^{(n)}_2/2)}{T^{(n)}_2}
    - 
    \frac{C^{(n)}_2}{T^{(n)}_1+T^{(n)}_2}
    \cos\frac{C^{(n)}_1}{2}
    \frac{\sin(T^{(n)}_1/2)}{T^{(n)}_1}
    }
  \end{align*}
  whose absolute value again becomes 
  arbitrarily small. 
  We then complete the proof of Theorem~\ref{Thm(conv_of_cnt)}. 
\end{proof}
\subsection{Numerical computations for convergence of the Mackay crystal}
\label{subsec:convergence:Mackay}
We construct GC-subdivisions of the Mackay crystal as follows.
First we construct GC-construction of the abstract graph of 
the fundamental region of the Mackay crystal (See Figure~\ref{Figure(hex_dom_of_P)}), 
and then construct the standard realization of the GC-constructed abstract graph.
By using numerical computations, 
we obtain 
distributions of the curvatures of 
several steps of $\GC_{k,0}$-subdivisions 
of the Mackay crystal are shown in 
Figure~\ref{Figure(K_for_GC-subdiv_of_Mackay)} 
for the Gauss curvature 
and 
Figure~\ref{Figure(H_for_GC-subdiv_of_Mackay)} 
for the mean curvature. 
\par
The sequence of the GC-subdivisions of the Mackay crystal 
seems convergent to the Schwarzian surface of type $P$, 
however, mean curvatures of vertices around octagonal rings may not converges 
(see Figure~\ref{Figure(H_for_GC-subdiv_of_Mackay)}).
Table~\ref{tab:mackay-values} shows min/max values of the mean curvature and the Gauss curvature for each subdivision.
The maximum value of $|H|$ attains on vertices of octahedral rings.
Table~\ref{tab:mackay-length} also shows the min/max values of the length of edges, 
and the maximum value of the length of edges attains on edges of octahedral rings.
Actually the octagonal rings, 
which are considered as topological defects, 
of the Mackay crystals seem to be 
obstructions of the convergence. 
\par\newpage
\begin{table}[H]
  \centering
  \begin{tabular}{r|l|l|l|l|l}
    &
      \multicolumn{1}{c|}{ave. of $H$}
    &
      \multicolumn{1}{c|}{min. of $|H|$}
    &
      \multicolumn{1}{c|}{max. of $|H|$}
    &
      \multicolumn{1}{c|}{min. of $K$}
    &
      \multicolumn{1}{c|}{max. of $K$} \cr
      \hline
      $1$ 
    &
      $-0.000000$ 
    &
      $+0.029880$ 
    &
      $+0.586578$ 
    &
      $-3.771350$ 
    &
      $+0.000000$ \cr
      $2$ 
    &
      $+0.000000$ 
    &
      $+0.000646$ 
    &
      $+0.679771$ 
    &
      $-0.737990$ 
    &
      $+0.000000$ \cr
      $3$ 
    &
      $+0.000000$ 
    &
      $+0.000077$ 
    &
      $+0.727594$ 
    &
      $-0.305908$ 
    &
      $+0.000000$ \cr
      $4$ 
    &
      $+0.000000$ 
    &
      $+0.000018$ 
    &
      $+0.751009$ 
    &
      $-0.167605$ 
    &
      $+0.000000$ \cr
      $5$ 
    &
      $+0.000000$ 
    &
      $+0.000006$ 
    &
      $+0.764183$ 
    &
      $-0.105835$ 
    &
      $+0.000000$ \cr
      $6$ 
    &
      $+0.000000$ 
    &
      $+0.000002$ 
    &
      $+0.772395$ 
    &
      $-0.072912$ 
    &
      $+0.000000$ \cr
      $7$ 
    &
      $+0.000000$ 
    &
      $+0.000001$ 
    &
      $+0.777901$ 
    &
      $-0.053290$ 
    &
      $+0.000000$ \cr
      $8$ 
    &
      $+0.000000$ 
    &
      $+0.000001$ 
    &
      $+0.781799$ 
    &
      $-0.040653$ 
    &
      $+0.000000$ \cr
      $9$ 
    &
      $-0.000000$ 
    &
      $+0.000000$ 
    &
      $+0.784674$ 
    &
      $-0.032036$ 
    &
      $+0.000000$ \cr
      $10$ 
    &
      $+0.000000$ 
    &
      $+0.000000$ 
    &
      $+0.786866$ 
    &
      $-0.025898$ 
    &
      $+0.000000$ \cr
      $12$ 
    &
      $+0.000000$ 
    &
      $+0.000000$ 
    &
      $+0.789952$ 
    &
      $-0.017935$ 
    &
      $+0.000000$ \cr
      $14$ 
    &
      $+0.000000$ 
    &
      $+0.000000$ 
    &
      $+0.791993$ 
    &
      $-0.013153$ 
    &
      $+0.000000$ \cr
      $16$ 
    &
      $+0.000000$ 
    &
      $+0.000000$ 
    &
      $+0.793424$ 
    &
      $-0.010057$ 
    &
      $+0.000000$ \cr
      $18$ 
    &
      $-0.000000$ 
    &
      $+0.000000$ 
    &
      $+0.794472$ 
    &
      $-0.007939$ 
    &
      $+0.000000$ \cr
      $20$ 
    &
      $-0.000000$ 
    &
      $+0.000000$ 
    &
      $+0.795268$ 
    &
      $-0.006426$ 
    &
      $+0.000000$ \cr
  \end{tabular}
  \caption{
    The table of values of the mean curvature and the Gauss curvature, 
    if the translation vector of each discrete surface is 
    $e_x = (1, 0, 0)$, $e_y = (0, 1, 0)$ and $e_z = (0, 0, 1)$.
  }
  \label{tab:mackay-values}
\end{table}
\begin{table}[H]
  \centering
  \begin{tabular}{r|l|l|l}
    & \multicolumn{1}{c|}{min. length}
    & \multicolumn{1}{c|}{max. length}
    & \multicolumn{1}{c}{ratio}
      \cr\hline
      1
    & +0.08861403
    & +0.10810811
    & +1.2200
      \cr
      2
    & +0.04511223
    & +0.06204098
    & +1.3753
      \cr
      3
    & +0.03022620
    & +0.04543279
    & +1.5031
      \cr
      4
    & +0.02271783
    & +0.03650569
    & +1.6069
      \cr
      5
    & +0.01819462
    & +0.03083187
    & +1.6946
      \cr
      6
    & +0.01517240
    & +0.02686634
    & +1.7707
      \cr
      7
    & +0.01301066
    & +0.02391848
    & +1.8384
      \cr
      8
    & +0.01138784
    & +0.02162994
    & +1.8994
      \cr
      9
    & +0.01012480
    & +0.01979505
    & +1.9551
      \cr
      10
    & +0.00911387
    & +0.01828677
    & +2.0065
      \cr
      12
    & +0.00759670
    & +0.01594444
    & +2.0989
      \cr
      14
    & +0.00651247
    & +0.01420059
    & +2.1805
      \cr
      16
    & +0.00569903
    & +0.01284540
    & +2.2540
      \cr
      18
    & +0.00506621
    & +0.01175807
    & +2.3209
      \cr
      20
    & +0.00455986
    & +0.01086381
    & +2.3825
      \cr
  \end{tabular}
  \caption{
    The table of length of edges.
    In each subdivision, 
    edges of hexagons of center of dihedral action attain
    minimums of length, 
    and 
    two edges of octahedron attain
    maximums of length.
  }
  \label{tab:mackay-length}
\end{table}
\par\newpage
\begin{figure}[H]
  \begin{center}
    \begin{tabular}{ccc}
      \includegraphics[bb=0 0 464 464, width=\figsize]{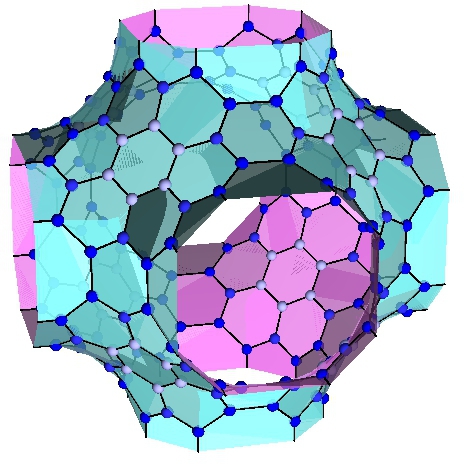}
      & 
        \includegraphics[bb=0 0 464 464, width=\figsize]{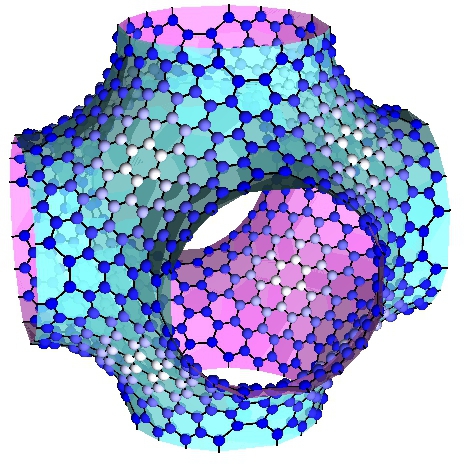}
      & 
        \includegraphics[bb=0 0 464 464, width=\figsize]{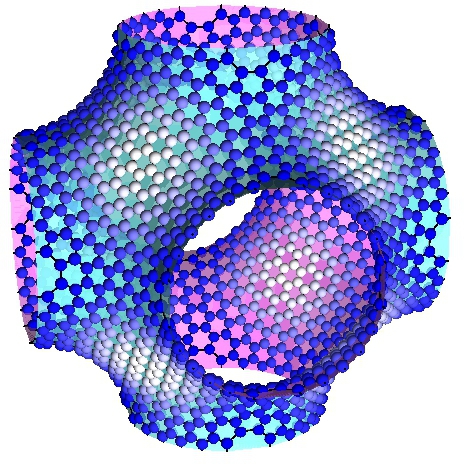}
      \\
      \includegraphics[bb=0 0 464 464, width=\figsize]{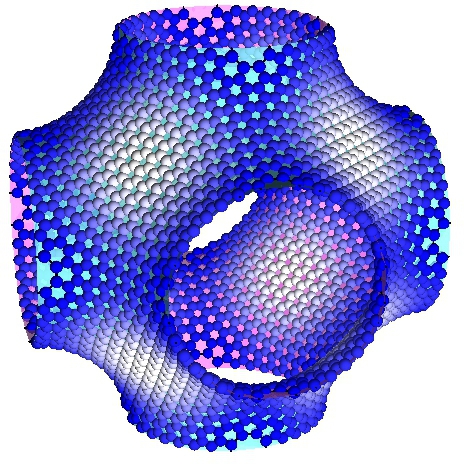}
      & 
        \includegraphics[bb=0 0 464 464, width=\figsize]{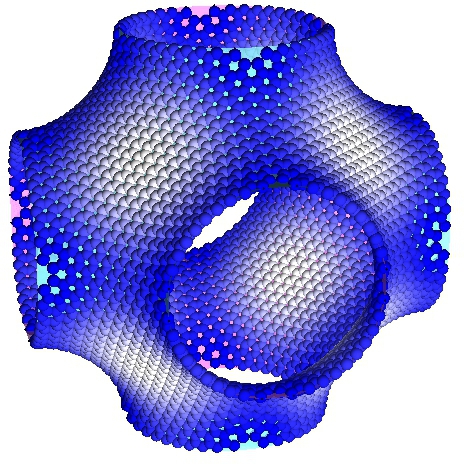}
      & 
        \includegraphics[bb=0 0 464 464, width=\figsize]{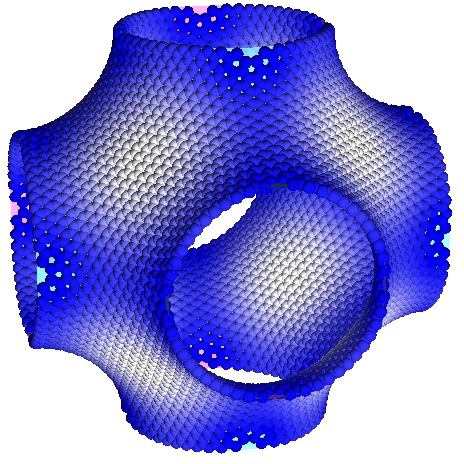}
      \\
      \includegraphics[bb=0 0 464 464, width=\figsize]{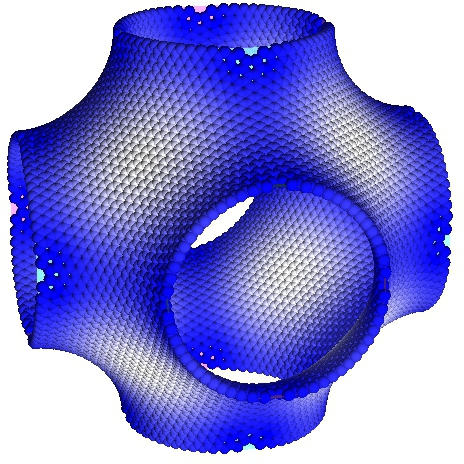}
      & 
        \includegraphics[bb=0 0 464 464, width=\figsize]{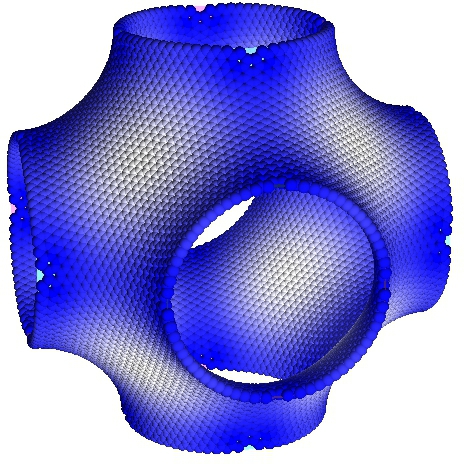}
      & 
        \includegraphics[bb=0 0 464 464, width=\figsize]{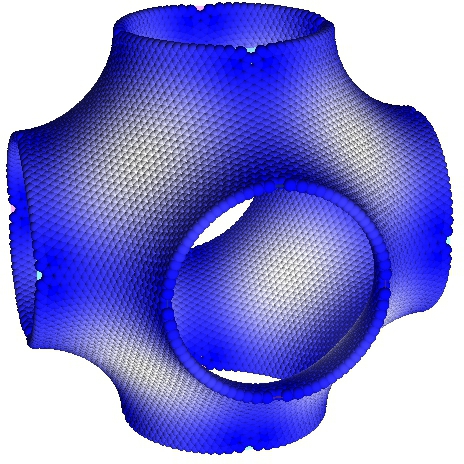}
      \\
    \end{tabular}
  \end{center}
  \caption{
    The Gauss curvature of $\GC_{(k, 0)}$ of Mackay crystals for $k = 1, \ldots, 9$.
    The Gauss curvature attain the smallest (negative, largest absolutely) values 
    in the respective pictures
    at the most red points, 
    while white points are those where the mean curvature is zero, 
    and colors are linearly interpolated between blue and white.
  }
  \label{Figure(K_for_GC-subdiv_of_Mackay)}
\end{figure}
\begin{figure}[H]
  \begin{center}
    \begin{tabular}{ccc}
      \includegraphics[bb=0 0 464 464, width=\figsize]{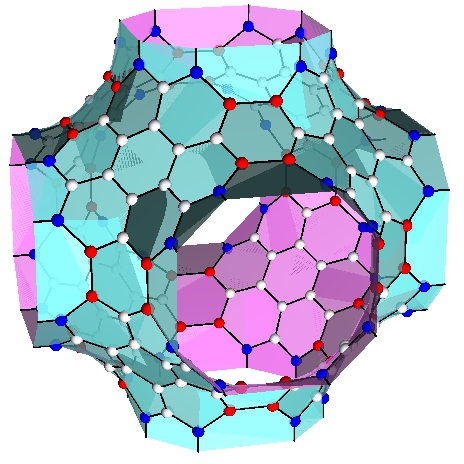}
      & 
        \includegraphics[bb=0 0 464 464, width=\figsize]{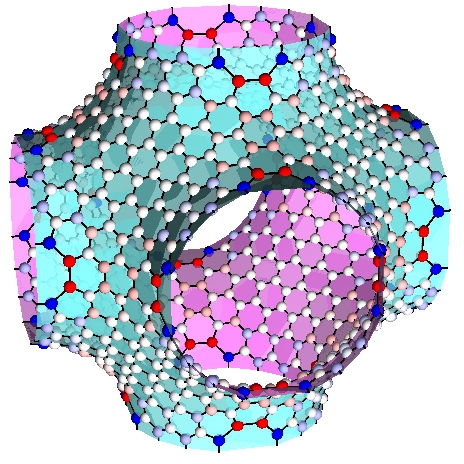}
      & 
        \includegraphics[bb=0 0 464 464, width=\figsize]{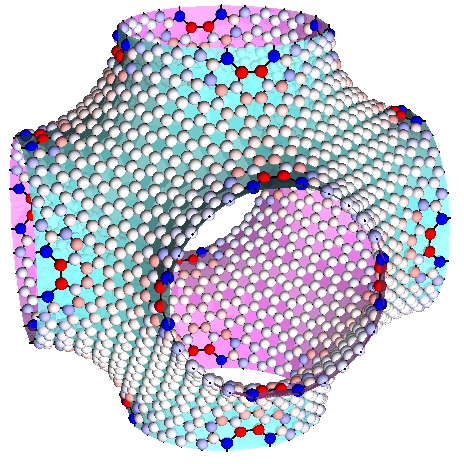}
      \\
      \includegraphics[bb=0 0 464 464, width=\figsize]{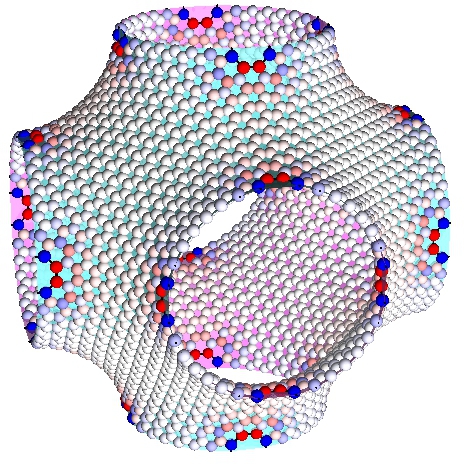}
      &
        \includegraphics[bb=0 0 464 464, width=\figsize]{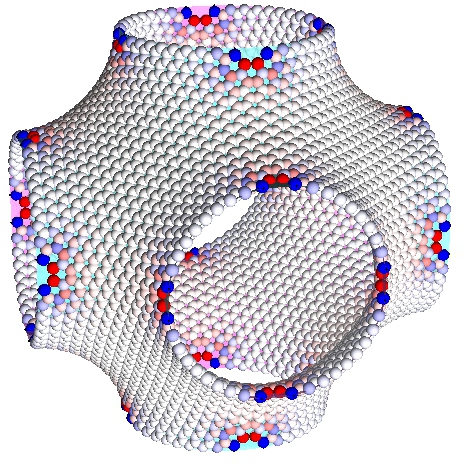}
      & 
        \includegraphics[bb=0 0 464 464, width=\figsize]{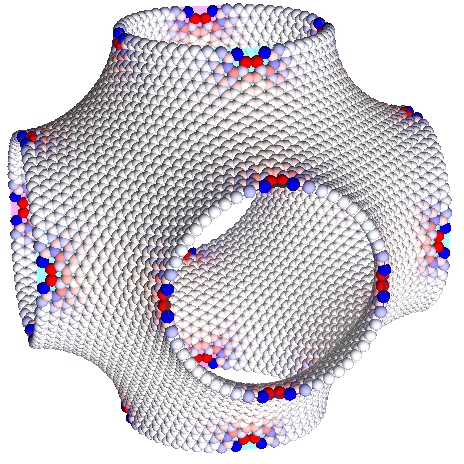}
      \\
      \includegraphics[bb=0 0 464 464, width=\figsize]{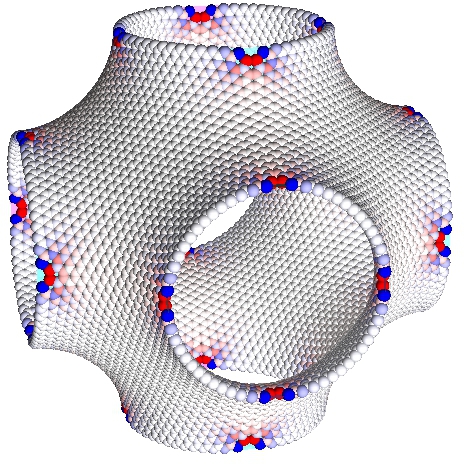}
      & 
        \includegraphics[bb=0 0 464 464, width=\figsize]{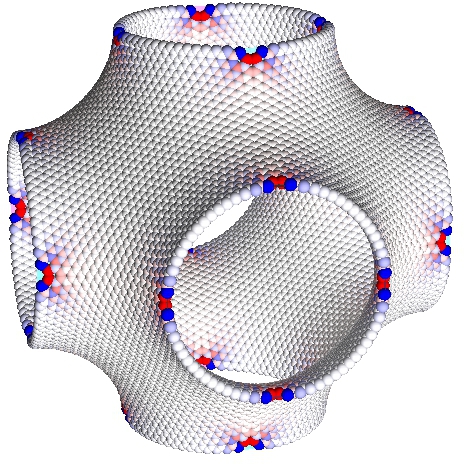}
      & 
        \includegraphics[bb=0 0 464 464, width=\figsize]{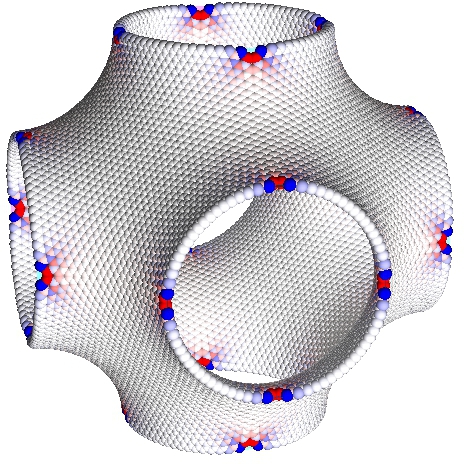}
      \\
    \end{tabular}
  \end{center}
  \caption{
    The mean curvature of $\GC_{(k, 0)}$ of Mackay crystals for $k = 1, \ldots, 9$.
    The mean curvature attain the smallest (negative, largest absolutely)/largest (positive) values 
    in the respective pictures
    at the most blue/red points, 
    while white points are those where the mean curvature is zero, 
    and colors are linearly interpolated between blue/white/red. 
  }
  \label{Figure(H_for_GC-subdiv_of_Mackay)}
\end{figure}
\subsubsection*{Acknowledgment}
Authors are partially supported by JSPS KAKENHI Grant Number
40211411, 15H02055, 24244004, 15K13432, and 15K17546.
Moreover, authors was partially supported by JST, CREST, 
``A mathematical challenge to a new phase of material sciences'' (2008-2013).
\subsubsection*{Dedication}
Authors dedicate this paper to Hisashi's late wife Yumiko Naito (September 8, 1963 -- October 2, 2015). 
While battling with breast cancer, she spent her days with a positive and enthusiastic attitude, supporting his research and caring for their son.
Hisashi's research up to now could not have existed without her support.
\par\newpage
\nocite{MR1743611}
\nocite{MR1748964}
\nocite{MR1783793}
\nocite{MR1500145}
\nocite{MR2902247}
\nocite{MR3014418}
\nocite{MR2299728}

\end{document}